%% file: CritQuiv-article.tex
\documentclass[a4paper]{article}

\usepackage{microtype}
\usepackage{amscd}
\usepackage{amsmath,amssymb,amsfonts, accents}
\usepackage{xstring}
\usepackage{xcolor}
\usepackage[mathscr,mathcal]{euscript}
\usepackage{xypic}
\xyoption{rotate}
\usepackage[cmtip,all,2cell]{xy}
\entrymodifiers={+!!<0pt,\fontdimen22\textfont2>}
\usepackage{url}
\usepackage{latexsym}
\usepackage{amsthm}
\usepackage[latin1]{inputenc}
\usepackage{multicol}
\usepackage{enumitem}
\usepackage[colorlinks=true, linkcolor={blue!50!black}, pdfhighlight=/O, ocgcolorlinks=true]{hyperref}
\usepackage{graphicx}
\usepackage{color}
\usepackage{mathtools}
\usepackage[colorinlistoftodos]{todonotes}
\usepackage[affil-it]{authblk}

\usepackage{adjustbox}

\usepackage{tikz-cd}

\usepackage{comment}

\usepackage[english,french]{babel}

\usepackage[a4paper, total={5.5in,8.5in}]{geometry}

\numberwithin{equation}{section}
\theoremstyle{plain}
\newtheorem{theorem}[equation]{Theorem}
\newtheorem*{theorem*}{Theorem}
\newtheorem{lemma}[equation]{Lemma}
\newtheorem{proposition}[equation]{Proposition}
\newtheorem{corollary}[equation]{Corollary}

\theoremstyle{definition}

\newtheorem{definition}[equation]{Definition}

\newtheorem{example}[equation]{Example}
\newtheorem{remark}[equation]{Remark}

\newtheorem{notation}[equation]{Notation}

\setlist[enumerate]{label=(\arabic*), leftmargin=*}
\setenumerate{label=(\arabic*), leftmargin=*}
\setlist[itemize]{label=$\vcenter{\hbox{\footnotesize$\bullet$}}$, leftmargin=*}
\setitemize{label=$\vcenter{\hbox{\footnotesize$\bullet$}}$, leftmargin=*}

\newcommand{\mb}[1]{\mathbf{#1}}
\newcommand{\mm}[1]{\mathrm{#1}}

\newcommand{\ul}[1]{\underline{#1}}

\newcommand{\cat}[1]{
\StrLen{#1}[\mystrlen]
\ifnum\mystrlen=1 \mathscr{#1}
\else \mathrm{#1}
\fi}
\newcommand{\scat}[1]{\mb{#1}}

\newcommand{\Hom}[0]{\mm{Hom}}

\newcommand{\op}[0]{\mm{op}}

\newcommand{\gr}[0]{\mm{gr}}
\newcommand{\cl}[0]{\mm{cl}}

\newcommand{\cC}{\mathcal{C}}

\newcommand{\bPerf}{\mathbf{Perf}}
\newcommand{\Perf}{\mathrm{Perf}}

\newcommand{\cG}{\mathcal G}
\newcommand{\cA}{\mathcal A}

\newcommand{\cB}{\mathcal B}
\newcommand{\bk}{\mathbf{k}}
\newcommand{\bT}{\mathbf{T}}

\newcommand{\bcrit}{\mathbf{crit}}
\newcommand{\bbT}{\mathbb{T}}
\newcommand{\bbL}{\mathbb{L}}

\DeclareMathOperator\Rep{Rep}
\DeclareMathOperator\GL{GL}
\DeclareMathOperator\PGL{PGL}
\DeclareMathOperator\End{End}
\DeclareMathOperator\Isom{Isom}
\DeclareMathOperator\Ho{Ho}
\DeclareMathOperator\BGL{BGL}
\DeclareMathOperator\Mat{Mat}

\newcommand{\pullbackcorner}[1][dl]{\save*!/#1-1pc/#1:(-1,1)@^{|-}\restore}

\title{Relative critical loci and quiver moduli\\\textit{\Large Lieux critiques relatifs et espaces de modules de carquois}}

\author[]{
	Tristan Bozec\thanks{IMAG, Univ Montpellier, CNRS, Montpellier, France \\
											\href{mailto:tristan.bozec@umontpellier.fr}{tristan.bozec@umontpellier.fr}}, 
	Damien Calaque\thanks{IMAG, Univ Montpellier, CNRS, Montpellier, France \\
											\href{mailto:damien.calaque@umontpellier.fr}{damien.calaque@umontpellier.fr}}, 
	Sarah Scherotzke\thanks{Mathematical Institute, University of Luxembourg, Luxembourg \\
											\href{mailto:sarah.scherotzke@uni.lu}{sarah.scherotzke@uni.lu}}}

\date{}

\begin{document}

\maketitle

{\selectlanguage{english}
\begin{abstract}
In this paper we identify the cotangent stack to the derived stack of representations of a quiver $Q$ with the derived moduli stack 
of modules over the Ginzburg dg-algebra associated with $Q$. More generally, we extend this result to finite type dg-categories, 
to a relative setting as well, and to deformations of these. It allows us to recover and generalize some results of Yeung, and 
leads us to the discovery of seemingly new lagrangian subvarieties in the Hilbert scheme of points in the plane. 
\end{abstract}}

{
\selectlanguage{french}
\begin{abstract}
Dans cet article nous identifions le champ cotangent du champ d\'eriv\'e des repr\'esentations d'un carquois $Q$ avec le champ  d\'eriv\'e des modules sur la dg-alg\`ebre de Ginzburg associ\'ee \`a $Q$. Nous \'etendons ce r\'esultat aux dg-cat\'egories de type fini, ainsi qu'\`a un cadre relatif, et \`a des d\'eformations de ce qui pr\'ec\`ede. Cela nous permet de retrouver et g\'en\'eraliser des r\'esultats de Yeung, et nous m\`ene \`a la d\'ecouverte de nouvelles sous-vari\'et\'es lagrangiennes du sch\'ema de Hilbert de points dans le plan.
\end{abstract}}

\selectlanguage{english}

\setcounter{tocdepth}{1}
\tableofcontents


\section{Introduction}

\input{CritQuiv-Section1}

\section{Motivation: new lagrangian subvarieties}\label{sec:motiv}

\input{CritQuiv-Section2}

\section{Derived symplectic geometry}\label{sec:recap}

\input{CritQuiv-Section3}

\section{Derived and non-derived examples}\label{sec:examples}

\input{CritQuiv-Section4}

\section{Calabi--Yau structures}\label{sec:CY}

\input{CritQuiv-Section5}

\section{Comparison}\label{sec:comparison}

\input{CritQuiv-Section6}

\appendix\section{A Lemma on commutators}

\input{CritQuiv-Appendix}


\end{document}

%% file: CritQuiv-Section1.tex
This paper is concerned with Calabi--Yau structures on smooth dg-categories and functors thereof, as well as with symplectic 
and lagrangian structures on derived Artin stacks and morphisms between these. Its main goal can be summarized 
as giving a precise meaning (and a proof) to the claim that deformations of Calabi--Yau completions of dg-categories 
(resp.~relative Calabi--Yau completions) are a non-commutative analog of deformations of cotangent stacks (resp.~conormal stacks). 

\medskip

A guiding idea of noncommutative (algebraic) geometry is that a noncommutative analog of a standard 
geometric notion $T$ is some algebraic notion $T^\mathrm{nc}$ that induces functorially a structure of type $T$ on schemes 
of representations (see e.g.~\cite{Gi0} and references therein). 
It has been advocated by several authors that the noncommutative analog of the cotangent bundle should be obtained as a 
free construction, and that, in the case of quivers, if one applies a noncommutative version of hamiltonian reduction, 
then one recovers the preprojective algebra (see e.g.~\cite{CBEG}). We would like to amend this approach in several ways: 
\begin{itemize}
\item First of all, we would like to replace representation schemes by moduli stacks of representations, so that the 
preprojective algebra itself becomes the correct noncommutative cotangent. 
\item Then, in order for most functors to remain well-behaved, we work within a homotopy invariant framework, 
and so we will in fact consider the \textit{derived} moduli of representations. The preprojective algebra will then be 
replaced by its differential graded variant, introduced by Ginzburg~\cite{Gi}. 
\item Finally, we will consider the whole derived moduli of perfect dg-modules, also called \textit{moduli of objects}, 
after To\"en--Vaqui\'e~\cite{ToVa}. This also has the advantage that it extends to dg-categories. 
\end{itemize}
In 2011, Keller~\cite{Keller} introduced a broad generalization of Ginzburg dg-algebras: the so-called 
Calabi--Yau completions of smooth dg-categories (which, surprisingly enough, are indeed given as a free construction), 
and their deformations. Recently, Yeung~\cite{Ye} advocated that Calabi--Yau completions should be viewed as noncommutative 
cotangent bundles. 
It is actually known, after To\"en~\cite{ToEMS,ToCY} and Brav--Dyckerhoff~\cite{BD2}, that the moduli of objects functor 
$\mathbf{Perf}$ sends finite type $n$-Calabi--Yau dg-categories to $(2-n)$-shifted symplectic stacks in the sense of 
Pantev--To\"en--Vaqui\'e--Vezzosi~\cite{PTVV}. In other words, Calabi--Yau structures on dg-categories are the non-commutative (shifted) symplectic structures. 
In this paper, we prove the following.

\begin{theorem*}[Theorem~\ref{theo: compare}]
For a finite type dg-category $\cat{A}$, there is an equivalence of exact $(2-n)$-shifted symplectic stacks between the shifted 
cotangent stack $\mathbf{T}^*[2-n]\mathbf{Perf}_{\cat{A}}$ and the perfect moduli $\mathbf{Perf}_{\mathcal{G}_n(\cat{A})}$ 
of the $n$-Calabi--Yau completion $\mathcal{G}_n(\cat{A})$ of $\cat{A}$. 
\end{theorem*}
We also prove a deformed version of the above, identifying the moduli of objects of deformed Calabi--Yau completions 
with twisted cotangent stacks, as shifted symplectic stacks. 

\medskip

Our major geometric motivation lies in the construction of lagrangian (rather than just symplectic) structures, 
and their connections with critical loci or Calabi--Yau geometry. To be more specific, inspired by classical symplectic 
geometry~\cite{RS}, we are interested in deformations of conormal stacks that appear as \emph{relative critical loci} 
(see~\cite{CalSurv,AC}). In certain cases, these relative critical loci can be obtained as moduli of representations of 
relative versions of Ginzburg dg-algebras. 
If we use the formalism of derived symplectic geometry, introduced by Pantev--To\"en--Vaqui\'e--Vezzosi~\cite{PTVV}, we will 
show that our constructions yield very down-to-earth examples, such as lagrangian subvarieties of some quiver 
moduli spaces. Moduli of quiver representations form indeed a large and fruitful class of applications in this area, and 
will serve as an ongoing example in the present study. More generally, moduli of objects of finite type dg-categories, as defined 
by To\"en--Vaqui\'e~\cite{ToVa}, encompass quiver representations and provide a lot more examples to play with. In fact, as we have 
already seen, it is essential in the statement itself of the results that achieve our main goal. 

\medskip

Elaborating on an idea of To\"en~\cite{ToEMS}, Brav--Dyckerhoff~\cite{BD1} have introduced a very useful notion of relative 
(left) Calabi--Yau structure on a dg-functor, and have shown~\cite{BD2} (see also~\cite{ToCY}) that these relative Calabi--Yau structures are the 
correct noncommutative analog of lagrangian structures on morphisms between derived stacks: the moduli of objects functor 
$\mathbf{Perf}$ sends functors (between finite type dg-categories) equipped with a relative Calabi--Yau structure to morphisms 
(between derived Artin stacks) equipped with a lagrangian structure. Soon after, Yeung~\cite{Ye0} introduced a relative 
analog of Keller's Calabi--Yau completions along functors between finite cellular dg-categories, and showed that they indeed 
carry a relative Calabi--Yau structure in the sense Brav--Dyckerhoff. 

\medskip

In this paper, with every dg-functor $f:\cat{A}\to\cat{B}$ between smooth dg-categories we associate a cospan 
\[
\mathcal G_n(\cat{A})\longrightarrow\mathcal G_n(f)\longleftarrow\mathcal G_n(\cat{B})
\]
and we prove the following.

\begin{theorem*}[Theorem \ref{thm:Gn-functoriality} \& Theorem \ref{theo: comparerel}]
\begin{enumerate}
\item The above cospan carries an exact $n$-Calabi--Yau structure. 
\item For finite type dg-categories, the moduli of objects functor sends this exact $n$-Calabi--Yau cospan to the 
exact lagrangian correspondence 
\[
\mathbf{T}^*[2-n]\mathbf{Perf}_{\cat{A}}\longleftarrow \varphi^*\mathbf{T}^*[2-n]\mathbf{Perf}_{\cat{A}}
\longrightarrow \mathbf{T}^*[2-n]\mathbf{Perf}_{\cat{B}}\,,
\]
where $\varphi=f^*:\mathbf{Perf}_{\cat{B}}\to\mathbf{Perf}_{\cat{A}}$. 
\end{enumerate}
\end{theorem*}
As a consequence of the first point, we recover Yeung's deformed relative Calabi--Yau completion (without the finite 
cellular assumption) thanks to a simple use of the composition of Calabi--Yau cospans. 
As a consequence of the second point, we obtain that the moduli of objects of certain deformed relative Calabi--Yau completions 
can be identified with relative critical loci. We understand very concrete examples of these relative critical loci in the case 
of quiver representations and exhibit in particular some new (to our knowledge) lagrangian subvarieties of the Hilbert schemes 
of points on the plane. 

\medskip

We now summarize the new contributions of this paper: 
\begin{itemize}
\item We extend several results of Yeung on the existence of (relative) Calabi--Yau structures from finite cellular 
dg-categories to smooth dg-categories, simplifying the proofs. 
\item The simplification is made possible thanks to a new exact Calabi--Yau cospan associated with every dg-functor 
$f:\cat{A}\to\cat{B}$, that we also use to identify the Calabi--Yau completion with a noncommutative hamiltonian reduction. 
\item We identify the moduli stack of objects of a (relative) deformed Calabi--Yau completion with the (relative) critical locus. 
\item In the case of the inclusion of the one loop quiver in its tripled version, we discover new lagrangian subvarieties in 
the Hilbert scheme of $\mathbb{C}^2$. 
\end{itemize}

\subsection*{Outline of the paper}

\begin{tikzpicture}
\node[draw,rectangle,rounded corners=3pt, text justified] (A) at (-5,0) 
{Section~\ref{sec:motiv}} ;
\node[draw,rectangle,rounded corners=3pt, text justified] (B) at (0,0) 
{Section~\ref{sec:examples}} ;
\node[draw,rectangle,rounded corners=3pt, text justified] (C) at (0,2) 
{Section~\ref{sec:recap}} ;
\node[draw,rectangle,rounded corners=3pt, text justified] (D) at (6,0) 
{Section~\ref{sec:CY}} ;
\node[draw,rectangle,rounded corners=3pt, text justified] (E) at (3,0) 
{Section~\ref{sec:comparison}} ;
\draw [>=latex,->] (A) to[bend left=5] node[midway,above]{motivation} (B) ;
\draw [>=latex,->] (B) to[bend left=5] node[midway,below]{example} (A) ;
\draw [>=latex,->] (C) to node[midway,left]{examples} (B) ;
\draw [>=latex,->] (B) to  (E) ;
\draw [>=latex,->] (D) to (E) ;
\node (G) at (3,.6) {comparison} ;
\end{tikzpicture}

\medskip

In section~\ref{sec:motiv} we start with what motivated us in the first place, which is the study of some subvarieties defined 
as constrained critical loci in the cotangent space of $\mathfrak{gl}_n(\mathbb C)$. The general idea goes back to the work~\cite{RS}, 
which gives criteria, in the differentiable setting, for these submanifolds to be lagrangian. We apply their idea to a 
particular algebro-geometric example associated with the inclusion of the one loop quiver in its tripled version, we perform 
a thorough description of the irreducible components of the subvarieties we define, and we prove that they are all 
indeed lagrangian. This section serves as a motivation for the rest of the paper, but it can easily been skipped as all 
other sections are independant from it. On the contrary, readers who are only interested in the new lagrangian in the Hilbert 
scheme may skip the rest of the paper. 

\medskip

In section~\ref{sec:recap} we give a recollection on shifted symplectic and lagrangian structures, according to~\cite{PTVV}. 
We introduce in particular mixed realization functors that go from graded mixed complexes to mixed complexes. It is an essential 
tool that allows to associate closed forms with negative cyclic classes, on commutative differential graded algebras. 
It is used later on to construct shifted symplectic (and lagrangian) structures on moduli of objects of Calabi--Yau categories 
(and functors). We insist on exact structures, as they sometimes make the proof simpler. 

\medskip

We give numerous examples of shifted symplectic and lagrangian structures in section~\ref{sec:examples}, starting with shifted 
cotangent and conormal stacks (after~\cite{CalCot}). 
This leads to a functor $\mathbf{T}^*[n]$ going from the homotopy category of derived stacks and morphisms thereof 
to the homotopy category of exact $n$-shifted symplectic stacks and lagrangian correspondences. As an application, we get 
that the cotangent stack to the representation stack of a quiver $Q$ is the derived representation stack of Ginzburg $2$-Calabi--Yau 
algebra with no potential. Using composition of lagrangian correspondences one can then define critical loci (or twisted 
cotangent stacks) and relative (or constrained) critical loci, which are deformations of cotangent and conormal stacks. 
We then study examples associated with inclusions of quivers, and again identify them with derived representation stacks 
of more general Ginzburg dg-algebras~\cite{Gi}. 
We end the section with a result that provides a systematic way of building genuine lagrangian subvarieties from the derived 
setting, and retrieve our motivating examples from section~\ref{sec:motiv}. 

\medskip

The reader who is only interested in Calabi--Yau categories may directly go to section~\ref{sec:CY}: 
we investigate the dg-categories at stake in our geometric considerations, and study the counterpart of the 
lagrangian correpondences of the previous section using the so-called Calabi--Yau cospans introduced in~\cite{BD1}. 
We introduce exact Calabi--Yau cospans, as they help to simplify some proofs (in this section, but above all in the next one), 
and we provide examples of these. As already mentioned, we extend some results of Yeung, and prove that the exact Calabi--Yau 
structure on a Calabi--Yau completion can be obtained \textit{via} a composition of Calabi--Yau cospans, that can be understood 
as a noncommutative hamiltonian reduction procedure. 

\medskip

In the last section~\ref{sec:comparison}, we consider the moduli of objects of (deformed) Calabi--Yau completions. We identify 
these with cotangent stacks or (relative) critical loci. We show that the symplectic/lagrangian structures induced by the 
Calabi--Yau structures on these moduli of objects (thanks to~\cite{BD2,ToCY}) coincide with the classical ones on shifted cotangent 
stacks and relative critical loci. 

\subsection*{Acknowledgements}
We are grateful to Mathieu Anel for many conversations about relative critical loci, to Benjamin Hennion who suggested a strategy 
to get genuine lagrangian subvarieties from derived lagrangians, to Bernhard Keller for several discussions about Calabi--Yau 
and dg-categories, and to Wai-Kit Yeung for his comments on a previous version.  
The first and second author have received funding from the European Research Council (ERC) under the
European Union's Horizon 2020 research and innovation programme (Grant Agreement No.\ 768679).

%% file: CritQuiv-Section2.tex
Lagrangian subvarieties are particularly important in moduli problems as they correspond to singular supports (or 
characteristic varieties) of sheaves on given moduli spaces.  
Our idea here is to construct new lagrangian subvarieties rather than working on the support of given sheaves. We work in the 
fruitful framework of quiver moduli which yields a fundamental class of examples of $3$-Calabi--Yau algebraic structures 
(see~\cite{Gi}). Lagrangian subvarieties have already proved being of great interest in this particular context as they 
were used by Lusztig~\cite{Lusztig} and Kashiwara--Saito~\cite{KaSa} to define canonical bases on quantum groups.

We have two rather distinct inspirations here, one coming from traditional symplectic geometry, more precisely from the 
work~\cite{RS} of Robbin and Salamon. They define lagrangian manifolds attached to a variational family, \textit{i.e.}\ the 
datum of a function $\phi:X\to\mathbb R$ and a submersion $\pi:X\to B$, as the critical locus of $\phi$ in $T^*B$ constrained 
by $\pi$. Our second inspiration comes from the recent definition of relative Calabi--Yau structures by Brav and 
Dyckerhoff~\cite{BD1} and their study in the context of quivers by Yeung~\cite{Ye0}. More precisely he carries out therein 
computations with respect to a pair of nested quivers $(Q_1\subseteq Q_2)$ (see \S3.4 in \textit{loc.\ cit.}). It turns out 
that the combination of these works already gives a nontrivial example of lagrangian subvarieties in the classical situation 
of the one loop quiver in its tripled version
\[
S_1=\!\!\!\raisebox{-0.46\height}{
\begin{adjustbox}{max totalsize=4cm}
\begin{tikzpicture}
\node(1)at(0,0){$\bullet$};
\draw[-latex](1)edge[out=-40,in=40,looseness=12,right]node{$x$}(1);
\end{tikzpicture}
\end{adjustbox}}
~~~~\subset~~~~
S_3=\!\!\!\!\!\!\!\!\!\!\!\!\!\raisebox{-0.46\height}{
\begin{adjustbox}{max totalsize=4cm}
\begin{tikzpicture}
\node(1)at(0,0){$\bullet$};
\draw[-latex](1)edge[out=-40,in=40,looseness=12,right]node{$x$}(1);
\draw[-latex](1)edge[out=80,in=160,looseness=12,above ]node{$y$}(1);
\draw[-latex](1)edge[out=200,in=280,looseness=12,below]node{$z$}(1);
\end{tikzpicture}
\end{adjustbox}}
\]
the latter being endowed with the potential $W=[x,y]z$ defined as the product of the moment map of the doubled quiver 
with the third loop. The aim of this section is to study the geometry of the algebraic varieties we obtain following~\cite{RS} 
in this particular case as a motivation to a deeper work in the framework of derived symplectic geometry. 

\subsection{A lagrangian subvariety in the commuting variety}\label{lagcomm}

We will work with the linear submersion $\pi:{\Rep}_{\mathbb C}(S_3,n)\rightarrow {\Rep}_{\mathbb C}(S_1,n)$ of representations of 
dimension $n$ over $\mathbb C$ given by the projection $(x,y,z)\mapsto x$, endowed with the map 
$\phi=\mathrm{tr}(W):{\Rep}_{\mathbb C}(S_3,n)\rightarrow\mathbb C$, $(x,y,z)\mapsto \mathrm{tr}([x,y]z)$.
The variety defined in~\cite{RS} (relative critical locus of $\phi$ constrained by $\pi$, or Lagrange multiplier space) is given by 
\[
\Lambda_n=\left\{(x,x^*)\in T^*{\Rep}_{\mathbb C}(S_1,n)\middle|
\begin{aligned}&\exists (y,z)\text{ s.t.\ }(x,y,z)\in {\Rep}_{\mathbb C}(S_3,n)\\
&\forall \gamma=(\gamma_1,\gamma_2,\gamma_3)\in {\Rep}_{\mathbb C}(S_3,n):\\
&\mathrm{tr}(\gamma_1x^*)=d\phi_{x,y,z}(\gamma)\end{aligned}\right\}.
\]
Since $d\phi_{x,y,z}(\gamma)=\mathrm{tr}(\gamma_1[y,z]+\gamma_2[z,x]+\gamma_3[x,y])$, we get
\[
\Lambda_n=\big\{(x,[y,z])\mid[x,y]=[x,z]=0\big\}\subset T^*{\Rep}_{\mathbb C}(S_1,n)\,.
\]
Our aim is to prove that $\Lambda_n$ is purely lagrangian, \textit{i.e.}\ that all its irreducible components are isotropic 
of maximal dimension $n^2$.

\begin{remark}~

\begin{enumerate}[label=(\roman*)]
\item
Via the trace map, one can identify $T^*{\Rep}_{\mathbb C}(S_1,n)$ with the representation space ${\Rep}_{\mathbb C}(S_2,n)$ 
of the doubled quiver. Then, the moment map associated with the action of $\GL_n(\mathbb C)$ on ${\Rep}_{\mathbb C}(S_1,n)$ 
by conjugation reads
\[
\mu_n:~{\Rep}_{\mathbb C}(S_2,n)\to\mathfrak{gl}_n(\mathbb C)~,~(x,x^*)\mapsto[x,x^*]\,.
\]						
Hence $\Lambda_n$ is included in the \emph{commuting variety} $\mu^{-1}_n(0)$. 
\item Denote by $\mathcal N \subset \mathfrak{gl}_n(\mathbb C)$ the nilpotent cone. We want to emphasize that $\Lambda_n$ 
is distinct from $T^*_\mathcal N\mathfrak{gl}_n(\mathbb C)\subset\mu^{-1}_n(0)$, which is lagrangian in 
$T^* \mathfrak{gl}_n(\mathbb C)$. Its irreducible components are given by the closures of 
$T^*_{\mathcal O_\lambda}\mathfrak{gl}_n(\mathbb C)$, where $\mathcal O_\lambda$ denotes the nilpotent orbit associated 
with a partition $\lambda\vdash n$.
\end{enumerate}\end{remark}

\begin{proposition}
The subvariety $\Lambda_n$ is isotropic in $T^*{\Rep}_{\mathbb C}(S_1,n)$.
\end{proposition}
\begin{proof}
Write $X={\Rep}_{\mathbb C}(S_3,n)$ and $B={\Rep}_{\mathbb C}(S_1,n)$, and form the following composition of lagrangian 
correspondences
\[
  \xymatrix{& & L \ar[rd] \ar[dl] && \\
&X \ar[rd]^-{d\phi}\ar[ld] && \pi^* T^* B  \ar[ld]\ar[rd]& \\
~~\mathrm{pt}~~&&T^* X &&T^*B.}
\]
Then one can check that $\Lambda_n$ is in the image of $L$ in $T^*{\Rep}_{\mathbb C}(S_1,n)$, and is therefore isotropic.
\end{proof}
\begin{remark}
One sees that $L$ itself is neither necessarily a subspace, nor necessarily smooth. Both issues will be irrelevant when we will be working in the derived setting. 
\end{remark}

Using Lie notations, one can denote by $\mathfrak g_x$ the centraliser of $x\in\mathfrak g=\mathfrak{gl}_n(\mathbb C)$ 
and hence write
\[
\Lambda_n=\{(x,t)\in\mathfrak g^2\mid t\in[\mathfrak g_x,\mathfrak g_x]\}\subset T^*\mathfrak g\]
as it turns out that the subvariety $\{[y,z]\mid(y,z)\in\mathfrak g_x^2 \}\subset\mathfrak g_x$ is actually a vector space, hence coincides with $[\mathfrak g_x,\mathfrak g_x]$ defined as a Lie algebra. Indeed we have the following proposition:

\begin{proposition}\label{propAbel} Decompose $x=\oplus_k(\alpha_k+n_k)$ with respect to generalized eigenspaces, where the $\alpha_k's$ are pairwise distinct complex numbers and $n_k$ is nilpotent of type $\lambda^{(k)}=(\lambda^{(k)}_1\ge\lambda^{(k)}_2\ge\dots)$, a partition of $\dim \ker(x-\alpha_k)^n$. Then $\{[y,z]\mid(y,z)\in\mathfrak g_x^2 \}$ is a vector space of codimension $\sum_k\lambda^{(k)}_1$ in $\mathfrak g_x$.\end{proposition}

\begin{proof}
It is enough to prove this result when $x$ is nilpotent of type $\lambda$. 
Denote by $\lambda'$ the conjugate partition, so that in a basis adapted to iterated kernels of $x$ one has 
\[
x=  \left(\!\!\!\!\!\!
     \raisebox{0.5\depth}{
       \xymatrixcolsep{1ex}
       \xymatrixrowsep{1ex}
       \xymatrix{
        0 \ar@{.}[ddd]\ar @{.}[dddrrr] &          I_{1,2}\ar @{.}[ddrr] &0\ar @{.}[dr] \ar @{.}[r]&      0 \ar@{.}[d]\\
      &&&0\\
         &&& I_{s-1,s}\\
         0 \ar@{.}[rrr]  & &  & 0
       }
     }
   \right)
\]
where $I_{k,k+1}$ is the matrix of size $\lambda'_k\times\lambda'_{k+1}$ with coefficients $\delta_{i,j}$, and 
$s=l(\lambda')=\lambda_1$. Then a quick computation shows that $\mathfrak g_x$ consists in matrices $T=(T_{i,j})_{i\le j}$ 
triangular by blocks of size $\lambda'_i\times\lambda'_j$ satisfying
\[
T_{i,j-1}I_{j-1,j}-I_{i,i+1}T_{i+1,j}=0
\]
for every $i\le j-1$. From there, thanks to Lemma~\ref{vsdim}, we see that the variety of commutators of such matrices 
consists in matrices of the same type with the extra conditions $\mathrm{tr}(T_{i,i})=0$ for every $1\le i\le s$ 
(\textit{i.e.} diagonal blocks are commutators themselves). Indeed, we get a vector space of the correct dimension.
\end{proof}

Consider a partition $\mu\vdash n$ and denote by $\Lambda_\mu$ the subvariety of $\Lambda_n$ consisting of pairs such that $x$ is diagonalizable of type $\mu$. We write $l(\mu)$ for $\mu'_1$, the number of parts of $\mu$.

\begin{theorem} The irreducible components of $\Lambda_n$ are the closures of the $\Lambda_\mu$'s, which are all of dimension $n^2$.\end{theorem}

\begin{proof} Fix $x\in \Lambda_\mu$ and denote by $L$ the Levi group associated to the decomposition $E_1\oplus\dots\oplus E_r$ of $\mathbb C^n$ in eigenspaces. We have\[
\Lambda_\mu=\GL_n\times_L\{(\oplus_k z_k \mathrm{id}_{E_k},y)\mid y\in [\mathfrak l,\mathfrak l]\}\]
where the $z_k$'s are pairwise distinct given complex numbers and $\mathfrak l=Lie(L)$. 
The map $g.(\oplus_k z_k,y)\mapsto (g.E_k)_k$ turns $\Lambda_\mu$ into an open subvariety of a vector bundle $V$ over an open subvariety of $\prod_k \mathrm{Grass}(\mu_k,n)$. Hence it is smooth of dimension\[
\dim V+\sum_k\dim \mathrm{Grass}(\mu_k,n)=l(\mu)+\dim [\mathfrak l,\mathfrak l]+\sum_k\mu_k(n-\mu_k)=n^2.\]
Indeed, since $[\mathfrak{gl}_n,\mathfrak{gl}_n]=\mathrm{tr}^{-1}(0)$, we have $\dim  [\mathfrak l,\mathfrak l]=\sum_k(\mu_k^2-1)=\sum_k\mu_k^2-l(\mu)$.

It remains to show that any $(x,y)$ for which $x$ is not diagonalizable lies in the closure of ${\Lambda_\mu}$ for some $\mu$. Once again it is enough to show it for $x$ nilpotent, say of type $\lambda$. We write\[
x=  \left(\!\!\!\!\!\!
     \raisebox{0.5\depth}{
       \xymatrixcolsep{1ex}
       \xymatrixrowsep{1ex}
       \xymatrix{
        0 \ar@{.}[ddd]\ar @{.}[dddrrr] &          I_{1,2}\ar @{.}[ddrr] &0\ar @{.}[dr] \ar @{.}[r]&      0 \ar@{.}[d]\\
      &&&0\\
         &&& I_{s-1,s}\\
         0 \ar@{.}[rrr]  & &  & 0
       }
     }
   \right)\]
as in the proof of Proposition~\ref{propAbel}, and approximate it by\[
x(\varepsilon)=  \left(\!\!\!\!\!\!
     \raisebox{0.5\depth}{
       \xymatrixcolsep{1ex}
       \xymatrixrowsep{1ex}
       \xymatrix{
        0  &          I_{1,2}\ar @{.}[ddrr] &0\ar @{.}[dr] \ar @{.}[r]&      0 \ar@{.}[d]\\
   0\ar@{.}[dd]\ar @{.}[ddrr]   &\varepsilon\ar @{.}[ddrr]&&0\\
         &&& I_{s-1,s}\\
         0 \ar@{.}[rr]  & &0 &(s-1)\varepsilon 
       }
     }
   \right).\]
We have $x(\varepsilon)\exp(x/\varepsilon)=\exp(x/\varepsilon)(x(\varepsilon)-x)$, hence $x(\varepsilon)$ is diagonalisable. As in the proof of~\ref{propAbel}, one can prove that the variety of commutators of elements in $\mathfrak g_{x(\varepsilon)}$ consists in matrices $T(\varepsilon)=(T_{i,j}(\varepsilon))_{i\le j}$ triangular by blocks of size $\lambda'_i\times\lambda'_j$ satisfying\begin{equation}\label{approx}
   T(\varepsilon)_{i,j-1}I_{j-1,j}-I_{i,i+1}T(\varepsilon)_{i+1,j}=\varepsilon(i-j) T(\varepsilon)_{i,j}\end{equation}
   for every $i\le j-1$ and $\mathrm{tr}(T(\varepsilon)_{i,i})=0$ for every $1\le i\le s$.
   Now take $T(\varepsilon)=\sum_kT^{(k)}\varepsilon^k$ such that $T^{(0)}=y$ and $\deg T(\varepsilon)_{i,j}=s-(j-i)$, so that~\eqref{approx} becomes\[
   T^{(k)}_{i,j-1}I_{j-1,j}-I_{i,i+1}T^{(k)}_{i+1,j}=(i-j) T^{(k-1)}_{i,j}
   \] for every $k=0...s$, where $T^{(-1)}=0$.
   Picking \[T^{\langle k\rangle}=(T_{i,j}^{(k-(j-i))})_{i\le j}\in[\mathfrak g_{x(1)},\mathfrak g_{x(1)}]\] concludes the proof as $x(\varepsilon)$ is diagonalisable and $(x(\varepsilon),y(\varepsilon):=T(\varepsilon))\to(x,y)$ as $\varepsilon\to0$. 
\end{proof}

As a consequence, we get that $\Lambda_n$ is a lagrangian subvariety of $T^*\mathfrak{gl}_n(\mathbb C)$ as expected.

\subsection{New lagrangian subvarieties of the Hilbert scheme of $\mathbb{C}^2$}

\subsubsection{Recollection on Nakajima quiver varieties}

Recall that the Hilbert scheme $(\mathbb C^2)^{[n]}$ of $n$ points on $\mathbb C^2$ can be seen as a Nakajima quiver variety associated to the Jordan quiver $S_1$. Let us be a bit more precise. First introduce the \emph{framed} version $S_g^+$ of the quiver $S_g$ with one vertex and $g$ loops: 

\vspace{-1cm} 

\[
S_g^+=\!\!\!\raisebox{-0.48\height}{
\begin{adjustbox}{max totalsize=10cm}
\begin{tikzpicture}
\node(2)at(-1,0){$\bullet$};
\node(1)at(0,0){$\bullet$};
\draw[-latex](2)edge(1);
\draw[-latex](1)edge[out=-30,in=30,looseness=10,right]node{$\dots$}(1);
\draw[-latex](1)edge[out=-50,in=50,looseness=17,right]node{$g.$}(1);
\end{tikzpicture}
\end{adjustbox}}
\] 
\vspace{-1cm}

\noindent We will always consider $1$-dimensional vector spaces on the framing vertex, hence will still abusively consider $n\in\mathbb N$-dimensional representations of $S_g^+$. Consider the doubled quiver $\overline{S_g^+}$ where each arrow is replaced by a pair of opposite ones. Write

\vspace{-1cm}

\[
{\Rep}_\mathbb C\Big(\overline{S_1^+},n\Big)=\Bigg\{\raisebox{-0.48\height}{
\begin{adjustbox}{max totalsize=10cm}
\begin{tikzpicture}
\node(2)at(-1,0){$\mathbb C$};
\node(1)at(0.5,0){$\mathbb C^n$};
\draw[-latex,transform canvas={yshift=-0.08cm}](2)edge[below]node{$v$}(1);
\draw[-latex,transform canvas={yshift=0.08cm}](1)edge[above]node{$v^*$}(2);
\draw[-latex](1)edge[out=30,in=-30,looseness=10,left]node{$x$}(1);
\draw[-latex](1)edge[out=-50,in=50,looseness=14,right]node{$x^*$}(1);
\end{tikzpicture}
\end{adjustbox}}\Bigg\}
\]
\vspace{-1cm}

\noindent where we identify $v$ with its image.

\begin{remark}The action of $\GL_n(\mathbb C)$ on ${\Rep}_\mathbb C(S_1^+,n)$ given by $g.(x,v)=(gxg^{-1},gv)$ induces a moment map given by\[
\mu_n^+:~{\Rep}_{\mathbb C}\Big(\overline{S_1^+},n\Big)\to\mathfrak{gl}_n(\mathbb C)~,~(x,x^*,v,v^*)\mapsto[x,x^*]+vv^*.\]
\end{remark}

 Recall:

\begin{definition}
We say that $(x,x^*,v,v^*)\in {\Rep}_\mathbb C\Big(\overline{S_1^+},n\Big)$ is \emph{stable} if $\mathbb C\langle x,x^*\rangle .v=\mathbb C^n$, and denote by ${\Rep}_\mathbb C\Big(\overline{S_1^+},n\Big)^\text{\!st}$ the open stable locus.
\end{definition}

The following is proved in~\cite{Nakajima}.

\begin{theorem}\label{nakrecap}
We have the following:
\begin{enumerate}[label=(\roman*)]
\item the action of $\GL_n(\mathbb C)$ on ${\Rep}_\mathbb C\Big(\overline{S_1^+},n\Big)^\textup{\!st}$ given by\[
g.(x,x^*,v,v^*)=(gxg^{-1},gx^*g^{-1},gv,v^* g^{-1})\] is free;
\item stable representations $(x,x^*,v,v^*)$ of $(\mu_n^+)^{-1}(0)$ satisfy $v^*=0$;
\item the Hilbert scheme of points on $\mathbb C^2$ is given by the following \emph{symplectic reduction}\[
(\mathbb C^2)^{[n]}=\big((\mu_n^+)^{-1}(0)\big)^\textup{st}/\GL_n(\mathbb C),\]
the quotient being free. As such, it is smooth and symplectic of dimension $2n$.\end{enumerate}
\end{theorem}

\begin{remark}\label{remsthilb} Written shortly\[
(\mathbb C^2)^{[n]}=\left\{(x,x^*,v)\in {\Rep}_{\mathbb C}(S_2,n)\times\mathbb C^n\middle|\begin{aligned}& [x,x^*]=0\\&\mathbb C[x,x^*].v=\mathbb C^n\end{aligned}\right\}/\GL_n(\mathbb C).\]
Note that thanks to Theorem~\ref{nakrecap}(ii), one can also write\[
(\mathbb C^2)^{[n]}=\big(\mu_{n,1}(0)^{-1}\big)^\text{\!st}/\PGL_{n,1}(\mathbb C)\]
where $\mu_{n,1}$ is the moment map associated to the action of $\GL_{n,1}(\mathbb C):=\GL_n(\mathbb C)\times\mathbb C^*$ on ${\Rep}_{\mathbb C}(S_1^+,(n,1))$ given by $(g,z).(x,v)=(gxg^{-1},gvz^{-1})$ (\emph{i.e.}\ forgetting about the framing character of $S_1^+$).
\end{remark}

\begin{remark}
As in~\S\ref{lagcomm}, one gets a lagrangian subvariety of $(\mathbb C^2)^{[n]}$ when requiring $x$ to be nilpotent. It can be understood as the preimage of $(\{0\}\times\mathbb C)^{(n)}$ by the \emph{Hilbert--Chow morphism} \[
\rho: (\mathbb C^2)^{[n]}\rightarrow (\mathbb C^2)^{(n)}\]
 to the $n$-th symmetric power of $\mathbb C^2$, which maps $(x,x^*,v)$ to the joint spectrum of $(x,x^*)$ (which make sense since $x$ and $x^*$ commute).\end{remark}

\subsubsection{New lagrangian subvarieties}\label{laghilb}

 Note that \[
 \mathrm{tr}(W):~~\raisebox{-0.45\height}{
\begin{adjustbox}{max totalsize=10cm}
\begin{tikzpicture}
\node(2)at(-1.5,0){$\mathbb C$};
\node(1)at(0,0){$\mathbb C^n$};
\draw[-latex](2)edge[above]node{$v$}(1);
\draw[-latex](1)edge[out=-30,in=30,looseness=10,left]node{$x$}(1);
\draw[-latex](1)edge[out=50,in=110,looseness=10,above ]node{$y$}(1);
\draw[-latex](1)edge[out=-110,in=-50,looseness=10,below]node{$z$}(1);
\end{tikzpicture}
\end{adjustbox}}~~~~\mapsto~~\mathrm{tr}([x,y]z)
\]  is still defined on the quotient stack $[{\Rep}_{\mathbb C}(S_3^+,(n,1))/\PGL_{n,1}(\mathbb C)]$, allowing us to apply the recipe 
of~\S\ref{lagcomm} to the projection 
\begin{align*}
\pi :\big[{\Rep}_{\mathbb C}(S_3^+,(n,1))/\PGL_{n,1}(\mathbb C)\big]&\rightarrow \big[{\Rep}_{\mathbb C}(S_1^+, (n,1))/\PGL_{n,1}(\mathbb C)\big]\\
(x,y,z,v)&\mapsto(x,v)\,.
\end{align*}
We will see later in~\S\ref{quivsch} that the $0$-th truncation of the cotangent stack to the quotient stack 
$[{\Rep}_{\mathbb C}(S_1^+, (n,1))/\PGL_{n,1}(\mathbb C)]$ is given by $[\mu_{n,1}^{-1}(0)/\PGL_{n,1}(\mathbb C)$]. 
Hence, thanks to Remark~\ref{remsthilb}, 
\[
(\mathbb C^2)^{[n]}=\Big(\tau_0\bT^*\big[{\Rep}_{\mathbb C}(S_1^+, (n,1))/\PGL_{n,1}(\mathbb C)\big]\Big)^\text{st}
\]
where $\tau_0$ stands for the $0$-truncation induced by the morphism from a complex to its $0$-th cohomology group.
In this framework, we can still make sense of the definition given in~\cite{RS}, and set \[
\Lambda_{n,1}=\left\{(x,x^*=[y,z],v)\middle|\begin{aligned}& [x,y]=[x,z]=0\\&\mathbb C[x,x^*].v=\mathbb C^n\end{aligned}\right\}/\GL_n(\mathbb C)\subset (\mathbb C^2)^{[n]}.\]
Note that it is distinct from $\rho^{-1}((\{0\}\times\mathbb C)^{(n)})$.

\begin{proposition}
The subvariety $\Lambda_{n,1}$ is lagrangian in $(\mathbb C^2)^{[n]}$.
\end{proposition}

\begin{proof}
The subvariety $\Lambda_{n,1}$ is a free quotient by $\GL_n(\mathbb C)$ of an open subvariety of $\Lambda_n\times\mathbb C^n$. 
Hence  irreducible components of $\Lambda_{n,1}$ are $n$-dimensional, and in bijection with those of $\Lambda_n$ as they 
generically admit a cyclic vector (nulltrace matrices generically admit cyclic vectors). Components of $\Lambda_{n,1}$ are 
isotropic as the symplectic form on $(\mathbb C^2)^{[n]}$ is induced by the one on $T^*{\Rep}_\mathbb C(S_1,n)$: it is given by
\[
\big((x,x^*,v),(x',x'^*,v')\big)\mapsto\mathrm{tr}(xx'^*-x^*x')\,.
\qedhere\]
\end{proof}

\begin{remark}
We will see an alternative direct proof in~\S\ref{lag2lagsec} relying on derived symplectic geometry. 
\end{remark}

We can ``saturate'' $\Lambda_{n,1}$ with respect to the Hilbert--Chow morphism $\rho$. We make sense of this saturation in the following statement, and prove that it gives a new family of lagrangian subvarieties of $(\mathbb C^2)^{[n]}$.

\begin{proposition}\label{satu}
The \emph{saturation} $L_n:=\rho^{-1}(\rho(\Lambda_{n,1}))\subset(\mathbb C^2)^{[n]}$ of $\Lambda_{n,1}$ with 
respect to $\rho$ is lagrangian, with irreducible components indexed by {nested partitions} of $n$. 
\end{proposition}

By nested partition of $n$, we mean a tuple of partitions $\mu^1,\dots,\mu^r$ such that 
$|\mu^1|\ge\dots\ge|\mu^r|$ defines a partition $\lambda$ of $n$.
For a given partition $\lambda=(\lambda_1\ge\dots\ge\lambda_s)$ of $n$, write 
$\bar\lambda_k=\sum_{p\le k}\lambda_p$ and define 
\[
(\mathbb C^2)_0^\lambda=\left\{[(a_1,b_1),\dots,(a_n,b_n)]\middle|\begin{aligned}&\forall  k=0\ldots r-1:\\& a_{\bar\lambda_k+1}=\ldots=a_{\bar\lambda_{k+1}}\\&\text{the }a_{\bar\lambda_k}\text{ are pairwise distinct}\\&\sum_{\bar\lambda_k+1\le p\le\bar\lambda_{k+1}}b_p=0\end{aligned}\right\}\subset (\mathbb{C}^2)^{(n)}
\]
where $\bar\lambda_0=0$.

\begin{lemma}\label{satulem}
We have $\rho(\Lambda_{n,1})=\sqcup_{\lambda\vdash n}(\mathbb C^2)_0^\lambda$. 
\end{lemma}

\begin{proof}
We have $\rho(\Lambda_{n,1})\subseteq \sqcup_{\lambda\vdash n}(\mathbb C^2)_0^\lambda$ because if 
$(x,[y,z],v)\in\Lambda_{n,1}$, $y$ and $z$ stabilize all generalized eigenspaces (GES) of $x$, and 
the trace of $[y,z]$ is zero on each GES. Reciprocally, consider $[a_\bullet,b_\bullet]\in(\mathbb C^2)_0^\lambda$ 
and $x$ diagonalizable with eigenvalues $\alpha_k=a_{\bar\lambda_k}$ such that $E_k=\ker(x-\alpha_k)$ has dimension 
$\lambda_k$. Pick $y_k\in \End(E_k)$ with spectrum $b_{\bar\lambda_k+1},\dots,b_{\bar\lambda_{k+1}}$ such that there 
exists $v_k\in E_k$ cyclic under the action of $y_k$. Then, since any endomorphism commutes with homotheties, and 
since being a commutator is equivalent to having a zero trace, $\oplus_k(\alpha_k,y_k,v_k)$ is in $\Lambda_{n,1}$, 
and is mapped to $[a_\bullet,b_\bullet]$. Note that $\oplus_kv_k$ is cyclic because the $\alpha_k$'s are pairwise 
distinct.
\end{proof}

\begin{proof}[Proof of Proposition~\ref{satu}]
Thanks to Lemma~\ref{satulem}, we have
\[
L_{n}=\left\{(x,x^*,v)\middle|\begin{aligned}& [x,x^*]=0\\&\mathbb C[x,x^*].v=\mathbb C^n\\&\mathrm{tr}(x^*)=0
\text{ on each GES of }x\end{aligned}\right\}/\GL_n(\mathbb C)\,.
\]
Consider a nested partition $\mu^\bullet$ of $n$, and denote by $\lambda$ the partition $|\mu^\bullet|\vdash n$. 
Set $L_{\mu^\bullet}\subset \rho^{-1}((\mathbb C^2)_0^\lambda)$ the locally closed subvariety defined by
\[
(x-\alpha_k)_{E_k}\in\mathcal O_{\mu^k}
\]
with the same notations as in the proof of Lemma~\ref{satulem}, and where $\mathcal O_{\mu^k}$ is the nilpotent 
orbit of type $\mu^k\vdash\lambda_k$.
Define
\begin{align*}
\widehat{L_{\mu^\bullet}}&=\left\{(x,x^*,v)\middle| [x,x^*,v]\in L_{\mu^\bullet}\right\}\\
\widetilde{L_{\mu^\bullet}}&=\left\{(x,x^*,v,\phi_\bullet)\middle| [x,x^*,v]\in L_{\mu^\bullet}\text{ and }\phi_k\in \Isom(E_k,\mathbb C^{\lambda_k})\right\}
\end{align*}
which fit in the following diagram
\begin{equation}\label{facto}
\xymatrix{
\widehat{L_{\mu^\bullet}}\ar[d]_-{\pi_n}& \widetilde{L_{\mu^\bullet}}\ar[l]_-{p}\ar[r]^-{q}	&	 (\prod_k \widehat{L_{\mu^k}})^\circ\ar[d]^-{\Pi_k\pi_k}\\
{L_{\mu^\bullet}}	&											& (\prod_k L_{\mu^k})^\circ
}
\end{equation}
where $\pi_m$ is the free quotient by $\GL_m(\mathbb C)$, $p$ (which simply forgets $\phi_\bullet$) is a principal 
$\prod_k\GL_{\lambda_k}(\mathbb C)$-bundle, and $q=\prod_k(\phi_{k})_*$ is a principal $\GL_n(\mathbb C)$-bundle on the open subvariety $(\prod_k \widehat{L_{\mu^k}})^\circ$ corresponding to distinct $\alpha_k$'s (again 
it matters that the $\alpha_k$'s are pairwise distinct so that the direct sum of a family of cyclic vectors remains cyclic). 
Thanks to the following Lemma~\ref{irredbase} we get that $L_{\mu^\bullet}$ is irreducible of dimension 
\begin{align*}
&\sum_k\lambda_k+\sum_k\dim\pi_k+\dim q-\dim p-\dim\pi_n\\
&=n+\sum_k\lambda_k^2+n^2-\sum_k\lambda_k^2-n^2=n
\end{align*}
as wished since $(\mathbb C^2)^{[n]}$ is of dimension $2n$. 
\end{proof}

\begin{lemma}
\label{irredbase} Each $L_{\mu^k}$ is irreducible of dimension $\lambda_k$. 
\end{lemma}

\begin{proof}
The variety $L_{\mu^k}$ is a free $\GL_{\lambda_k}(\mathbb C)$-quotient of 
\[
\{(\alpha+n,x^*,v)\mid n\in\mathcal O_{\mu^k}, [n,x^*]=0, \mathrm{tr}(x^*)=0,\mathbb C[n,x^*].v=\mathbb C^n\}
\]
which is a nonempty open subvariety of
\[
\mathbb C\times \mathbb C^{\lambda_k} \times( T^*\mathcal O_{\mu^k}/\mathbb C)
\]	
where we quotient by the action by translation on $x^*$. Hence $L_{\mu^k}$ is irreducible of dimension
\[
-\lambda_k^2+1+\lambda_k+(\lambda_k^2-1)=\lambda_k
\]
since $T^*\mathcal O_{\mu^k}$ is irreducible of dimension $\lambda_k^2$. The nonemptiness statement 
comes from the fact that it is enough to take $v$ cyclic on each Jordan block of $n$ as long as the 
eigenvalues of $x^*$ are pairwise distinct with respect to these blocks.
\end{proof}

\begin{remark}\label{annder}~

\begin{enumerate}[label=(\roman*)]
\item
The interest of the saturation process lies in the diagram~\eqref{facto}, which enjoys a typical \emph{factorization} property, as described for instance in~\cite{GiFi} or more recently in~\cite{KAVA}. We can expect $U\mapsto L_n(U)=L_n\cap\rho^{-1}(U^{(n)})$ to define a factorization algebra on $\mathbb C^2$ with values in lagrangian subvarieties. Actually, going back to~\cite{GoSt} or~\cite{GaGi}, it is very natural to wonder whether the irreducible components of our lagrangian subvarieties appear in the characteristic cycles of representations of some Cherednik algebra.
\item It is clear from what is explained at the beginning of this section that there are underlying derived concepts at stake here. We will explore these in section~\ref{sec:examples}.\end{enumerate}
\end{remark}

%% file: CritQuiv-Section3.tex
The content of this section is not new, but we recall it for the convenience of the reader, and in order to fix the notation 
and clarify the exposition. 

\medskip

In the remainder of this paper, $k$ is a field of characteristic zero, and we denote by $\cat{Mod}_k$ 
the category of (unbounded, cochain) complexes of $k$-vector spaces. It comes equipped with a model structure for which 
weak equivalences are quasi-isomorphisms, fibrations are degreewise surjections, and cofibrations are degreewise 
injections. We write $\scat{Mod}_k$ for its $\infty$-categorical localization along quasi-isomorphisms (in the sense of \cite[\S 1.1.2]{Hinich}). 
We refer to \cite{HTT} for the foundations of $\infty$-category theory (in particular, an $\infty$-category is by definition a quasi-category).  
For a complex $M$ we write 
\[
|M|:=\cat{Map}_{\scat{Mod}_k}(k,M)\,,
\]
where $\cat{Map}_{\scat{C}}(x,y)$ denotes the space (i.e.~Kan complex) of morphisms from $x$ to $y$ in an $\infty$-category $\scat{C}$. 
Note that $|M|$ has a fairly explicit description, as the Kan complex associated with the connective truncation $\tau_{\leq0}M$ \textit{via} the 
Dold--Kan correspondence  (this is because the standard simplicial enrichment 
of $\cat{Mod}_k$ is such that homotopy equivalences coincide with quasi-isomorphisms, hence the corresponding simplicial category presents $\scat{Mod}_k$, 
in the sense that its simplicial nerve is equivalent to $\scat{Mod}_k$). 
\begin{remark}
All our results remain true if $k$ is a commutative ring containing $\mathbb{Q}$, at the cost of taking $k$-flat replacements. 
\end{remark}

\medskip

More generally, if $\cat{M}$ is a model category, we write $\scat{M}$ for the corresponding 
$\infty$-category obtained by localizing along weak equivalences. Note that the full embeddings 
$\cat{M}^f\hookrightarrow \cat{M}\hookleftarrow \cat{M}^c$ of fibrant, resp.~cofibrant, objects 
induce equivalences $\scat{M}^f\tilde\longrightarrow \scat{M}\tilde\longleftarrow \scat{M}^c$ between localizations (see \cite[1.3.4]{Hinich}). 
\begin{notation}
The corresponding homotopy category is denoted $\Ho(\cat{M})$, or $\Ho(\scat{M})$. 
\end{notation}
We make use on several occasions of the fact that a Quillen adjunction between model categories gives rise to an adjunction between $\infty$-categories 
(see \cite[1.5.1]{Hinich}), the former being an equivalence if and only if the latter is. 

\medskip

Finally recall (\cite[Proposition A.1.1]{CPTVV}) 
that if $\cat{M}$ is a $\cat{Mod}_k$-model category, then it is a stable model category, and thus $\scat{M}$ is 
a $k$-linear stable $\infty$-category. 
\begin{notation}
We use the notation ``$\cat{Map}$'' to distinguish the \textit{space} of $\infty$-categorical morphisms from the \textit{set} 
of $1$-categorical morphisms, for which we use the notation ``$\cat{Hom}$''. The underlined versions designate their enriched 
counterparts. 
In particular, for $\cat{M}$ as above, we have $\mathbb{R}\underline{\cat{Hom}}_{\cat{M}}\simeq\underline{\cat{Map}}_{\scat{M}}$. 
\end{notation}
If moreover $\cat{M}$ has a symmetric monoidal structure that is compatible with the $\cat{Mod}_k$-enrichment and the 
model structure (see e.g.~\cite[\S1.1]{CPTVV}), then we write, for every object $M$ of $\scat{M}$ (or, equivalently, of $\cat{M}$)
\[
|M|:=\cat{Map}_{\scat{M}}(\mathbf{1},M)\simeq |\mathbb{R}\ul{\cat{Hom}}_{\cat{M}}(\mathbf{1},M)|\,,
\]
where $\mathbf{1}$ is the monoidal unit. 
\begin{notation}
If the symmetric monoidal structure is closed, then we will use upper case letters for the internal enrichment: i.e.~
$\textsc{Hom}$ and $\textsc{Map}$, for categories and $\infty$-categories, respectively. 
\end{notation}

\subsection{Graded mixed complexes and realizations}

\subsubsection{Graded complexes}\label{sssec:graded}

We write $\cat{Mod}_k^{\gr}$ for the category of functors $\mathbb{Z}\to \cat{Mod}_k$, where the set $\mathbb{Z}$ is seen 
as a category with no non-identity morphism. An object $M=(M_{p})_{p\in\mathbb{Z}}$ in $\cat{Mod}_k^{\gr}$ is called a graded 
complex, and the complex $M_p$ is called the weight $p$ component of $M$. We consider the weightwise model structure on 
$\cat{Mod}_k^{\gr}$: a morphism $\varphi=(\varphi_p)_{p\in\mathbb{Z}}$ is a weak equivalence/fibration/cofibration if and only 
if every $\varphi_p$ is so. 

There are two shift autofunctors of $\cat{Mod}_k^{\gr}$: 
\begin{itemize}
\item The (cohomological) degree shift $M\mapsto M[1]$, defined by $M[1]^n=M[n+1]$; 
\item The weight shift $M\mapsto M(1)$, defined by $M(1)_p=M_{p+1}$. 
\end{itemize}
The weight $p$ component functor $M\mapsto M_p$, going from graded complexes to complexes, has both left and right adjoint 
the functor $\iota_p$ that sends a complex $V$ to the graded complex having $V$ as only nontrivial weight component, 
in weight $p$. As a matter of notation, we will often drop $\iota_0$ from the notation, implicitely viewing a complex as a 
graded complex sitting in weight $0$. 
\begin{remark}
All the functors we have introduced so far are both left and right Quillen. 
\end{remark}
There is also an interesting functor $k[v]\otimes_k-\,:\,\cat{Mod}_k\to\cat{Mod}_k^{\gr}$, with $v$ a variable of degree $2$ 
and of weight $1$, that has both a left and a right adjoint: 
\begin{itemize}
\item The left adjoint $|-|^\ell$ sends a graded complex $M$ to $\underset{p\geq0}{\coprod}M_p[2p]$, and it preserves 
weak equivalences and cofibrations for the injective model structure; 
\item the right adjoint $|-|^r$ sends a graded complex $M$ to $\underset{p\geq0}{\prod}M_p[2p]$, and it preserves weak 
equivalences and fibrations for the projective model structure. 
\end{itemize}
In particular, each of these functors is equivalent to its derived counterpart. Moreover, one observes that there is 
a natural transformation $|-|^\ell\Rightarrow |-|^r$. 

\subsubsection{Mixed complexes}\label{sssec:mixed}

The category $\cat{Mod}_k^\delta$ is the symmetric monoidal category of differential graded $k[\delta]$-modules, 
where $\delta$ is a degree $-1$ variable. We will call its objects \textit{mixed complexes}: in other words, a mixed 
complex is a complex $C$ together with a morphism $\delta:C\to C[-1]$, the \textit{mixed differential}, such that 
$\delta[-1]\circ\delta=0$ (which we will often abbreviate $\delta^2=0$). 
As usual there are two (Quillen equivalent) model structures structures on $\cat{Mod}_k^\delta$: 
\begin{itemize}
\item The projective one, for which weak equivalences are quasi-isomorphisms, and fibrations are epimorphisms; 
\item The injective one, for which weak equivalences are quasi-isomorphisms, and cofibrations are monomorphisms. 
With the injective model structure, it satisfies the standing assumptions of \cite[\S1.1]{CPTVV}.  
\end{itemize} 
We have a sequence of Quillen adjunctions (the first one obviously being an equivalence) 
\[
\xymatrix{
(\cat{Mod}_k^{\delta})_{\mathrm{inj}}\ar@<-1ex>[r]_-{\mathrm{id}} 
& (\cat{Mod}_k^{\delta})_{\mathrm{proj}}\ar@<-1ex>[l]_-{\mathrm{id}}\ar@<-1ex>[r]_-{(-)^\natural}
& \cat{Mod}_k \ar@<-1ex>[l]_-{k[\delta]\otimes_k-}
}
\]
where $(-)^\natural$ denotes the functor that forgets the mixed differential $\delta$ (i.e.~it is the restriction 
along the unit morphism $k\to k[\delta]$). 

\medskip

The morphism of mixed complexes $k[1]=k\delta\to k[\delta]$ induces, for every $n$, a morphism of spaces
\begin{equation}\label{eq:Connes}
|C^\natural[n]|	\simeq	\cat{Map}_{\scat{Mod}_k^\delta}\big(k[\delta][-n],C\big)
				\longrightarrow		|C[n-1]|
\end{equation}
witnessing the fact that $\delta$ maps the space of $n$-cocyles in $C^\natural$ to the space of mixed $(n-1)$-cocycles 
in $C$. A \textit{mixed cocycle} of $C$ is a cocycle of the homotopy $\delta$-fixed points complex 
\[
\mathbb{R}\ul{\cat{Hom}}_{\cat{Mod}_k^\delta}(k,C)\simeq (C[\![u]\!],d-u\delta)\,,
\]
with $u$ a formal variable of degree $2$. Indeed, in order to compute the latter derived enriched Hom in the projective 
model structure, we first consider the following quasi-free resolution of $k$: $Q(k)=k[\delta,\xi]$, with $deg(\xi)=-2$ 
and $d\xi=\delta$, and we observe that   
\[
\ul{\cat{Hom}}_{\cat{Mod}_k^\delta}\big(Q(k),C\big)=(C[\![u]\!],d-u\delta)\,.
\]
Therefore the map \eqref{eq:Connes} in $\scat{Mod}_k$ is represented by 
\[
C^\natural[1]\overset{\delta}{\longrightarrow}(C[\![u]\!],d-u\delta)\,.
\]

\subsubsection{Graded mixed complexes}\label{sssec:grmixed}

A \textit{graded mixed complex} is a graded complex equipped with a compatible $k[\varepsilon]$-module structure, 
where $\varepsilon$ has degree $1$ and weight $1$: in other words, a graded mixed complex is a graded complex $M$ 
together with a morphism $\varepsilon:M\to M(1)[1]$ such that $\varepsilon(1)[1]\circ\varepsilon=0$ (which we abbreviate 
$\varepsilon^2=0$). There are again two Quillen equivalent model structures on the category $\cat{Mod}_k^{\varepsilon-\gr}$ 
of graded mixed complexes: 
\begin{itemize}
\item The projective one, for which weak equivalences are weightwise quasi-isomorphisms, and fibrations are weightwise epimorphisms; 
\item The injective one, for which weak equivalences are weightwise quasi-isomorphisms, and cofibrations are weightwise monomorphisms. 
With the injective model structure, it satisfies the standing assumptions of \cite[\S1.1]{CPTVV}.  
\end{itemize}
There is a sequence of Quillen adjunctions (the first one being a Quillen equivalence) 
\[
\xymatrix{
(\cat{Mod}_k^{\varepsilon-\gr})_{\mathrm{inj}}\ar@<-1ex>[r]_-{\mathrm{id}} 
& (\cat{Mod}_k^{\varepsilon-\gr})_{\mathrm{proj}}\ar@<-1ex>[l]_-{\mathrm{id}}\ar@<-1ex>[r]_-{(-)^\sharp}
& \cat{Mod}_k^{\gr} \ar@<-1ex>[l]_-{k[\varepsilon]\otimes_k -}
}
\]
where $(-)^\sharp$ denotes the functor that forgets the mixed differential $\varepsilon$. Both the weight and the degree 
shift autofunctors lift along $(-)^\sharp$, and are denoted the same way. 

\medskip

The morphism of graded mixed complexes $k(-1)[-1]=k\varepsilon \to k[\varepsilon]$ gives rise, for every $n$ and every $p$, 
to a morphism 
\begin{equation}\label{eq:epsilon}
|M^\sharp(p)[n]|	\simeq				\cat{Map}_{\scat{Mod}_k^{\varepsilon-\gr}}\big(k[\varepsilon](-p)[-n],M\big)
					\longrightarrow		|M(p+1)[n+1]|
\end{equation}
witnessing the fact that $\varepsilon$ maps the space of weight $p$ $n$-cocyles of $M^\sharp$ to the space of weight $p+1$ mixed 
$(n+1)$-cocycles of $M$; a \textit{weight $q$ mixed cocycle} of $M$ is a cocycle in the complex 
\[
\mathbb{R}\ul{\cat{Hom}}_{\cat{Mod}_k^{\varepsilon-\gr}}\big(k(-q),M\big)\simeq \big(\prod_{p\geq q}M_p,d-\varepsilon\big)\,.
\]
Indeed, in order to compute the latter derived enriched Hom in the projective model structure, we first consider the following 
quasi-free resolution of $k(-q)$: $Q_q(k)=k[\varepsilon]\{x_q,x_{q+1},\dots\}$, with $x_p$ of degree $0$ and weight $p$, 
and $d(x_p)=\varepsilon(x_{p-1})$ (with the convention that $x_p=0$ for $p<q$). We have    
\[
\ul{\cat{Hom}}_{\cat{Mod}_k^{\varepsilon-\gr}}\big(Q_q(k),M\big)=\big(\prod_{p\geq q}M_p,d-\varepsilon\big)\,.
\]
Therefore the map \eqref{eq:epsilon} in $\scat{Mod}_k$ is represented by 
\[
M_p[-1]\overset{\varepsilon}{\longrightarrow}\big(\prod_{r\geq p+1}M_r,d-\varepsilon\big)\,.
\]

\medskip

The functors $k[v]\otimes$, $|-|^\ell$ and $|-|^r$ from \S\ref{sssec:graded}, between $\cat{Mod}_k$ and $\cat{Mod}_k^\gr$, 
also lift (along the functors $\natural$ and $\sharp$) to functors between $\cat{Mod}_k^\delta$ and $\cat{Mod}_k^{\varepsilon-\gr}$, 
that we denote the same way, and that are defined as follows: 
\begin{itemize}
\item The mixed differential on $k[v]\otimes C$ is defined as $\varepsilon:=v\otimes\delta$; 
\item The mixed differential on $|M|^\ell$ and $|M|^r$ is defined as $\delta:=\varepsilon$. 
\end{itemize}
The functors $|-|^\ell$ and $|-|^r$ are called the (left and right) \textit{mixed realization functors}. 
We let the reader check that $|-|^\ell$ (resp.~$|-|^r$) is still left adjoint (resp.~right adjoint) to $k[v]\otimes_k-$, 
that it still preserves weak equivalences and cofibration of the injective model structure (resp.~fibrations of the projective 
model structure), and that the natural transformation $|-|^\ell\Rightarrow |-|^r$ also lifts.  

Note that for any $n\in\mathbb{Z}$ and any $p\geq0$, both mixed realization functors send the morphism 
$k(-1-p)[-n-1]\to k[\varepsilon](-p)[-n]$ to $k[n+1]\to k[\delta][n]$. 

\medskip

Let us draw some consequences of the above. Assume that we are given a morphism of mixed complexes $C\to |M|^r$, 
which is equivalent, by adjunction, to have a morphism $C[v]\to M$ of graded mixed complexes. For every $p\geq0$, 
we have a commuting diagram 
\[
\xymatrix{
\big|C[n]\big| 					\ar[r]\ar[d]_{(-)^\natural} 		& 
\big|C[v](p)[n+2p]\big| 		\ar[r]\ar[d]_{(-)^\sharp}			& 
\big|M(p)[n+2p]\big| 			\ar[d]_{(-)^\sharp}					\\
\big|C^\natural[n]\big|			\ar[r]\ar[d]_{\eqref{eq:Connes}}	&
\big|C[v]^\sharp(p)[n+2p]\big|	\ar[r]\ar[d]_{\eqref{eq:epsilon}}	& 
\big|M^\sharp(p)[n+2p]\big|		\ar[d]_{\eqref{eq:epsilon}}			\\
\big|C[n-1]\big|				\ar[r]								&
\big|C[v](p+1)[n+2p+1]\big|		\ar[r]								&
\big|M(p+1)[n+2p+1]\big|\,.
}
\]
In other words, the space of mixed cocycles of degee $n$ in $C$ maps to the space of weight $p$ mixed $(n+2p)$-cocycles in $M$, 
and this commutes with 
\begin{itemize}
\item Taking the underlying non-mixed cocycle of a mixed one; 
\item Creating a mixed cocycle from a non-mixed one by applying the mixed differential.  
\end{itemize}

\subsection{De Rham complex and (closed) forms}\label{ssec:derham}

Let $\cat{CAlg}_k$ be the category of commutative differential graded $k$-algebra sitting in non-positive degree 
(\textit{cdga}, for short). It has a model structure for which weak equivalences are quasi-isomorphisms, 
and fibrations are degreewise surjections. Recall from \cite{PTVV} the functor 
\[
\scat{DR}\,:\,\scat{CAlg}_k\longrightarrow \scat{Mod}_k^{\varepsilon-\gr}
\]
which sends a cdga $B$ to $\mathrm{Sym}_{\tilde{B}}(\Omega^1_{\tilde{B}}[-1])$ equipped with the mixed differential  $\varepsilon$ 
being the (extension by the Leibniz rule of the) universal derivation 
\[
d_{dR}:\tilde{B}\to\Omega^1_{\tilde{B}}=\Omega^1_{\tilde{B}}[-1][1]\,, 
\]
where $\tilde{B}$ is a cofibrant replacement of $B$. 

\medskip

Recall also the spaces 
\[
\mathcal{A}^p(B,n):=|\mathbf{DR}(A)^\sharp(p)[n+p]|\quad\textrm{and}\quad
\mathcal{A}^{p,\cl}(B,n):=|\mathbf{DR}(A)(p)[n+p]|
\]
of \textit{$p$-forms} of degree $n$ and \textit{closed $p$-forms} of degree $n$, respectively.  
Notice that a (closed) $p$-form of degree $n$ is a weight $p$ (mixed) $(n+p)$-cocycle. 

\medskip

The functor $\mathbf{DR}$, and thus $\mathcal{A}^p[n]:=\mathcal{A}^p(-,n)$ and 
$\mathcal{A}^{p,\cl}[n]:=\mathcal{A}^{p,\cl}(-,n)$, satisfies \'etale descent 
(see \cite[Proposition 1.11]{PTVV}). We then define these functors on any derived stack $F$ by Kan extension along 
$\scat{Spec}:\scat{CAlg}_k\to\scat{dSt}_k^\op$, where $\scat{dSt}_k$ denotes the $\infty$-category of derived $k$-stacks. 

Recall that derived $k$-stacks are functors $\scat{CAlg}_k\to\scat{S}$ satisfying \'etale descent.
We refer to \cite{HAG-II} for the background material on derived algebraic geometry (note that even though \cite{HAG-II} and other references 
that we use are written using simplicial model categories, there is no harm in phrasing everything in terms of $\infty$-categories, thanks to 
\cite[Appendices A.2 \& A.3]{HTT} -- see also \cite{Hinich}). 

\medskip

According to \S\ref{sssec:grmixed}, for every $p\geq0$ and every $n\in\mathbb{Z}$, we have a morphisms of stacks
\[
\mathcal{A}^{p,\cl}[n]\longrightarrow\mathcal{A}^p[n]\longrightarrow\mathcal{A}^{p+1,\cl}[n]
\]
respectively given by $(-)^\sharp$ and \eqref{eq:epsilon}. 
We denote the first one $(-)_0$ and the second one $d_{dR}$. 

\medskip

There is also a functor 
\[
\scat{HH}:=S^1\otimes_k^{\mathbb{L}}-\,:\,\scat{CAlg}_k\longrightarrow \scat{Mod}_k^\delta
\]
given by the Hochschild mixed complex. 
\begin{remark}[see e.g.~\cite{Hoyois}]\label{remHoyois}
\textit{A priori} $\scat{HH}$ takes values in the $\infty$-category of complexes with an $S^1$-action, which is equivalent 
to $\scat{Mod}_k^\delta$ (because $\scat{Mod}_k$ is $k$-linear, $S^1$ is formal, and $H_{-\bullet}(S^1,k)=k[\delta]$). 
Recall the simplicial model for $S^1$ from \cite[\S6.4.2 \& 7.1.2]{Loday}, given by the simplical set of finite cyclic orders. 
Therefore the underlying complex of $\scat{HH}(A)$ is the homotopy colimit of the simplicial diagram 
\[
\xymatrix{
\cdots A^{\otimes(n+1)}\cdots A^{\otimes 3}\ar@<1ex>[r]\ar[r]\ar@<-1ex>[r]
& A^{\otimes 2} \ar@<0.5ex>[r]^-{m}\ar@<-0.5ex>[r]_-{m^{\op}}
& A\,,
}
\]
where $m$ is the product of $A$. A model for this homotopy colimit is given by the standard Hochschild complex. 
Moreover, the above simplicial cochain complex actually has a cyclic structure, which induces a mixed structure 
on the Hochschild complex, given by the $B$-operator of Connes. 
\end{remark}
\begin{notation}\label{not:hhhc}
We write $\cat{HH}:=\scat{HH}^\natural$ for the underlying Hochschild complex, i.e.~the actual Hochschild complex without 
the mixed differential. We write $\cat{HC}^-:=\ul{\cat{Map}}_{\scat{Mod}_k^\delta}(k,\scat{HH})$ for the homotopy 
$\delta$-fixed point complex (also called \textit{negative cyclic complex}). 
\end{notation}
The Hochschild--Kostant--Rosenberg theorem (see e.g.~\cite{ToVeS1}) tells us that there is a natural equivalence 
\[
I_{HKR}\,:\,\scat{HH}\tilde\longrightarrow |\mathbf{DR}|^\ell\,,
\]
in $\scat{Mod}_k^\delta$, given explicitely on $A^{\otimes n+1}$ by $\frac1{n!}\mathrm{id}\otimes d_{dR}^{\otimes n}$, for a cofibrant $A$. 

\subsection{Shifted symplectic and lagrangian structures}

\subsubsection{Shifted presymplectic and isotropic structures}

Working with $n$-shifted presymplectic and isotropic structures amounts to work within the category 
${\scat{dSt}_k}_{/\mathcal{A}^{2,\cl}[n]}$. 

\medskip

Let $X\overset{f}{\longleftarrow} Z\overset{g}{\longrightarrow} Y$ be a correspondence of derived $k$-stacks. 
\begin{definition}
\begin{enumerate}
\item An \textit{$n$-shifted presymplectic structure} on $X$ is the data of a closed $2$ form of degree $n$ 
\[
\omega\,:\, X\longrightarrow \mathcal{A}^{2,\cl}[n]
\]
\item An \textit{$n$-shifted isotropic stucture} on the correspondence 
$X\overset{f}{\longleftarrow} Z\overset{g}{\longrightarrow} Y$ is the data of a (homotopy) commuting square 
\[
\xymatrix{
Z \ar[r]^-g\ar[d]_-f & Y \ar[d] \\
X \ar[r] & \mathcal{A}^{2,\cl}[n]
}
\]
\item If $Y$ (resp.~$X$) is a point, then we talk about an ($n$-shifted) isotropic structure on the morphism $f$ (resp.~$g$).
\item The above structures are said \textit{exact} if they lift along $d_{dR}:\mathcal{A}^1[n]\to\mathcal{A}^{2,\cl}[n]$, 
the data of the lift being part of the structure. 
\end{enumerate}
\end{definition}
According to \cite[\S5]{Haug}, for every $\infty$-category $\scat{C}$ with finite limits, one can construct an 
$\infty$-category $\scat{Span}_1(\scat{C})$ of spans, or correspondences: its objects are the ones of $\scat{C}$, 
and $1$-morphisms are spans. We are interested in the following two cases: 
\begin{itemize}
\item The $\infty$-category $\scat{Iso}_{[n]}$ of $n$-shifted isotropic correspondences, which corresponds to the case 
$\scat{C}={\scat{dSt}_k}_{/\mathcal{A}^{2,\cl}[n]}$; 
\item The $\infty$-category $\scat{ExIso}_{[n]}$ of exact $n$-shifted isotropic correspondences, which corresponds to 
the case $\scat{C}={\scat{dSt}_k}_{/\mathcal{A}^{1}[n]}$. 
\end{itemize}
The morphism $d_{dR}:\mathcal{A}^1[n]\to\mathcal{A}^{2,\cl}[n]$ induces a functor 
$\scat{D_{dR}}:\scat{ExIso}_{[n]}\to\scat{Iso}_{[n]}$. 

\subsubsection{Non-degeneracy conditions}

Recall that with every cdga $A$, one associates the symmetric monoidal $\cat{Mod}_k$-enriched model category 
$\cat{Mod}_A$ of $A$-modules. 
This assignment is functorial, leading to a functor $\scat{QCoh}$ from $\scat{Calg}_k$ to the 
$\infty$-category of stable symmetric monoidal $k$-linear categories: $\scat{QCoh}(A)=\scat{Mod}_A$. 

The functor $\scat{QCoh}$ is a stack for the \'etale topology, and one can define it on every derived stack $X$ by Kan 
extension along $\scat{Spec}$. Objects of $\scat{QCoh}(X)$ are called \textit{quasi-coherent sheaves} on $X$. 
Just like the one of $\scat{Mod}_A$, the symmetric monoidal structure of $\scat{QCoh}(X)$ is closed. 

\medskip 

For convenience, we will abuse the standard terminology and say that a derived $k$-stack $X$ is \textit{Artin} if it is 
locally geometric and locally of finite presentation (we refer to \cite{HAG-II} for the foundational material on geometric stacks). 
Local geometricity guaranties the existence of a cotangent complex $\mathbb{L}_X\in\scat{QCoh}(X)$, and local finitness 
ensures that $\mathbb{L}_X$ is perfect. If $X$ is Artin, then, according to \cite[Proposition 1.14]{PTVV}, 
\[
\mathcal{A}^p(X,n)\simeq \big|\Gamma\big(\mathrm{Sym}_{\mathcal{O}_X}^p(\mathbb{L}_X[-1])\big)[n+p]\big|\,,
\]
where $\Gamma:\scat{QCoh}(X)\to\scat{Mod}_k$ is the global section functor: $\Gamma\simeq\ul{\cat{Map}}(\mathcal O_X,-)$. 
\begin{remark}
When $p=0$, the above equivalence is true without any assumption on $X$: 
$\mathcal{A}^0(X,n)\simeq|\Gamma(\mathcal O_X)[n]|$. 
\end{remark}
Hence a $2$-form $\beta$ of degree $n$ on a derived Artin stack $X$ gives rise to a degree $n$ global section of the second exterior 
power of $\mathbb{L}_X$, and thus to an adjoint morphism $\beta^\flat:\mathbb{T}_X\to\mathbb{L}_X$, where $\mathbb{T}_X$ 
is the dual to the cotangent complex. 
\begin{definition}
\begin{enumerate}
\item An $n$-shifted presymplectic structure $\omega$ on a derived Artin stack $X$ is \textit{symplectic} if its underlying 
$2$-form $\omega_0$ is \textit{non-degenerate}, meaning that $\omega_0^\flat$ is an equivalence. 
\item An $n$-shifted isotropic structure on a correspondence $X\overset{f}{\longleftarrow} Z\overset{g}{\longrightarrow} Y$ 
between derived Artin stacks is \textit{lagrangian} if it is \textit{non-degenerate}, meaning that the two presymplectic 
structures (on $X$ and $Y$) are non-degenerate and the square 
\[
\xymatrix{
\mathbb{T}_Z\ar[r]\ar[d] & g^*\mathbb{T}_Y\simeq g^*\mathbb{L}_Y[n] \ar[d] \\
f^*\mathbb{T}_X\simeq f^*\mathbb{L}_X[n] \ar[r] & \mathbb{L}_Z[n]
}
\]
is (co)cartesian. 
\item If $Y$ (resp.~$X$) is a point, then we talk about a ($n$-shifted) lagrangian structure on the morphism $f$ (resp.~$g$).
\end{enumerate}
\end{definition}

We finally recall that $n$-shifted Lagrangian correspondences do compose well (see \cite[Theorem 4.4]{CalLag}). We therefore 
have a subcategory $\scat{Lag}_{[n]}$ of $\scat{Iso}_{[n]}$ whose objects are $n$-shifted symplectic derived Artin stacks and 
whose $1$-morphisms are $n$-shifted Lagrangian correspondences. 
Similarly, there is a subcategory $\scat{ExLag}_{[n]}$ of $\scat{ExIso}_{[n]}$ whose objects and morphisms are those 
that are sent to $\scat{Lag}_{[n]}$ under $\scat{D_{dR}}$. 
\begin{example}
Recall from \cite{CalLag} that an $n$-shifted (exact) lagrangian structure on $X\to\mathrm{pt}=\scat{Spec}(k)$ is equivalent to 
an $(n-1)$-shifted (exact) symplectic structure on $X$. This in particular implies that fiber products of $n$-shifted (exact) 
lagrangian morphisms, also called \textit{derived (exact) lagrangian intersections}, are naturally $(n-1)$-shifted (exact) symplectic, 
as was shown in \cite{PTVV}. 
\end{example}

%% file: CritQuiv-Section4.tex
\subsection{Shifted cotangent stacks}

\subsubsection{Linear stacks}\label{linst}

Given a quasi-coherent sheaf $E$ on a derived stack $X$, one can define a derived stack $\mathbb{A}_E$ over $X$ as follows: 
\[
\mathbb{A}_E\big(\scat{Spec}(A)\overset{u}{\to}X\big):=|u^*E|\,.
\] 
The construction is functorial in $E$, and we write $\pi_X:\mathbb{A}_E\to X$ for the map corresponding to $E\to 0$. 
The construction is also compatible with pull-backs: for every $f:Y\to X$, $\mathbb{A}_{f^*E}\simeq f^*\mathbb{A}_E$. 
\begin{remark}
If $X$ is Artin and if $E$ is perfect, then $\mathbb{A}_E$ is Artin too, and it coincides with the derived stack 
$\mathbb{V}(E^\vee)$ from \cite[\S3.3]{ToEMS}. 
\end{remark}
If $X$ is an Artin derived $k$-stack and $f:E\to F$ is a morphism between perfect quasi-coherent sheaves on $X$, then 
the relative tangent complex of $\mathbb{A}_f:\mathbb{A}_E\to\mathbb{A}_F$ is 
\[
\mathbb{T}_{\mathbb{A}_f}\simeq \pi_X^*\cat{fib}(f)\,.
\]
In particular, if $F=0$, then $\mathbb{T}_{\pi_X}\simeq\pi_X^*E$, so that we get equivalences 
\[
\mathbb{A}_{\mathbb{T}_{\pi_X}}
\simeq \mathbb{A}_{\pi_X^*E}
\simeq\mathbb{A}_E\underset{X}{\times}\mathbb{A}_E
\simeq\mathbb{A}_{E\oplus E}
\quad\textrm{and}\quad \mathbb{A}_{\mathbb{L}_{\pi_X}}\simeq \mathbb{A}_{E\oplus E^\vee}\,.
\]
\begin{example}\label{exGmod}
Let $G$ be an algebraic group, and let $V$ be $k$-module with a linear $G$-action; this defines a quasi-coherent 
sheaf $\mathbf{V}$ on $BG$. 
One has $\mathbb{A}_{\mathbf{V}}\simeq \big[\mathbb{A}_V/G\big]$, which is sometimes just denoted $[V/G]$. 
Then $\mathbb{L}_{[V/G]}$ is the pull-back along $p_{V,G}\,:\,[V/G]\to BG$ of the quasi-coherent sheaf on $BG$ corresponding 
to the dual representation $V^*$. Hence 
\[
\mathbb{A}_{\mathbb{L}_{p_{V,G}}}\simeq [V\oplus V^*/G]\,.
\]
\end{example}
The above situation specializes to the following
\begin{example}\label{exQuiv}
Let $Q$ be a quiver, with vertex set $V(Q)$ and edge set $E(Q)$. The edges are oriented, we denote by $s(e)$ and $t(e)$ the 
source and target of an edge $e$. Let $\vec{n}=(n_v)_{v\in V(Q)}$ be a dimension vector. 
We define
\[
{\Rep}_k(Q,\vec{n}):=\prod_{e\in E(Q)}\cat{Hom}_{\cat{Mod}_k}(k^{n_{s(e)}},k^{n_{t(e)}})\,,
\]
which is a $\GL_{\vec{n}}(k)$-module. The relative cotangent complex of the morphism 
\[
p_{k,Q,\vec{n}}\,:\,\big[{\Rep}_k(Q,\vec{n})/\GL_{\vec{n}}(k)\big]\longrightarrow \BGL_{\vec{n}}(k)
\]
is therefore the pull-back (along $p_{k,Q,\vec{n}}$) of the $\GL_{\vec{n}}(k)$-module ${\Rep}_k(Q^*,\vec{n})$, where $Q^*$ 
stands for the opposite quiver (built from $Q$ by reversing all edges), and where we view a $\GL_{\vec{n}}(k)$-module as 
a quasi-coherent sheaf on $\BGL_{\vec{n}}(k)$.
As a consequence, we have 
\[
\mathbb{A}_{\mathbb{L}_{p_{k,Q,\vec{n}}}}\simeq \big[{\Rep}_k(\overline{Q},\vec{n})/\GL_{\vec{n}}(k)\big]\,,
\]
where $\overline{Q}$ stands for the doubled quiver (built from $Q$ by adjoining to every edge $e$ a new edge $e^*$ going 
the reverse way). 
\end{example}
The following example is a variation on the previous one. 
\begin{example}\label{exQuiv2}
Let us consider the derived $k$-stack $\scat{Perf}_k$ of perfect complexes from \cite{HAG-II}, that is defined by 
\[
\scat{Perf}_k(A):=(\scat{Mod}_A^{\mathrm{perf}})^\simeq\,,
\]
where $\scat{Mod}_A^{\mathrm{perf}}$ is the full subcategory spanned by perfect $A$-modules, and the superscript $(-)^\simeq$ 
means that we only keep the $1$-morphisms that are equivalences. A morphism $X\to\scat{Perf}_k$ is by definition the 
data of a perfect sheaf on $X$; therefore the identity morphism $\scat{Perf}_k\to\scat{Perf}_k$ classifies a tautological 
perfect sheaf $E$ on $\scat{Perf}_k$. 

For a finite quiver $Q$, with vertex set $V(Q)$ and edge set $E(Q)$, we consider the perfect sheaf 
\[
E^Q:=\prod_{e\in E(Q)}\textsc{Map}\big(p_{s(e)}^*E,p_{t(e)}^*E\big)
\]
on $\scat{Perf}_k^{V(Q)}$, where $p_v:\scat{Perf}_k^{V(Q)}\to \scat{Perf}_k$ is the $v$-th projection ($v\in V(Q)$). 
The relative cotangent complex of the morphism 
\[
p_Q\,:\,\mathbb{A}_{E^Q}\longrightarrow \scat{Perf}_k^{V(Q)}
\]
is therefore $(E^Q)^\vee\simeq E^{Q^*}$, where $(-)^\vee=\textsc{Map}(-,\mathcal O)$ is the dualization functor. 
As a consequence, 
\[
\mathbb{A}_{\mathbb{L}_{p_Q}}\simeq \mathbb{A}_{E^{\overline{Q}}}\,.
\]
\end{example}
\begin{remark}\label{labelresBG}
Later on, we will identify $\mathbb{A}_{E^Q}$ with the moduli of objects $\scat{Perf}_{kQ}$. For the moment, 
it is sufficient to observe that the restriction of $E$ to the open substack $\BGL_n(k)\subset\scat{Perf}_k$ is $k^n$, 
so that the restriction of $E^Q$ to the open substack $\BGL_{\vec{n}}(k)\subset \scat{Perf}_k^{V(Q)}$ is ${\Rep}_k(Q,\vec{n})$. 
\end{remark}

\subsubsection{Shifted cotangent stacks are shifted symplectic}

The \textit{$n$-shifted cotangent stack} of a derived Artin $k$-stack $X$ is 
\[
\mathbf{T}^*[n]X:=\mathbb{A}_{\mathbb{L}_X[n]}\,.
\]
In other words, a morphism $Y\to\mathbf{T}^*[n]X$ is determined by a morphism $u:Y\to X$ and a global 
section of $u^*\mathbb{L}_X[n]$. 
The identity morphism $\mathbf{T}^*[n]X\to\mathbf{T}^*[n]X$ is thus determined by a global 
section $\ell_X$ of $\pi_X^*\mathbb{L}_X[n]$. Let us denote $\lambda_X$ the image of $\ell_X$ 
through $\pi_X^*\mathbb{L}_X[n]\longrightarrow \mathbb{L}_{\mathbf{T}^*[n]X}[n]$: 
\[
\lambda_X\in \mathcal A^1(\mathbf{T}^*[n]X,n)\,.
\]
We know (see \cite[Theorem 2.2(1)]{CalCot}) that $\omega_X:=d_{dR}\lambda_X\in\mathcal A^{2,\cl}(\mathbf{T}^*[n]X,n)$ 
defines an $n$-shifted symplectic structure on $\mathbf{T}^*[n]X$. 
\begin{remark}\label{remark: 1-form}
The $1$-form $\lambda_X$ can also be described as follows using the functor of points. 
Recall that a $B$-point of $\mathbf{T}^*[n]X$ is determined by a point $x\in X(B)$ and a section $s\in|x^*\mathbb{L}_X[n]|$. 
Then 
\[
\lambda_X(x,s)=x^*s\in|\mathbb{L}_B[n]|=\mathcal{A}^1(B,n)\,.
\]
\end{remark}
\begin{example}\label{BGex}
For a reductive affine algebraic group $G$, $\mathbb{L}_{BG}$ is the quasi-coherent sheaf associated with the shifted 
coadjoint $G$-module $\mathfrak{g}^*[-1]$. Therefore 
\[
\mathbf{T}^*[n](BG)\simeq \big[\mathfrak{g}^*[n-1]/G\big]\,.
\]
Moreover, the section $\ell_{BG}$ has a simple expression 
\[
\ell_{BG}=\mathrm{id}_{\mathfrak{g}[1-n]}\in\big|(\mathfrak{g}^*[n-1]\otimes\mathfrak{g}[1-n])^G\big|\subset
\big|\big(\mathfrak{g}^*[-1]\otimes\cat{Sym}(\mathfrak{g}[1-n])\big)^G[n]\big|\simeq|\Gamma(\pi_{BG}^*\mathbb{L}_{BG})[n]|\,.
\]
Actually, the above still makes sense without assuming $G$ reductive; we still have a map 
\[
\big|\big(\mathfrak{g}^*[-1]\otimes\cat{Sym}(\mathfrak{g}[1-n])\big)^G[n]\big|\longrightarrow
|\Gamma(\pi_{BG}^*\mathbb{L}_{BG})[n]|\,,
\]
which is not necessarily an equivalence (as there might be non-trivial higher $G$-cohomology). 
Notice that in the reductive case, one can choose a non-degenerate invariant quadratic form $c\in\mathrm{Sym}^2(\mathfrak{g}^*)^G$ 
and thus get an isomorphism $c^\flat:\mathfrak{g}\to\mathfrak{g}^*$ of $G$-modules, so that 
\[
\mathbf{T}^*[n](BG)\simeq \big[\mathfrak{g}[n-1]/G\big]\,.
\]
Through this identification, we have 
\[
\ell_{BG}=c\in (\mathfrak{g}^{\otimes2})^G\subset\big|\big(\mathfrak{g}^*[-1]\otimes\cat{Sym}(\mathfrak{g}^*[1-n])\big)^G[n]\big|
\simeq|\Gamma(\pi_{BG}^*\mathbb{L}_{BG})[n]|\,.
\]
\end{example}
\begin{example}\label{linPerfk}
It is known (see \cite[Corollary 3.29]{ToVa}) that the derived $k$-stack $\scat{Perf}_k$ is locally geometric 
and locally of finite presentation, and that 
\[
\mathbb{T}_{\scat{Perf}_k}\simeq\textsc{Map}(E,E)[1]\,.
\]
Therefore 
\[
\mathbf{T}^*[n]\scat{Perf}_k\simeq\mathbb{A}_{\textsc{Map}(E,E)^\vee[n-1]}\simeq \mathbb{A}_{\textsc{Map}(E,E)[n-1]}
\]
as there is an isomorphism of perfect sheaves
\[
c^\flat\,:\,\textsc{Map}(E,E)\tilde\longrightarrow \textsc{Map}(E,E)^\vee\,;\, f\longmapsto \mathrm{tr}(f\circ-)\,.
\]
Through the above identification, on $\mathbb{A}_{\textsc{Map}(E,E)[n-1]}$, $\ell_{\scat{Perf}_k}$ becomes a section of 
\[
\pi_{\scat{Perf}_k}^*\textsc{Map}(E,E)^\vee[n-1]\simeq\pi^*_{\scat{Perf}_k}\mathbb{L}_{\scat{Perf}_k}\,,
\]
which is the section classifying $c^\flat[n-1]$. 
\end{example}

\subsubsection{Lagrangian correspondences for shifted cotangent stacks}\label{lagshicotst}

Let $f:X\to Y$ be a morphism of derived Artin $k$-stacks. 
We have a morphism $f^*\mathbb{L}_Y\to \mathbb{L}_X$ leading, for every $n\in\mathbb{Z}$, to a correspondence  
\begin{equation}\label{eq:corrcot}
\bT^*[n]X\overset{q}\longleftarrow f^*\bT^*[n]Y\overset{p}{\longrightarrow} \bT^*[n]Y
\end{equation}
of derived Artin $k$-stacks. 
Observe that, by definition, $p^*\ell_Y$ is the section of $p^*\pi_Y^*\mathbb{L}_Y[n]\simeq q^* \pi_X^*f^*\mathbb{L}_Y[n]$ 
classifed by $p$, and it coincides with the one classified by $\mathrm{id}:f^*\mathbf{T}^*[n]Y\to f^*\mathbf{T}^*[n]Y$. 
Hence the image of $p^*\ell_Y$ through the morphism $q^* \pi_X^*f^*\mathbb{L}_Y[n]\to q^* \pi_X^*\mathbb{L}_X[n]$ is classified by 
$q$, and thus it coincides with $q^*\ell_X$. As a consequence, we get that $p^*\lambda_Y$ coincides with $q^*\lambda_X$ 
(as sections of $\mathbb{L}_{f^*\mathbf{T}^*[n]Y}$), and that the correspondence \eqref{eq:corrcot} lifts to an exact 
isotropic one: 
\[
\xymatrix{
& f^*\bT^*[n]Y \ar[rd]^-{p} \ar[dl]_-{q} & \\
\bT^*[n]X \ar[rd]_-{\lambda_X} && \bT^*[n] Y  \ar[ld]^-{\lambda_Y} \\
&\mathcal A^1[n] &}
\]
\begin{remark}
For a $B$-point given by $x\in X(B)$ and $s\in|x^*f^*\mathbb{L}_Y[n]|$, this boils down to the commutativity 
of the diagram of $B$-modules
\[
\xymatrix{
& x^*f^*\mathbb{L}_Y[n] \ar@{=}[rd] \ar[dl] & \\
x^*\mathbb{L}_X[n] \ar[rd] && (f\circ x)^*\mathbb{L}_Y[n]  \ar[ld] \\
&\mathbb{L}_B[n] &}
\]
\end{remark}
It can be shown (along the lines of \cite[Theorem 2.8]{CalCot}) that this isotropic correspondence is in fact lagrangian. 
\begin{example}
When $Y=\mathrm{pt}$, we get back the exact lagrangian structure on the zero section $X\to\bT^*[n]X$ from 
\cite[Theorem 2.2(2)]{CalCot}. 
\end{example}
\begin{proposition}\label{prop:cotangentfunctor}
The above construction defines a functor 
\[
\bT^*[n]\,:\,\Ho(\scat{dSt}_k^\mathrm{Art})\longrightarrow \Ho(\scat{ExLag}_{[n]})\,,
\]
where $\scat{dSt}_{k}^\mathrm{Art}$ is the full subcategory of $\scat{dSt}_{k}$ spanned by Artin stacks. 
\end{proposition}
\begin{proof}
First of all, the identity map $X\to X$ is easily seen to be sent to the identity correspondence 
\[
\xymatrix{
& \bT^*[n]X \ar[rd]^-{\mathrm{id}} \ar[dl]_-{\mathrm{id}} & \\
\bT^*[n]X \ar[rd]_-{\lambda_X} && \bT^*[n] X  \ar[ld]^-{\lambda_X} \\
&\mathcal A^1[n] &}
\]
Then, given a sequence $X\overset{f}{\longrightarrow} Y\overset{g}{\longrightarrow} Z$, we do have the following 
commuting diagram 
\[
\xymatrix{
&& (g\circ f)^*\bT^*[n]Z \ar[rd]\ar[ld] && \\
& f^*\bT^*[n]Y \ar[rd]\ar[ld] && g^*\bT^*[n]Y \ar[rd]\ar[ld] & \\
\bT^*[n]X \ar[rrd]^-{\lambda_X} && \bT^*[n]Y \ar[d]^-{\lambda_Y} && \bT^*[n]Z \ar[lld]_-{\lambda_Z} \\
&& \mathcal A^{1}[n] &&
}
\]
The fact that it commutes follows from the commutativity of the diagram of $B$-modules
\[
\xymatrix{
&& x^*(g\circ f)^*\mathbb{L}_Z[n] \ar@{=}[rd] \ar[ld] && \\
& x^*f^*\mathbb{L}_Y[n] \ar@{=}[rd]\ar[ld] && (f\circ x)^*g^*\mathbb{L}_Z[n] \ar@{=}[rd]\ar[ld] & \\
x^*\mathbb{L}_X[n] \ar[rrd] && (f\circ x)^*\mathbb{L}_Y[n]\ar[d] && (g\circ f\circ x)^*\mathbb{L}_Z[n] \ar[lld] \\
&& \mathbb{L}_B[n] &&
}
\]
where $x\in X(B)$. 
\end{proof}
\begin{remark}
We neither prove nor use use it, but one can actually show that $\bT^*[n]$ lifts to a functor 
\[
\scat{dSt}_{k}^\mathrm{Art}\longrightarrow\scat{ExLag}_{[n]}\,,
\]
at the level of $\infty$-categories. 
\end{remark}
\begin{definition}
The conormal exact $n$-shifted lagrangian $\bT^*_X[n]Y\to\bT^*[n]Y$ is the composition of $f^*\bT^*[n]B$ with the zero 
section $X\to\bT^*[n]X$ in $\scat{ExLag}_{[n]}$: 
\[
\xymatrix{
&& \bT^*_X[n]Y \ar[rd]\ar[ld] && \\
& X \ar[rd]\ar[ld] && f^*\bT^*[n]Y \ar[rd]\ar[ld] & \\
{\mathrm{pt}} \ar[rrd]^-{0} && \bT^*[n]X \ar[d]^-{\lambda_X} && \bT^*[n]Y \ar[lld]_-{\lambda_Y} \\
&& \mathcal A^{1}[n] &&
}
\]
\end{definition}
Then one proves that 
\begin{align}\label{cotformula}
\bT^*_X[n]Y\underset{{\bT^*[n]Y}}{\times}Y\simeq\bT^*[n-1]X\,,
\end{align}
as exact $(n-1)$-shifted symplectic stacks. This is a special case of \cite[Remark 2.7]{CalCot}, but we can also 
prove it thanks to the following diagram of exact lagrangian correspondences
\begin{align}\label{grodiag}
\xymatrix{
&&& \bT^*[n-1]X \ar[rd]\ar[ld] &&& \\
&& \bT^*_X[n]Y \ar[rd]\ar[ld] && X \ar[rd]\ar[ld] && \\
&X \ar[rd]\ar[dl] && f^*\bT^*[n]Y \ar[rd]^-{p} \ar[dl]_-{q} && Y \ar[ld]\ar[dr] &\\
\mathrm{pt} \ar@/_/[rrrd]_-0 && \bT^*[n]X \ar[rd]^-{\lambda_X} && \bT^*[n] Y  \ar[ld]_-{\lambda_Y} && \mathrm{pt} \ar@/^/[llld]^-0 \\
&&& \mathcal A^{1}[n] &&&
}
\end{align}
Indeed $\bT^*[n-1]X\simeq X\underset{\bT^*[n]X}{\times}X\simeq\bT^*_X[n]Y\underset{\bT^*[n]Y}{\times}Y$. 
\begin{example}\label{exRedux}
Our aim is to describe the cotangent stack to the quotient stack $[V/G]$, where $G$ is a reductive affine algebraic group 
and $V$ is a finite dimensional $G$-module. According to \eqref{cotformula}, and Examples~\ref{exGmod} and~\ref{BGex}, 
it is given by the derived exact lagrangian intersection 
\[
\xymatrix{
\bT^*[V/G] \ar[r] \ar[d]	& BG 		\ar[d]										\\
[V\oplus V^*/G] \ar[r]	& [\mathfrak{g}^*/G]
}\]
The bottom morphism is given by a $G$-equivariant map $V\times V^*\to \mathfrak{g}^*$. 
We let the reader check that this map is the usual moment map, sending $(v,v^*)$ to the linear form $x\mapsto v^*(x\cdot v)$, 
where $x\cdot$ is the $\mathfrak{g}$-action. 
The exact lagrangian structure of this morphism can be described as follows. Recall that 
\[
\Gamma(\mathbb{L}_{[V\oplus V^*/G]})\simeq 
\mathrm{hofib}\Big(\Omega^1(V\times V^*)\to \mathcal O(V\times V^*)\otimes\mathfrak{g}^*\Big)^G\,,
\]
where the map is the transpose of the infinitesimal action. 
Therefore, by definition, the pull-back of the canonical element $\mathrm{id}_{\mathfrak{g}^*}$ is the tensor of the 
$\mathfrak{g}$-action, lying in $(V\otimes V^*\otimes\mathfrak{g}^*)^G$, which is the image of the tautological $1$-form 
on $V\times V^*=T^*V$ under the transpose of the infinitesimal action. Hence the exact lagrangian structure on 
$[V\oplus V^*/G] \to [\mathfrak{g}^*/G]$ is given by the tautological $1$-form on $V\times V^*$, which is $G$-equivariant. 
\end{example}

\subsubsection{Symplectic structure on derived quiver representation schemes}\label{quivsch}

First for a given quiver $Q$, let us introduce the \emph{path algebra} $kQ$ generated by paths, with concatenation as product. 
For each $v\in V(Q)$ we denote by $e_v$ the idempotent path of length $0$. They generate a subalgebra of $kQ$ isomorphic to 
$k^{V(Q)}$ as a vector space, that will be denoted by $R$ when there is no ambiguity. We will make use of the Jordan quiver $L$ 
with one vertex and one loop, which satisfy $kL=k[x]$ for some degree $0$ variable $x$. Note that $kL^{V(Q)}$ is an $R$-algebra.

For a fixed dimension vector $\vec{n}=(n_v)_{v\in V(Q)}$, we want to describe the cotangent stack to 
the non-derived moduli stack $\big[{\Rep}_k(Q,\vec{n})/\GL_{\vec{n}}(k)\big]$ of $\vec{n}$-dimensional representations 
of $Q$, following the notation of Example \ref{exQuiv}. According to Example~\ref{exRedux}, it is obtained as the derived 
lagrangian intersection
\[
\xymatrix{
\bT^*\big[{\Rep}_k(Q,\vec{n})/\GL_{\vec{n}}(k)\big] \ar[r] \ar[d]	& \BGL_{\vec{n}}(k) 		\ar[d]										\\
\big[{\Rep}_k(\overline{Q},\vec{n})/\GL_{\vec{n}}(k)\big] \ar[r]	& \big[\mathfrak{gl}_{\vec{n}}(k)/\GL_{\vec{n}}(k)\big]
}\]
where the bottom map is given by the moment map 
\[
{\Rep}_k(\overline{Q},\vec{n})\longrightarrow \mathfrak{gl}_{\vec{n}}(k)={\Rep}_k(L^{V(Q)},\vec{n})\,,
\]
which is just pulling-back representations along the $R$-algebra morphism $kL^{V(Q)}\to k\overline Q$ given by 
\[
\sum_{v\in V(Q)}x_v\,\longmapsto\,\sum_{e\in E(Q)}(ee^*-e^*e)\,. 
\]
One can compute the above derived fiber product by first taking the derived fiber at $0$ of the moment map, 
and then taking the quotient by $\GL_{\vec{n}}(k)$. 

\medskip

Let us now explain shortly that the derived fiber at $0$ of the moment map can be described as the derived representation 
scheme of a certain dg-algebra. Let us consider the model category $\cat{dgAlg}_{R}$ of dg-$R$-algebras, which contains 
$kQ$, $kL^{V(Q)}$ or $\mathfrak{gl}_{\vec{n}}(k)$. 
Inspired by \cite{BKR}, 
we define, for a dg-$R$-algebra $A$, the derived representation stacks 
\[
\scat{DRep}(A,\vec{n})\,:\,B\mapsto \cat{Map}_{\scat{dgAlg}_{R}}\big(A,\mathfrak{gl}_{\vec{n}}(B)\big)\,.
\]
One can show that 
\begin{itemize}
\item If $A$ is sitting in non-positive degree, these derived stacks are in fact derived schemes; 
\item If $A$ is free and sitting in degree $0$, i.e.~$A$ is the path algebra of a quiver $Q$, then 
\[
\scat{DRep}(A,\vec{n})\simeq {\Rep}_k(Q,\vec{n})\,;
\]
\item The construction defines a limit preserving functor 
\[
\scat{DRep}(-,\vec{n})\,:\,(\scat{dgAlg}_{R})^{\op}\longrightarrow \scat{dSt}_k 
\]
\end{itemize}
As a consequence we get that the homotopy fiber at $0$ of the moment map is the derived representation scheme 
(of dimension $\vec{n}$) of the derived push-out 
\[
\mathcal G_2(kQ)	:=		{k}\overline{Q}\overset{\mathbb{L}}{\underset{{k}[x]^{V(Q)}}{\coprod}} k^{V(Q)}
							\simeq	{k}\overline{Q}\overset{\mathbb{L}}{\underset{{k}[x]}{\coprod}} k\,,
\]
where the second tensor product takes place in $\scat{dgAlg}_k$, and the morphism ${k}[x]\to k\overline{Q}$ 
sends $x$ to 
\[
\sum_{e\in E(Q)}(ee^*-e^*e)\,.
\]
We thus obtain that 
\[
\bT^*\big[{\Rep}_k(Q,\vec{n})/\GL_{\vec{n}}(k)\big]\simeq \big[\scat{DRep}\big(\mathcal{G}_{2}(kQ),\vec{n}\big)/\GL_{\vec{n}}(k)\big]\,.
\]

\begin{remark}\label{remG2}
The dg-algebra $\mathcal G_2(kQ)$ is the $2$-Calabi--Yau Ginzburg dg-algebras (see~\cite{Gi}) associated to $Q$ with no 
potential. It is generated by some variable $x$ in degree $-1$ and $k\overline Q$ in degree $0$, with 
$\mathfrak d(x)=\sum_{e\in E(Q)}ee^*-e^*e$. A precise study of this push-out will be done in~\S\ref{subsubginz}. 
Its $0$-th cohomology is the so-called \emph{preprojective algebra} $\Pi_k(Q)=k\overline Q/ (\sum_{e\in E(Q)}[e,e^*])$.
\end{remark}

\begin{remark}\label{remlinquiv}
We now describe a variation on the above, allowing complexes of representations, in the spirit of Examples \ref{exQuiv2} 
and \ref{linPerfk}. More precisely, we compute the cotangent stack of $\mathbb A_{E^Q}$.
According to~\eqref{cotformula}, it is given by derived (exact) lagrangian intersection
\[
\xymatrix{
\bT^* \mathbb A_{E^Q} \ar[r]	\ar[d] 	& \bPerf_{ k^{ V(Q)}} \ar[d] \\
\mathbb A_{E^{\overline{Q}}}													\ar[r]		&  \mathbb{A}_{\textsc{Map}(E,E)^{V(Q)}}}
\]
since $\bT^*_{\mathbb A_{E^Q}}[1] \bPerf_{ k^{ V(Q)}}\simeq\mathbb A_{E^{\overline{Q}}}$ thanks to Example~\ref{exQuiv2}, 
and $\bT^*[1] \bPerf_{ k^{ V(Q)}}\simeq\mathbb A_{\textsc{Map}(E,E)^{V(Q)}}$ thanks to Example~\ref{linPerfk}. 
The bottom line $\mathbb A_{E_{\overline Q}} \to  \mathbb A_{   \textsc{Map}(E,E)^{V(Q)}}$ is induced by 
$E_{\overline Q}\to\textsc{End}(E)^{V(Q)}$, $(x_e,x_{e^*})\mapsto\sum_e[x_e,x_{e^*}]$. 
We will see in Corollary~\ref{theorem: compareGinzburg} that $\bT^* \mathbb A_{E^Q}$ can be identified with the moduli of objects 
of $\mathcal{G}_2(kQ)$. 
\end{remark}


\subsection{Relative derived critical loci}\label{ss:critrel}

\subsubsection{Definition}\label{critrel}

Inspired by~\cite{RS}, as explained in the introductory part of section~\ref{sec:motiv}, we want to define a relative version 
of critical loci. We do so by replacing the zero section in~\eqref{grodiag} by the graph $X\to \bT^*X$ of a closed $1$-form. 
We thus start with the following datum:
\begin{itemize}
\item a closed $1$-form $\alpha$ on $X$;
\item a morphism $\pi:X\to B$ of derived $k$-stacks.
\end{itemize}
For simplicity and because it will always be the case in our applications, we assume that $\alpha=d\phi$ for some function $\phi:X\to\mathbb{A}^1$. We know from section~\ref{lagshicotst} that we have an exact isotropic correpondence 
$T^*X\leftarrow \phi^*T^*\mathbb A^1=X\times \mathbb{A}^1 \to T^*\mathbb{A}^1=\mathbb{A}^1\times\mathbb{A}^1$. 
When we compose it with the lagrangian correspondence $T^*\mathbb{A}^1\leftarrow \{(q,1)|q\in\mathbb{A}^1\}\to *$ 
we get the lagrangian graph of $d\phi:X\to \mathbf{T}^*X$.
 The composition 
 \begin{align}\label{critreldiag}
  \xymatrix{& & \bcrit_\pi(\phi) \ar[rd] \ar[dl] && \\
&X \ar[rd]^-{d\phi}\ar[ld] && \pi^* \bT^* B  \ar[ld]_-q\ar[rd]^-p& \\
~~~\mathrm{pt}~~~&&\bT^* X &&~\bT^*B}\end{align}
 defines the \emph{relative derived critical locus} of $\phi$ with respect to the constraint $\pi$. 
 We get the usual derived critical locus when the base $B$ is a point.
 
\subsubsection{The quiver case}\label{pot3CY}

Assume that $\pi$ is given by the restriction morphism $\iota^*:{\Rep}_k( Q,\vec n)\to {\Rep}_k(D,\vec m)$ induced by an inclusion 
of finite, arbitrary quivers $D\subset Q$ (not necessarily a full inclusion), with $m_v=n_v$ for every $v\in V(D)$.

We then consider a \emph{potential} $W$ of the quiver $Q$, that is a linear combination of cycles in $Q$. 
It determines a $\GL_{\vec{n}}$-invariant function $\phi=\mathrm{tr}(W)$ on ${\Rep}_k( Q,\vec n)$. For any $e\in E(Q)$ we denote 
by $\partial_eW$ the linear combination of paths obtained by removing $e$ from the cycles of $W$.

If we do not perform any quotient, $X$, $\bT^*X$ and $\pi^*\bT^*B$ are genuine vector spaces and 
\[
\bcrit_\pi(\mathrm{tr}(W))=\scat{DRep}_k\left(k Q\overset{\bbL}{\underset{k\overline Q}{\coprod}} k(Q\sqcup D^*),\vec n\right)\,,
\]
where $\delta W:k\overline Q\to kQ$ is generated by 
\[
E(Q)\ni e\mapsto e\quad,\quad E(Q^*)\ni e^*\mapsto \partial_eW
\] 
and $k\overline Q\to k(Q\sqcup D^*)$ is the obvious projection.
We resolve $kQ$ by the free $k\overline Q$-dg-algebra generated by $e'\in E(Q^*)$ in degree $-1$ with differential\[
e'\mapsto e^*-\partial_eW\]
 and get that the derived push-out\[
 k Q\underset{k\overline Q}{\overset{\bbL}{\coprod}} k(Q\sqcup D^*)\]
  is generated by $e'\in E(Q^*)$ in degree $-1$ and $e\in E(Q), e^*\in E(D^*)$ in degree $0$ with differential given by:\[
 E(D^*)\ni e'\mapsto e^*-\partial_eW\quad,\quad E(Q^*\setminus D^*)\ni e'\mapsto-\partial_eW.\]

Hence, the $0$-th truncation of $\bcrit_\pi(\mathrm{tr}(W))$ is\[
\tau_0\bcrit_\pi(\mathrm{tr}(W))=\left\{(x,x^*)\in {\Rep}_k(Q\sqcup D^*,\vec n)\middle|\begin{aligned}x_{\partial_eW}&=x^*_{e}\text{ if }e\in E(D)\\x_{\partial_eW}&=0\text{ otherwise}\end{aligned}\right\}.\]

\begin{example}\label{linex}
Let us consider some specific linear examples.\begin{enumerate}[label=(\roman*)]
\item $D=A_2=$\raisebox{-0.10\height}{
\begin{tikzpicture}\hspace{0cm}
\node(A)at(0,0){$\circ$};
\node(B)at(2,1){$\circ$};
\draw[->](A)to[bend left]node[midway,below]{$a$}(B);
\end{tikzpicture}}
$\subset Q=\tilde A_2=$\raisebox{-0.48\height}{
\begin{tikzpicture}\hspace{0cm}
\node(A)at(0,0){$\circ$};
\node(B)at(2,1){$\circ$};
\draw[->](A)to[bend left]node[midway,below]{$a$}(B);
\node(D)at(2,-1){$\circ$};
\draw[->](B)to[bend left]node[midway,left]{$b$}(D);
\draw[->](D)to[bend left]node[midway,above]{$c$}(A);
\end{tikzpicture}}, $W=abc$. We have
\[
\tau_0\bcrit_\pi(\mathrm{tr}(W))=\left\{(x_a,x_{a^*},x_b,x_{c})\middle|\begin{aligned}x_bx_c&=x_{a^*}\\x_cx_a&=0\\x_ax_b&=0\end{aligned}\right\}\,.
\]
One can show that the image of $\bcrit_{\pi}(\mathrm{tr}(W))$ in $T^*{\Rep}_k(A_2,\vec m)$ is actually the classical Lusztig 
 lagrangian subvariety~\cite[§12]{Lusztig1} for finite type quiver 
\[
\Lambda_{\vec m}(A_2):=\{(x_a,x_{a^*})\mid x_ax_{a^*}=0,~x_{a^*}x_a=0\}\,.
\]
\item The case of~\S\ref{lagcomm}: $S_1=\!\!$\raisebox{-0.46\height}{
\begin{adjustbox}{max totalsize=4cm}
\begin{tikzpicture}
\node(1)at(0,0){$\circ$};
\draw[-latex](1)edge[out=-40,in=40,looseness=12,right]node{$a$}(1);
\end{tikzpicture}
\end{adjustbox}}
$\subset
S_3=\!\!\!\!\!\!\!\!$\raisebox{-0.46\height}{
\begin{adjustbox}{max totalsize=4cm}
\begin{tikzpicture}
\node(1)at(0,0){$\circ$};
\draw[-latex](1)edge[out=-40,in=40,looseness=12,right]node{$a$}(1);
\draw[-latex](1)edge[out=80,in=160,looseness=12,above ]node{$b$}(1);
\draw[-latex](1)edge[out=200,in=280,looseness=12,below]node{$c$}(1);
\end{tikzpicture}
\end{adjustbox}}, $W=[a,b]c$. This time\[
\tau_0\bcrit_\pi(\mathrm{tr}(W))=\left\{(x_a,x_{a^*},x_b,x_{c})\middle|\begin{aligned} [x_b,x_c] &=x_{a^*}\\ [x_c,x_a]&=0\\ [x_a,x_b]&=0\end{aligned}\right\}\]
and observe that its image in $T^*{\Rep}_\mathbb C(S_1,n)$ is the subvariety $\Lambda_n$ defined in~\S\ref{lagcomm}.
\end{enumerate}
\end{example}

Now consider \[\pi:[{\Rep}_k( Q,\vec n)/\GL_{\vec n}(k)]\to[{\Rep}_k( D,\vec n)/\GL_{\vec m}(k)]\]
 where $[{\Rep}_k( Q,\vec n)/\GL_{\vec n}(k)]$ is still endowed with the function $\phi=\mathrm{tr}(W)$.
We want to describe the underlying dg-algebra hidden behind $\bcrit_\pi(f)$, as announced by Remark~\ref{annder}(ii).
To that end, we first recall the definition of Ginzburg dg-algebras~\cite{Gi}. 

\begin{definition}\label{defW3CY}
Consider $(Q,W)$ a quiver with potential. 
We denote by $\mathcal G_3(kQ,\delta W)$ the dg-$R$-algebra generated by $e\in E(Q)$ in degree $0$, $e'\in E(Q^*)$ in 
degree $-1$ and $x'$ in degree $-2$, 
with concatenation as product and the following differential
\[
\mathfrak d: \quad x'\mapsto \sum_{e\in E(Q)}(ee'-e'e)\quad,\quad e'\mapsto\partial_eW\,.
\]
\end{definition}
 
Our construction recovers these dg-algebras thanks to the following.
 
\begin{proposition}\label{critasginz}
If $D=\emptyset$, then 
\[
 \bcrit_\pi(\mathrm{tr}(W))=\big[\scat{DRep}(\mathcal{G}_3(kQ,\delta W),\vec{n})/\GL_{\vec{n}}(k)\big]\,.
\]
\end{proposition}

\begin{proof}
From~\S\ref{quivsch}, we know that 
\[
 \bcrit_\pi(\mathrm{tr}(W))=\big[\scat{DRep}(\mathcal A,\vec{n})/\GL_{\vec{n}}(k)\big]\,,
\]
where $\mathcal A$ is defined as the following homotopy push-out 
\[
\xymatrix{
\mathcal G_2(kQ)\ar[r]^{\delta1}\ar[d]_{\delta W}&kQ\ar[d]\\kQ\ar[r]&\mathcal A.}
\]
As pointed out in~\cite[Remark 6.4]{Keller}, this homotopy push-out defines $\mathcal G_3(kQ,\delta W)$. 
It can be seen by resolving $kQ$ with the quasi-free $\mathcal G_2(kQ)$-dg-algebra $\mathcal B(kQ)$ generated by 
$e'\in E(Q^*)$ in degree $-1$ and $x'\in V(Q)$ in degree $-2$
with differential
\[
\mathfrak d:\quad x'\mapsto x-\sum_{e\in E(Q)}(ee'-e'e),\quad e'\mapsto e^*-\partial_eW\,.
\]
One can then compute the desired push-out using $\mathcal B(kQ)$ and check that we get $\mathcal G_3(kQ,\delta W)$.
\end{proof}

When $D$ is non trivial, we define $\mathcal G_2(D\subseteq Q)$ as the homotopy push-out
\[
\xymatrix{
kD\ar[r]\ar[d]&kQ\ar[d]\\
\mathcal G_2(kD)\ar[r]&\mathcal G_2(D\subseteq Q)}
\]
so that
\[
\pi^*\bT^*B=\big[\scat{DRep}(\mathcal G_2(D\subseteq Q),\vec{n})/\GL_{\vec{n}}(k)\big]\,.
\]
 
As a consequence, keeping~\eqref{critreldiag} in mind, we get the following.
 
\begin{proposition}\label{defG3QDW}
 The dg-algebra $\mathcal G_3(kQ|kD,\delta W)$ defined by the homotopy push-out 
\[
\xymatrix{
\mathcal G_2(kQ)\ar[d]_{\delta W}\ar[r]&\mathcal G_2(D\subseteq Q)\ar[d]\\kQ\ar[r]&\mathcal G_3(kQ|k D,\delta W)}
\]
satisfies
\[
\bcrit_\pi(\mathrm{tr}(W))=\big[\scat{DRep}(\mathcal G_3(kQ|k D,\delta W),\vec{n})/\GL_{\vec{n}}(k)\big]\,.
\]
\end{proposition}
 
Using the same resolution as in the proof of Proposition~\ref{critasginz}, one can check that $\mathcal G_3(kQ|k D,\delta W)$ 
is generated by
\begin{itemize}
\item $e\in E(Q)$ and $e^*\in E(D^*)$ in degree $0$;
\item $e'\in E(Q^*)$ and $x_v$ a loop in degree $-1$ to every vertex $v\in V(D)$;
\item $x'_v$ a loop in degree $-2$ to every vertex $v\in V(Q)$
\end{itemize}
with differential
\[
x_v'\mapsto \left\{\begin{aligned} &x_v-\sum_{e\in E(Q)}e_v(ee'-e'e)e_v\text{ if }v\in V(D)\\
&-\sum_{e\in E(Q)}e_v(ee'-e'e)e_v\text{ otherwise,}\end{aligned}\right.\]
and\[
x_v\mapsto\sum_{e\in E(D)}e_v(ee^*-e^*e)e_v\,,\quad
e'\mapsto \left\{\begin{aligned} &e^*-\partial_eW\text{ if }e\in E(D)\\
&-\partial_eW\text{ otherwise.}\end{aligned}\right.
\]

\begin{remark} \label{remark: section56}
\begin{enumerate}
\item We retrieve our linear examples when taking the representation spaces of the $0$-th cohomology of 
$\mathcal G_3(kQ|k D,\delta W)$. 
\item We generalize these constructions in future sections. In section~\ref{sec:CY} we define dg-categories 
$\cG_{n+1}(\mathcal B|\cA,c)$ associated with any dg-functor $\cA\to\cB$. In section~\ref{sec:comparison} we investigate 
their moduli stacks of objects $\bPerf_{\cG_{n+1}(\cB|\cA,c)}$, and in the case of a quiver embedding $kD\subset kQ$, 
\[
\Big[\scat{DRep}(\mathcal G_3(kQ|k D,\delta W),\vec{n})/\GL_{\vec{n}}(k)\Big]\subset\bPerf_{\cG_{3}(kQ|kD,\delta W)}
\]
will be recovered as an open substack for the same reason as in Remark~\ref{labelresBG}.
\end{enumerate}
\end{remark}
 

\subsubsection{Lagrangian subvarieties from lagrangian morphisms}\label{lag2lagsec}

We want a direct proof of the lagrangian nature of the families of varieties studied in section~\ref{sec:motiv}, relying on 
the lagrangian structure of $\bcrit_\pi(\phi)\to\bT^*B$. By lagrangian subvariety, as in section~\ref{sec:motiv}, we mean 
isotropic of pure dimension half the ambient one. 
Thanks to the following proposition and Example~\ref{linex}(ii) we get that $\Lambda_n$ is lagrangian in $T^*{\Rep}_\mathbb C(S_1,n)$ 
as wished.

\begin{proposition}\label{lag2lag}
Consider a lagrangian morphism $g:L\to Y$, with $Y$ a genuine smooth and symplectic algebraic variety. Denote by $\tau_0(L)\to L$ 
the $0$-th truncation of $L$, and $Z\subset Y$ the smallest closed subscheme factorizing the following diagram
\[
\xymatrix{  \tau_0(L) \ar[d] \ar[r] &Z\ar[d] \\
L \ar[r]&Y }
\]
\textit{i.e.}\ $Z$ is the schematic image of $\tau_0(L)\to Y$. Then $Z\subset Y$ is a lagrangian subvariety.
\end{proposition}

\begin{proof}
The lagrangian structure yields the following fibre sequence\[
\bbT_L\to g^*T_Y\simeq g^*\Omega^1_Y\to\bbL_L\]
which in turn gives, since $T_Y$ is a complex concentrated in $0$, the following five terms sequence\[
0\to H^{-1}(\bbL_L)\to H^0(\bbT_L)\to g^*T_Y\to H^0(\bbL_L)\to H^1(\bbT_L)\to0\]
and then\begin{align*}
H^{-1}(\bbL_L)&\simeq\ker\big(H^0(\bbT_L)\to g^*T_Y\big)\\
H^1(\bbT_L)&\simeq\operatorname{coker}\big(g^*\Omega^1_Y\to H^0(\bbL_L)\big).\end{align*}
Now denote by $p$ the dominant morphism $\tau_0(L)\to Z$. Consider $\mathcal Z$ and $\mathcal L$ irreducible components of $Z$ and $\tau_0(L)$ respectively such that $p:\mathcal L\to\mathcal Z$ is still dominant. The two following commutative diagrams\[
\begin{tikzcd}T_\mathcal L\arrow[r]\arrow{d}[pos=0.1,rotate=-90,yshift=-0.8ex]{\sim}&p^*T_\mathcal Z\arrow[d,hook']\\H^0(\bbT_L)\arrow[r]&g^*T_Y\end{tikzcd}~~\text{ and }~~
\begin{tikzcd}g^*\Omega^1_Y\arrow[r]\arrow[d,two heads]&H^0(\bbL_L)\arrow{d}[pos=0.1,rotate=-90,yshift=0.8ex]{\sim}\\p^*\Omega^1_\mathcal Z\arrow[r]&\Omega^1_\mathcal L\end{tikzcd}\]
 imply\begin{align*}
H^{-1}(\bbL_L)&\simeq\ker\big(T_\mathcal L\to p^*T_\mathcal Z\big)\\
H^1(\bbT_L)&\simeq\operatorname{coker}\big(p^*\Omega^1_\mathcal Z\to \Omega^1_\mathcal L\big).\end{align*}
Now, by~\cite[Lemma 10.5]{Hartshorne}, $p$ is smooth on some non-empty open subscheme $\mathcal U\subseteq\mathcal L$, which make $T_\mathcal U\to p^*T_\mathcal Z$ and $p^*\Omega^1_\mathcal Z\to\Omega^1_\mathcal U$ respectively surjective and injective. Thus\begin{align*}
p^*T_{\mathcal Z}&\simeq\dfrac{T_\mathcal U}{H^{-1}(\bbL_L)}\simeq\dfrac{H^0(\bbT_L)}{H^{-1}(\bbL_L)}\\
p^*\Omega_{\mathcal Z}^1&\simeq\ker\big(\Omega^1_\mathcal U\to H^1(\bbT_L)\big)\simeq\ker\big(H^0(\bbL_L)\to H^1(\bbT_L)\big)\end{align*}
and\[
0\to p^*T_\mathcal Z\to g^*T_Y \to p^*\Omega^1_\mathcal Z\to0
\]is exact as expected.
\end{proof}

Unfortunately it does not directly apply to the example of~\S\ref{laghilb}. 
With the same notations as therein, we consider $S_1^+\subset S_3^+$, and 
\[
\pi :X=[{\Rep}_{\mathbb C}(S_3^+,(n,1))/\PGL_{n,1}(\mathbb C)]\rightarrow B=[{\Rep}_{\mathbb C}(S_1^+, (n,1))/\PGL_{n,1}(\mathbb C)]
\]
endowed with $\phi=\mathrm{tr}(W)$ where $W$ is the potential induced in Example~\ref{linex}(ii) through 
$S_3\subset S_3^+$. The issue is that
\[
\bT^*B=\big[\scat{DRep}\big(\mathcal{G}_{2}(\mathbb CS^+_1),{(n,1)}\big)/\PGL_{n,1}(\mathbb C)\big]
\]
is no longer concentrated in degree $0$.

Using the stability defined in~\S\ref{laghilb}, we consider the stable locus $Y=(\tau_0\bT^*B)^\text{st}\simeq(\mathbb C^2)^{[n]}$ 
of the $0$-th truncation of $\bT^*B$. We know that $Y$ is this time a smooth and symplectic variety. 
Set 
\[
L:=\bcrit_\pi(\phi)\underset{\bT^*B}{\times}Y
\]
and note that $L\to Y$ is lagrangian as the composition of $\bcrit_\pi(\phi)\to\bT^*B$ with the graph 
$Y\leftarrow {gr}(s) \rightarrow\bT^* B$ of the symplectic embedding $s:Y\hookrightarrow \bT^*B$, 
which is a lagrangian correspondence. Applying Proposition~\ref{lag2lag} to $L\to Y$, we get that the schematic 
image $\Lambda_{n,1}$ of $\tau_0L$ in $Y$ is a lagrangian subvariety of $Y=(\mathbb C^2)^{[n]}$, as obtained 
in~\S\ref{laghilb}.

%% file: CritQuiv-Section5.tex
\subsection{Hochschild and cyclic homology for dg categories} 

\subsubsection{Homotopy theory of dg-categories}

A \textit{dg-category} is a $\cat{Mod}_k$-enriched category. The category of dg-categories and enriched functors 
(also called dg-functors) is denoted $\cat{Cat}_k$. We refer to \cite{KellerDG,ToDG2} for a thorough introduction to 
dg-categories and their homotopy theory. 

With a dg-category $\cat{A}$ we can associate its category $\cat{Mod}_{\cat{A}}$ of right modules, which are dg-functors 
$\cat{A}^{\op}\to\cat{Mod}_k$. There is a $\cat{Mod}_k$-model structure on $\cat{Mod}_{\cat{A}}$. 

There is a model structure on $\cat{Cat}_k$ (see \cite{Tab,ToDG}), for which weak equivalences are quasi-equivalences. 
This model structure admits a left Bousfield localization, for which weak equivalences are Morita morphisms, 
i.e.~dg-functors $\cat{A}\to\cat{B}$ inducing an equivalence $\scat{Mod}_{\cat{B}}\to\scat{Mod}_{\cat{A}}$ (see \cite{Tab2}). 

On the level of $\infty$-categories, this defines a reflective localization $\scat{stCat}_k$ of $\scat{Cat}_k$, 
that can be identified with the full subcategory spanned by triangulated dg-categories, 
and that is a model for the $\infty$-category of small idempotent complete $k$-linear stable 
$\infty$-categories (see \cite{Cohn}). The localization $\infty$-functor is 
$\cat{A}\mapsto\cat{Mod}_{\cat{A}}^\mathrm{perf}$. 

\medskip

The model category $\cat{Cat}_k$ has a symmetric monoidal structure, which is closed, but not compatible with the 
model structure. Nevertheless, the tensor product of dg-categories preserves quasi-equivalences and thus descends to $\scat{Cat}_k$ 
(recall that $k$ is a field; if $k$ would more generally be a $\mathbb{Q}$-algebra, then one would have to consider $k$-flat dg-categories). 

\medskip

To\"en has proven (see \cite{ToDG}) that, given two dg-categories $\cat{A}$ and $\cat{B}$, there is a canonical equivalence 
\[
\scat{BiMod}_{\cat{A}-\cat{B}}
:=\scat{Mod}_{\cat{A}^{\op}\otimes\cat{B}}
\simeq\textsc{Map}^c_{\scat{Cat}_k}(\scat{Mod}_{\cat{A}},\scat{Mod}_{\cat{B}}\big)
\]
in $\scat{Cat}_k$, where $\textsc{Map}^c$ denotes the full sub-dg-category consisting of colimit preserving functors. 
This induces a tensor product functor
\[
-\underset{\cat{B}}{\overset{\mathbb{L}}{\otimes}}-\,:\,\scat{BiMod}_{\cat{A}-\cat{B}}\otimes\scat{BiMod}_{\cat{B}-\cat{C}}
\to \scat{BiMod}_{\cat{A}-\cat{C}}\,,
\]
given by the composition of functors. As usual, this tensor product functor can be computed on models, using the 
genuine tensor product $\underset{\cat{B}}{\otimes}$ of dg-bimodules, assuming that one of the two factors is cofibrant 
as a $\cat{B}$-module. 

\subsubsection{Hochschild and cyclic homology}

Recall the \textit{Hochschild homology} functor 
\[
\cat{HH}\,:\,\scat{Cat}_k\longrightarrow \scat{Mod}_k\,;\,
\cat{A}\longmapsto \cat{A}\underset{\cat{A}^e}{\overset{\mathbb{L}}{\otimes}}\cat{A}^{\op}\,,
\]
where $\cat{A}^e:=\cat{A}\otimes\cat{A}^{\op}$. 
Hochschild homology can be understood as a categorical trace of the identity functor $\scat{Mod}_{\cat{A}}\to\scat{Mod}_{\cat{A}}$, 
and as such it is Morita invariant. In particular, the natural Morita morphism $\cat{A}\to\cat{Mod}_{\cat{A}}^\mathrm{perf}$ induces an 
equivalence 
\[
\cat{HH}(\cat{A})\,\tilde\longrightarrow\, \cat{HH}(\cat{Mod}_{\cat{A}}^\mathrm{perf})\,.
\]
The fact that Hochschild homology is a categorical trace suggests that it lifts to a functor 
\[
\scat{HH}\,:\,\scat{Cat}_k\longrightarrow \scat{Mod}_k^\delta\,.
\]
This is indeed the case, and we are going to provide a very explicit approach to this. 
In order to compute Hochschild homology, one can use the (functorial) simplicial Bar resolution 
of $\cat{A}$ by free (right) $\cat{A}^e$-modules 
\[
\mathrm{B}(\cat{A}):=
\left(\cat{A}^{\underset{k\mathrm{Ob}(\cat{A})}{\otimes}(n+2)}\right)_{n\geq0}\,,
\]
where the face maps, resp.~degeneracy maps, are given by (opposite) composition morphisms, resp.~identity insertion morphisms. 
\begin{notation}
As a matter of notation, for a set $S$, $kS$ denotes the $k$-linear category having $S$ as objects and no non-trivial morphisms. 
The functor $k-$ is left adjoint to $\mathrm{Ob}:\cat{Cat}_k\to\cat{Sets}$. 
\end{notation}
This exhibits $\cat{HH}(\cat{A})$ as the colimit of the simplicial cochain complex
\[
\xymatrix{
\mathrm{B}(\cat{A})\underset{\cat{A}^e}{\otimes}\cat{A}^{\op}=
\cdots\underset{b_0,b_1\in\mathrm{Ob}(\cat{A})}{\coprod}\cat{A}(b_1,b_0)\otimes\cat{A}(b_0,b_1)
\ar@<0.5ex>[r]\ar@<-0.5ex>[r]
& \underset{c\in\mathrm{Ob}(\cat{A})}{\coprod} \cat{A}(c,c)\,.
}
\]
The above simplicial cochain complex happens to have an obvious cyclic structure, and thus its $\infty$-colimit, that gives 
back the \textit{standard Hochschild complex}, inherits a mixed structure, which is given by Connes's $B$-operator as we have 
already seen in Remark \ref{remHoyois} (we again refer to \cite{Hoyois} for more details). 

We borrow the Notation \ref{not:hhhc}: the negative cyclic complex of $\cat{A}$ $\cat{HC}^-(\cat{A})$ is the homotopy fixed 
$\delta$-points of $\scat{HH}(\cat{A})$. Notice that, since the forgetful functor $\scat{Mod}_k^\delta\to\scat{Mod}_k$ 
is conservative, the natural morphism 
\[
\scat{HH}(\cat{A})\longrightarrow \scat{HH}(\cat{Mod}_{\cat{A}}^\mathrm{perf})
\]
is an equivalence. 
\begin{remark}
The functor $\scat{HH}$ is $\mathbb{Z}/2\mathbb{Z}$-equivariant. More precisely 
\begin{itemize}
\item There is an involution on $\cat{Cat}_k$ given by $(-)^{\op}$; 
\item There is an involution on $\cat{Mod}_k^\delta$ given by the involution $\delta\mapsto-\delta$ of $k[\delta]$; 
\item The standard Hochschild complexes of $\cat{A}$ and $\cat{A}^{\op}$ are canonically isomorphic (by reversing the order 
of factors and multiplying by $(-1)^n$ at each level); 
\item Under the above identification, the $B$-operators of Connes match up to a sign.  
\end{itemize}
Therefore, $\cat{HH}$ and $\cat{HC}^-$ are invariant with respect to $(-)^{\op}$. 

This has a nice interpretation in terms of the orientation reversing involution of $S^1$, but we will not elaborate on that. 
\end{remark}

\subsubsection{Hochschild homology of a free category}\label{ssec:hhfree}

Given a dg-category $\cat{A}$, the forgetful functor $(\cat{Cat}_k)_{\cat{A}/} \to \cat{Mod}_{\cat{A}^e}$, 
sending a dg-functor $f:\cat{A}\to\cat{B}$ to the $\cat{A}$-bimodule given by $(b,b')\mapsto \cat{B}\big(f(b'),f(b)\big)$, 
has a left adjoint $T_{\cat{A}}$, that can be described as follows: given a $\cat{A}$-bimodule $\cat{M}$ 
\[
T_{\cat{A}}(\cat{M}):= \cat{A} \oplus \cat{M} \oplus (\cat{M} \underset{\cat{A}}{\otimes} \cat{M}) \oplus \ldots
\]
is the dg-category freely generated over by $\cat{M}$ over $\cat{A}$. 
This adjunction is Quillen, and thus induces an adjunction between $(\scat{Cat}_k)_{\cat{A}/}$ and $\scat{Mod}_{\cat{A}^e}$. 

\medskip

Let now $\cat{B}:=\mathbb{L}T_{\cat{A}}(\cat{M})$. As a $\cat{B}$-bimodule, 
\[
\xymatrix{
\cat{B}\simeq\mathrm{coeq}\left(
\cat{B}\underset{\cat{A}}{\overset{\mathbb{L}}{\otimes}}\cat{M}\underset{\cat{A}}{\overset{\mathbb{L}}{\otimes}}\cat{B}\right.
\ar@<0.5ex>[r]^-{\circ\otimes\mathrm{id}}\ar@<-0.5ex>[r]_-{\mathrm{id}\otimes \circ}
& \left. \cat{B}\underset{\cat{A}}{\overset{\mathbb{L}}{\otimes}}\cat{B}\right)\,,
}
\]
where we call the above r.h.s.~a \textit{small resolution} of $\cat{B}$, generalizing the one from~\cite[Remark 3.1.3]{Loday}. 
Indeed, given a cofibrant replacement $\tilde{\cat{M}}$ of the $\cat{A}$-bimodule $\cat{M}$, then $\cat{B}=T_{\cat{A}}(\tilde{\cat{M}})$ 
and the sequence of $\cat{B}$-bimodules 
\[
0\longrightarrow \cat{B}\underset{\cat{A}}{\otimes}\tilde{\cat{M}}\underset{\cat{A}}{\otimes}\cat{B}
\overset{\substack{\circ\otimes\mathrm{id} \\ -\mathrm{id}\otimes \circ}}{\longrightarrow} 
 \cat{B}\underset{\cat{A}}{\otimes}\cat{B}\longrightarrow\cat{B}\longrightarrow 0\,,
\]
is exact (see e.g.~\cite[proof of Theorem 4.8]{Keller} or~\cite[proof of Theorem 1.1]{KellerE}). 
Observe that the diagram $\mathrm{R}(\cat{B})$ 
in the r.h.s.~can actually be written in a little more compact way: 
\[
\xymatrix{
\mathrm{R}(\cat{B})=\left(\cat{M}\underset{\cat{A}^e}{\overset{\mathbb{L}}{\otimes}}\cat{B}^e\right.
\ar@<0.5ex>[r]\ar@<-0.5ex>[r]
& \left.\cat{A}\underset{\cat{A}^e}{\overset{\mathbb{L}}{\otimes}}\cat{B}^e\right)\,,
}
\]
so that applying $-\underset{\cat{B}^e}{\overset{\mathbb{L}}{\otimes}}\cat{B}$, we get 
\[
\xymatrix{
\cat{HH}(\cat{B})\simeq\mathrm{coeq}\left(
\cat{M}\underset{\cat{A}^e}{\overset{\mathbb{L}}{\otimes}}\cat{B}\right.
\ar@<0.5ex>[r]\ar@<-0.5ex>[r]
& \left. \cat{A}\underset{\cat{A}^e}{\overset{\mathbb{L}}{\otimes}}\cat{B}\right)\,.
}
\]
\begin{example}[Free categories]
We say that a category is free whenever it is freely generated over its category of objects. 
This corresponds to the case $\cat{A}=kS$, with $S$ a set. 
If $\cat{B}=T_{kS}(\cat{M})$ then the above coequalizer becomes 
\[
\xymatrix{
\cat{HH}(\cat{B})\simeq\mathrm{coeq}\left(
\underset{s_1,s_0\in S}{\coprod}\cat{M}(s_1,s_0)\otimes\cat{B}(s_0,s_1)\right.
\ar@<0.5ex>[r]^-{\circ}\ar@<-0.5ex>[r]_-{\circ^{\op}}
& \left. \underset{t\in S}{\coprod}\cat{B}(t,t)\right)\,.
}
\]
The above homotopy coequalizer can be computed as a c\^one. We call it the \textit{small Hochschild complex} 
of a free category, and it naturally maps quasi-isomorphically into the standard Hochschild complex (as we have 
a map of diagrams $\mathrm{R}(\cat{B})\to\mathrm{B}(\cat{B})$). 
\end{example}









\begin{remark}\label{setquivfree}
For a finite set $S$, there is a Morita morphism from the category $kS$ to the algebra $k^S$ (which we view as a 
$k$-linear category with only one object). Actually, even their $1$-categories of (bi)modules are equivalent. 
Given such a bimodule $\cat{M}$, that is a collection $(\cat{M}(s,t))_{s,t\in S}$ of 
cochain complexes, we also have a Morita morphism $T_{kS}(\cat{M})\to T_{k^S}(\cat{M})$. 
As all our constructions are Morita invariant, we will tend not to distinguish them. 

Observe that, for a quiver $Q$, with vertex set $S=V(Q)$ and edge set $E(Q)$, the linear span $\cat{M}:=kE(Q)$ 
becomes a $kS$-bimodule such that $T_{k^S}(\cat{M})=kQ$. 
\end{remark}

\subsection{Calabi--Yau structures, absolute and relative}\label{subscy}

\subsubsection{Calabi--Yau structures on dg-categories}

Let $\cat{A}$ be a dg-category. Observe that the symmetry $\cat{A}\otimes\cat{A}^{\op}\to \cat{A}^{\op}\otimes\cat{A}$ 
defines an isomorphism $\tau:\cat{A}^e\tilde\longrightarrow(\cat{A}^e)^{\op}$, which is an anti-involution: 
$\tau^{\op}\tau=\mathrm{id}$. This allows one to define a preduality functor 
\[
(-)^\vee\,:\,\scat{Mod}_{\cat{A}^e}\longrightarrow \scat{Mod}_{\cat{A}^e}^{\op}
\]
given as follows: for a right $\cat{A}^e$-module $\cat{M}$, and an object $a\in\mathrm{Ob}(\cat{A}^{\op}\otimes\cat{A})$, 
\[
\cat{M}^\vee(a):=\ul{\cat{Map}}_{\scat{Mod}_{(\cat{A}^e)^{\op}}}\big(\cat{M}\circ\tau,\cat{A}^e(a,-)\big)
\simeq\mathbb{R}\ul{\cat{Hom}}_{\cat{Mod}_{(\cat{A}^e)^{\op}}}\big(\cat{M}\circ\tau,\cat{A}^e(a,-)\big)\,.
\]
Notice that $\cat{A}\circ\tau=\cat{A}^{\op}$. 
The preduality functor restricts to an anti-equivalence on perfect modules. 

A dg-category $\cat{A}$ is \emph{smooth} if $\cat{A}$ is a perfect $\cat{A}^e$-module. In this case, by duality, 
we have an equivalence 
\[
(-)^\flat\,:\,\cat{HH}(\cat{A})\tilde\longrightarrow
\ul{\cat{Map}}_{\scat{Mod}_{\cat{A}^e}}(\cat{A}^\vee,\cat{A})
\simeq\mathbb{R}\ul{\cat{Hom}}_{\cat{Mod}_{\cat{A}^e}}(\cat{A}^\vee,\cat{A})\,.
\]

\begin{definition}\label{def:CY}
Let $\cat{A}$ be a dg-category. 
\begin{enumerate}
\item A \emph{$n$-pre-Calabi--Yau structure} on $\cat{A}$ is a class $c:k[n]\to\scat{HH}(\cat{A})$ (or, equivalenty, 
a class $k[n]\to \cat{HC}^-(\cat{A})$). 
\item An \emph{almost $n$-Calabi-Yau structure} on $\cat{A}$ is a class $c:k[n]\to\cat{HH}(\cat{A})$ that is \textit{non-degenerate}, 
meaning that $\cat{A}$ is smooth and $c^\flat:\cat{A}^\vee[n]\to \cat{A}$ is an equivalence. 
\item A \emph{$n$-Calabi-Yau structure} on $\cat{A}$ is a $n$-pre-Calabi--Yau structure $c$ such that $c^\natural$ 
is an almost $n$-Calabi--Yau structure. 
\item An exact \emph{$n$-(pre-)Calabi--Yau structure} on $\cat{A}$ is a class $c:k[n-1]\to\cat{HH}(\cat{A})$ such that its image 
$\delta c$ through \eqref{eq:Connes} is a $n$-(pre-)Calabi--Yau structure. 
\end{enumerate}
\end{definition}
\begin{remark}
We warn the reader that our notion of pre-Calabi--Yau structure drastically differs from the one of several authors 
(see e.g.~\cite{Ye}). Ours is a non-commutative version of a pre-symplectic structure while in \cite{Ye} it is rather 
a non-commutative version of a Poisson structure. Our terminology is nevertheless consistent with To\"en's one~\cite{ToCY}. 
\end{remark}
\begin{example}\label{example: loop}
Let $\cA=k[x]$ the polynomial ring, which we view as a dg-category with one object. 
It has a canonical exact $1$-Calabi--Yau structure, that we describe now. 
As it is a commutative algebra, every element of $k[x]$ defines a degree $0$ Hochschild class. 
Let us consider the one given by $x\in k[x]$, and prove that $\delta(x)=1\otimes x$ is indeed 
non-degenerate in the sense of definition \ref{def:CY}(2). 
As $k[x]$ is free (it is freely generated by the Jordan quiver $L$), according to \S\ref{ssec:hhfree} 
we have a small resolution of it as a bimodule, given by 
\[
\mathrm{R}(k[x])=\mathrm{cofib}\Big(k[x]\otimes kx\otimes k[x]\overset{p\otimes x\otimes q\mapsto px\otimes q-p\otimes xq}
{\xrightarrow{\hspace*{3cm}}}k[x]\otimes k[x]\Big)\,.
\]
Therefore we have 
\[
k[x]^\vee[1]\simeq \mathrm{cofib}\Big(k[x]\otimes k[x] 
\overset{x^n\otimes x^m\mapsto x^{n+1}\otimes x^*\otimes x^m-x^n\otimes x^*\otimes x^{m+1}}{\xrightarrow{\hspace*{3cm}}}
k[x]\otimes kx^*\otimes k[x]\Big)\,.
\]
Viewed as a morphism $k[x]^\vee[1]\to \mathrm{R}(k[x])$, the Hochschild class $1\otimes x$ can thus 
be described as follows: 
\begin{itemize}
\item In degree $-1$, it sends $p\otimes q$ to $p\otimes x\otimes q$;
\item In degree $0$, it sends $p\otimes x^*\otimes q$ to $p\otimes q$. 
\end{itemize}
It is non-degenerate, as it is an isomorphism on models. 
\end{example}

The above example is a special case of a large class of examples of Calabi--Yau dg-categories: 
the so-called \emph{Calabi--Yau completions}, introduced in \cite{Keller}. 
\begin{definition}[\cite{Keller}]
Let $\cat{A}$ be a dg-category. We call $\cG_{n}(\cat{A}):=\mathbb{L}T_{\cat{A}}(\cat{A}^\vee[n-1])$ the $n$-Calabi--Yau 
completion of $\cat{A}$. 
\end{definition}
From now on, we will always give ourselves a cofibrant model for $\cat{A}^\vee$, and thus we can safely drop the 
$\mathbb{L}$ from the notation. 
Thanks to a result of Keller (see \cite[Theorem 1.1]{KellerE}), Calabi--Yau completions of smooth dg-categories are 
exact Calabi--Yau:
\begin{theorem}\label{thm:CYcom}
If $\cat{A}$ is smooth (resp.~of finite type), then its $n$-Calabi--Yau completion $\cG_{n}(\cat{A})$ is smooth 
(resp.~of finite type) and carries an exact $n$-Calabi--Yau structure. 
\end{theorem}
\begin{proof}
We first recall the construction of the class $c:k[n-1]\to\cat{HH}(\cat{B})$, 
where $\cat{B}:=\cG_{n}(\cat{A})$. We have a morphism of $\cB^e$-bimodules 
$\cat{A}^\vee[n-1]\overset{\mathbb{L}}{\underset{\cat{A}^e}{\otimes}}\cat{B}^e\to \cat{B}$, along 
with the split morphism of dg-categories $\cat{A}\to\cat{B}$, which thus induces a morphism 
\begin{equation}\label{exact nCY}
\cat{B}^\vee[n-1]\to \cat{A}^\vee[n-1]\overset{\mathbb{L}}{\underset{\cat{A}^e}{\otimes}}\cat{B}^e\to \cat{B}\,.
\end{equation}
In other words, it is the image of $\mathrm{id}_{\cA^\vee[n-1]}$ through
\[
\ul{\cat{Map}}_{\scat{Mod}_{\cat{A}^e}}(\cat{A}^\vee,\cat{A}^\vee)\to 
\ul{\cat{Map}}_{\scat{Mod}_{\cat{B}^e}}(\cat{B}^\vee[n-1],\cat{B})\simeq\cat{HH}(\cB)[1-n]\,.
\]
Another way to understand $c$ is as follows: it is the image of $\mathrm{id}_{\cat{A}}$ through 
\begin{equation}\label{eq:image of id}
\ul{\cat{Map}}_{\scat{Mod}_{\cat{A}^e}}(\cat{A},\cat{A})[n-1]
\overset{\sim}{\to} (\cat{A}^\vee[n-1])\overset{\mathbb{L}}{\underset{\cat{A}^e}{\otimes}}\cat{A}^{\op}
\to \cat{B}\overset{\mathbb{L}}{\underset{\cat{B}^e}{\otimes}}\cat{B}^{\op}=\cat{HH}(\cat  B)\,.
\end{equation}
We now prove that the underlying Hochschild class $\gamma:=(\delta c)^\natural$ of $\delta c$ is non-degenerate. 
Using small resolutions (as in the proof of~\cite[Theorem 4.8]{Keller}) 
we have
\begin{equation}\label{coeq-for-B}
\cB\simeq\mathrm{coeq}\left(\xymatrix{\cA^\vee[n-1]\overset{\mathbb{L}}{\underset{\cat{A}^e}{\otimes}}\cB^e
\ar@<0.5ex>[r]\ar@<-0.5ex>[r]
& \cA\overset{\mathbb{L}}{\underset{\cat{A}^e}{\otimes}}\cB^e}\right),
\end{equation}
so that 
\begin{align*}
\cat{HH}(\cat{B})& \simeq\mathrm{coeq}\left(\xymatrix{\cA^\vee[n-1]\overset{\mathbb{L}}{\underset{\cat{A}^e}{\otimes}}\cB^{\op}
\ar@<0.5ex>[r]\ar@<-0.5ex>[r]
& \cA\overset{\mathbb{L}}{\underset{\cat{A}^e}{\otimes}}\cB^{\op}}\right) \\
& \simeq \mathrm{eq}\left(\xymatrix{\cA^\vee[n]\overset{\mathbb{L}}{\underset{\cat{A}^e}{\otimes}}\cB^{\op}
\ar@<0.5ex>[r]\ar@<-0.5ex>[r]
& \cA[1]\overset{\mathbb{L}}{\underset{\cat{A}^e}{\otimes}}\cB^{\op}}\right)
\end{align*}
naturally receives a map from $\cA^\vee[n]\overset{\mathbb{L}}{\underset{\cat{A}^e}{\otimes}}\cat{A}^{\op}$. 
Moreover, Keller shows in \cite{KellerE} that $\gamma$ is the image of $\mathrm{id}_{\cat{A}}$ through 
\[
\ul{\cat{Map}}_{\scat{Mod}_{\cat{A}^e}}(\cat{A},\cat{A})[n]
\overset{\sim}{\to} (\cat{A}^\vee[n])\overset{\mathbb{L}}{\underset{\cat{A}^e}{\otimes}}\cat{A}^{\op}
\to \cat{HH}(\cat{B}). 
\]
Dually, we also have 
\begin{align}
\cB^\vee&\simeq\mathrm{eq}\left(\xymatrix{\cA^\vee\overset{\mathbb{L}}{\underset{\cat{A}^e}{\otimes}}\cB^e
\ar@<0.5ex>[r]\ar@<-0.5ex>[r]
& \cA[1-n]\overset{\mathbb{L}}{\underset{\cat{A}^e}{\otimes}}\cB^e}\right) \nonumber\\
&\simeq\mathrm{coeq}\left(\xymatrix{\cA^\vee[-1]\overset{\mathbb{L}}{\underset{\cat{A}^e}{\otimes}}\cB^e
\ar@<0.5ex>[r]\ar@<-0.5ex>[r]
& \cA[-n]\overset{\mathbb{L}}{\underset{\cat{A}^e}{\otimes}}\cB^e}\right), \label{coeq-for-Bdual}
\end{align}
and thus the map $\gamma^\flat:\cat{B}^\vee[n]\to\cat{B}$ is induced by a morphism of coequalizers that is 
the identity on each component (this in particular shows that the diagrams appearing in \eqref{coeq-for-B} and \eqref{coeq-for-Bdual} are actually equivalent, 
which was \emph{a priori} not obvious). 
\end{proof}

\begin{example}
The Calabi--Yau completion $\mathcal G_n(k)$ coincides with the free algebra $k[x]$, where $x$ is a degree $1-n$ variable. 
Following Example \ref{example: loop}, one sees that the exact $n$-Calabi--Yau structure is given by $x$. 
\end{example}

\subsubsection{Calabi--Yau structures on cospans and functors}

We recall the notion of Calabi--Yau structures on cospans of dg-categories, following Brav--Dyckerhoff (see \cite{BD1}), 
who were themselves insipred by an earlier definition of To\"en (see \cite[\S5.3]{ToEMS}). Note that the notion we consider 
is called \textit{relative left Calabi--Yau} in \cite{BD1}, while we only write \textit{Calabi--Yau}, as right Calabi--Yau 
structures do not appear in this paper and the adjective ``relative'' can be unambiguously removed. 
\begin{definition}
Let $\cat{A}\overset{f}{\longrightarrow}\cat{C}\overset{g}{\longleftarrow}\cat{B}$ be a cospan of dg-categories. 
\begin{enumerate}
\item An $n$-pre-Calabi--Yau structure, resp.~exact $n$-pre-Calabi--Yau structure, on the above cospan is a homotopy 
commuting diagram 
\[
\xymatrix{
k[n] \ar[r]^{c_{\cat{B}}}\ar[d]_{c_{\cat{A}}}  & \cat{HC}^-(\cat{B}) \ar[d] \\
\cat{HC}^-(\cat{A}) \ar[r] & \cat{HC}^-(\cat{C})
}
\quad\textrm{resp.}\quad
\xymatrix{
k[n] \ar[r]\ar[d]  & \cat{HH}(\cat{B})[1] \ar[d] \\
\cat{HH}(\cat{A})[1] \ar[r]& \cat{HH}(\cat{C})[1]
}
\]
\item If $\cat{A}=\emptyset$, resp.~$\cat{B}=\emptyset$, then call these (exact) $n$-pre-Calabi--Yau structures on the 
morphism $g$, resp.~$f$. 
\end{enumerate}
\end{definition}
As for isotropic correspondences, $n$-pre-Calabi--Yau cospans compose, and so do their exact couterparts. 
More precisely, $n$-pre-Calabi--Yau cospans are cospans in the fiber product $\infty$-category 
\[
(\scat{Cat}_k)_{k[n]/\cat{HC}^-}:={\scat{Cat}_k}\!\!\underset{~^{\cat{HC}^-}{\scat{Mod}_k}^{k[n]}}{\times}{*}\,,
\]
and as this category admits finite colimits, we can compose cospans. 
Actually, according to \cite[\S5]{Haug}, we have an $\infty$-category 
\[
\scat{PrCY}_{[n]}:=\scat{Span}_1\left((\scat{Cat}_k)_{k[n]/\cat{HC}^-}^{\op}\right)\,.
\]
Similarly, there is an $\infty$-category 
\[
\scat{ExPrCY}_{[n]}:=\scat{Span}_1\left((\scat{Cat}_k)_{k[n]/\cat{HH}[1]}^{\op}\right)\,,
\]
and the natural transfomation $\delta:\cat{HH}[1]\to\cat{HC}^-$ given by \eqref{eq:Connes} defines an $\infty$-functor 
\[
\mathbf{\Delta}\,:\,\scat{ExPrCY}_{[n]}\longrightarrow \scat{PrCY}_{[n]}\,.
\]
\begin{remark}
Recall that any category of cospans is equipped with an involution $\sigma$, which is the identity on objects, and sends a cospan 
$\cA\rightarrow \cC\leftarrow \cB$ to the ``opposite'' cospan $\cB\rightarrow \cC\leftarrow \cA$. 
\end{remark}
\begin{definition}\label{definition: CYspan}
\begin{enumerate}
\item An $n$-pre-Calabi--Yau structure on a cospan $\cat{A}\overset{f}{\longrightarrow}\cat{C}\overset{g}{\longleftarrow}\cat{B}$ 
as in the previous definition is called \textit{$n$-Calabi--Yau} if the following conditions are satisfied: 
\begin{itemize}
\item All three categories $\cat{A}$, $\cat{B}$ and $\cat{C}$ are smooth. 
\item The Hochschild classes $c_{\cat{A}}^\natural$ and $c_{\cat{B}}^\natural$ are non-degenerate in the sense of 
Definition \ref{def:CY}(2). 
\item The square 
\[
\xymatrix{
\cat{C}^\vee[n]\ar[r]^-{g^\vee}\ar[d]_-{f^\vee} 
& (\cat{B}^\vee[n])\overset{\mathbb{L}}{\underset{\cat{B}^e}{\otimes}}\cat{C}^e 
\simeq \cat{B}\overset{\mathbb{L}}{\underset{\cat{B}^e}{\otimes}}\cat{C}^e \ar[d]^-{g\otimes\mathrm{id}} \\
(\cat{A}^\vee[n])\overset{\mathbb{L}}{\underset{\cat{A}^e}{\otimes}}\cat{C}^e 
\simeq \cat{A}\overset{\mathbb{L}}{\underset{\cat{A}^e}{\otimes}}\cat{C}^e \ar[r]^-{f\otimes\mathrm{id}}
& \cat{C}
}
\]
is (co)cartesian. This condition is also referred to as a \textit{non-degeneracy condition}. 
\end{itemize}
\item An exact $n$-pre-Calabi--Yau structure on a cospan is called exact $n$-Calabi--Yau if and only if its image through 
$\delta$ is $n$-Calabi--Yau. 
\item As usual, we adopt the same terminology for morphisms if $\cat{A}$ or $\cat{B}$ is $\emptyset$. 
\end{enumerate}
\end{definition}
We finally recall that $n$-Calabi--Yau cospans do compose well: indeed, after \cite[Theorem 6.2]{BD1}, 
the non-degeneracy property is preserved under the composition of $n$-pre-Calabi--Yau cospans. 
We therefore have a subcategory $\scat{CY}_{[n]}$ of $\scat{PrCY}_{[n]}$ whose objects are $n$-Calabi--Yau 
dg-categories and whose $1$-morphisms are $n$-Calabi--Yau cospans. 
Similarly, there is a subcategory $\scat{ExCY}_{[n]}$ of $\scat{ExPrCY}_{[n]}$ whose objects and morphisms are 
those that are sent to $\scat{CY}_{[n]}$ under $\mathbf{\Delta}$. Finally notice that the non-degeneracy condition 
is preserved under the involution $\sigma$ that takes cospans to their opposites. 
\begin{example}
Clearly, (exact) $n$-Calabi--Yau structures on the initial functor $\emptyset \to \cat{C} $ are equivalent to (exact) 
$(n+1)$-Calabi--Yau structures on $\cat{C}$. As a consequence, the push-out of two (exact) $n$-Calabi--Yau morphisms 
is (exact) $(n+1)$-Calabi--Yau. 
\end{example}

\medskip

A source of examples of Calabi--Yau morphisms are given by $n$-Calabi--Yau completion. 
\begin{proposition} \label{proposition: leftCY}
Let $\cat{B}:=\mathcal{G}_n(\cat{A})=T_{\cat{A}}(\cat{A}^\vee[n-1])$ denote the $n$-Calabi--Yau completion of 
a smooth dg-category $\cat{A}$, together with its exact $n$-Calabi--Yau structure $c$ from Theorem \ref{thm:CYcom}, 
and let $f:\cat{A}^\vee[n-1] \to \cat{A}$ a morphism of right $\cat{A}^e$-modules.  
\begin{enumerate}
\item The space of lifts of $\delta c$ to a $n$-pre-Calabi--Yau structures on the functor 
\[
T_{\cat{A}}(f):\cat{B} \longrightarrow \cat{A}\,,
\]
is the space of negative cyclic lifts of $f:k[n-1]\to\cat{HH}(\cat{A})$. 
\item All such $n$-pre-Calabi--Yau structures on $T_{\cat{A}}(f)$ are non-degenerate (i.e.~are $n$-Calabi--Yau). 
\end{enumerate}
\end{proposition}
\begin{proof}
As the exact $n$-Calabi--Yau structure $c$ on $\mathcal G_n(\cat{A})$ is the image of $\mathrm{id}_{\cat{A}}$ through 
\eqref{eq:image of id}, one can see that 
\[
\cat{HH}(T_{\cat{A}}(f))[1-n]: \cat{HH}(\cat{B})[1-n] \to \cat{HH}(\cat{A})[1-n]
\]
maps $c$ to $f$ seen as an element of $|\cat{HH}(\cat{A})[1-n]|$ via
\[
\ul{\cat{Map}}_{\scat{Mod}_{\cat{A}^e}}(\cat{A}^\vee[n-1],\cat{A})\overset{\sim}{\to} |\cat{HH}(\cat{A})[1-n]|\,.
\]
This proves (1). We now turn to the proof of non-degeneracy. 
We want to compute $\cB^\vee[n]\otimes^\bbL_{\cB^e}\cA^e\simeq \cB\otimes^\bbL_{\cB^e}\cA^e$. 
Recall from the proof of Theorem \ref{thm:CYcom} that we have
\begin{align*}
\cB&\simeq\mathrm{coeq}\left(\xymatrix{\cA^\vee[n-1]\overset{\mathbb{L}}{\underset{\cat{A}^e}{\otimes}}\cB^e
\ar@<0.5ex>[r]\ar@<-0.5ex>[r]
& \cA\overset{\mathbb{L}}{\underset{\cat{A}^e}{\otimes}}\cB^e}\right), \\
\cB^\vee&\simeq\mathrm{eq}\left(\xymatrix{\cA^\vee\overset{\mathbb{L}}{\underset{\cat{A}^e}{\otimes}}\cB^e
\ar@<0.5ex>[r]\ar@<-0.5ex>[r]
& \cA[1-n]\overset{\mathbb{L}}{\underset{\cat{A}^e}{\otimes}}\cB^e}\right)\\
&\simeq\mathrm{coeq}\left(\xymatrix{\cA^\vee[-1]\overset{\mathbb{L}}{\underset{\cat{A}^e}{\otimes}}\cB^e
\ar@<0.5ex>[r]\ar@<-0.5ex>[r]
& \cA[-n]\overset{\mathbb{L}}{\underset{\cat{A}^e}{\otimes}}\cB^e}\right),
\end{align*}
and that $\gamma^\flat:\cB^\vee[n]\overset{\sim}{\to}\cB$, where $\gamma:=(\delta c)^\natural$, is just given by the identity of the coequalizer of \[
\xymatrix{\cA^\vee[n-1]\overset{\mathbb{L}}{\underset{\cat{A}^e}{\otimes}}\cB^e
\ar@<0.5ex>[r]\ar@<-0.5ex>[r]
& \cA\overset{\mathbb{L}}{\underset{\cat{A}^e}{\otimes}}\cB^e.}\]
Now, using $T_\cA(f)$, we apply $-\otimes_{\cB^e}\cA^e$  and get \[
\cB\overset{\bbL}{\underset{\cB^e}{\otimes}}\cA^e\simeq\mathrm{coeq}\left(\xymatrix{\cA^\vee[n-1]
\ar@<0.5ex>[r]^{~~~~0}\ar@<-0.5ex>[r]_{~~~~0}
& \cA}\right)\simeq\cA^\vee[n]\oplus\cA.\]
over which $T_\cA(f)\otimes\mathrm{id}$ acts as the projection on $\cA$. The result follows from the fact that 
\[
\xymatrix{ 
\cA^\vee[n] \ar[d]_{ \mathrm{id}\oplus 0 }  \ar[r] & 0 \ar[d]    \\
 \cA^\vee[n]\oplus\cA  \ar[r]_-{0\oplus \mathrm{id} }& \cA} 
\]
is (co)cartesian. 
\end{proof}

\begin{definition}[Keller~\cite{Keller}]\label{def: deformed}
Let $\cat{A}$ be a smooth dg-category together with a class $c:k[n-1]\to\cat{HH}(\cat{A})$. Denote by 
$\mathcal G_{n+1}(\cat{A},c)$ the dg-category obtained by deforming $\mathcal{G}_{n+1}(\cat{A})$ with 
(the extension by the Leibniz rule of) $c[1]:\cat{A}^\vee[n]\to\cat{A}[1]$. 
In other words, $\mathcal{G}_{n+1}(\cat{A},c)$ is given as the homotopy pushout
\[
\xymatrix{
\cG_{n}(\cA) \ar[r]^{T_{\cat{A}}(c)} 	\ar[d]	_{T_{\cat{A}}(0)}& \cA \ar[d] \\
\cA														\ar[r]			& \cG_{n+1}(\cA,c).
					}
\]
\end{definition}

\begin{corollary}
Let $\cat{A}$ be a smooth dg-category together with a class $c:k[n-1]\to\cat{HH}(\cat{A})$. 
Then every negative cyclic lift of $c$ induces an $(n+1)$-Calabi--Yau structure on $\mathcal G_{n+1}(\cat{A},c)$.
\end{corollary}
\begin{proof}
On the one hand, it follows from Proposition~\ref{proposition: leftCY} that every negative cyclic lift of 
$c$ provide a $n$-Calabi--Yau structure on $T_{\cat{A}}(c)$. On the other hand, $0$ trivially lifts, therefore 
$T_{\cat{A}}(0)$ has a canonical $n$-Calabi--Yau structure. As a consequence, we obtained an $(n+1)$-Calabi--Yau 
structure on their homotopy push-out. 
\end{proof}
\begin{remark}
This corollary is a generalization, for smooth dg-categories, of a result of Yeung for finite cellular dg-categories 
(see \cite[Theorem 3.17]{Ye0}). Added in proof: in a new and completely revised version of \cite{Ye0}, Yeung 
was also able to get rid of the cellular hypothesis. 
\end{remark}

\subsubsection{Calabi--Yau completion as a functor with values in exact Calabi--Yau cospans}

Let us assume we are given a morphism $f:\cat{A}\to\cat{B}$ of smooth dg-categories. 
We define 
\[
\mathcal{G}_n(f):=T_{\cat{B}}\left(\cat{A}^\vee[n-1]\overset{\mathbb{L}}{\underset{\cat{A}^e}{\otimes}}\cat{B}^e\right)\,.
\]
\begin{theorem}\label{thm:Gn-functoriality}
There is an exact $n$-Calabi--Yau structure on the cospan 
\[
\mathcal{G}_n(A)\rightarrow\mathcal{G}_n(f)\leftarrow\mathcal{G}_n(\cat{B})\,.
\]
\end{theorem}
\begin{proof}
We already have exact $n$-Calabi--Yau structures on $\mathcal G_n(\cat{A})$ and $\mathcal G_n(\cat{B})$, 
respectively denoted $c_{\cat{A}}$ and $c_{\cat{B}}$. The fact that their images in $\cat{HH}(\mathcal G_n(f))$ 
coincide follows from the commuting diagram 
\[
\xymatrix{
& \ul{\cat{Map}}_{\scat{Mod}_{\cat{B}^e}}(\cat{B},\cat{B})[n-1] \ar[r]^-{\sim} \ar[d]_{\circ f}
& (\cat{B}^\vee[n-1])\overset{\mathbb{L}}{\underset{\cat{B}^e}{\otimes}}\cat{B}^{\op} \ar[r] \ar[d]^{f^\vee\otimes\mathrm{id}}
& \mathcal G_n(\cat{B})\overset{\mathbb{L}}{\underset{\mathcal G_n(\cat{B})^e}{\otimes}}\mathcal G_n(\cat{B})^{\op} \ar[d]\\
  k[n-1] \ar[r]^-{f}\ar[rd]_-{\mathrm{id}_{\cat{A}}}\ar[ur]^-{\mathrm{id}_{\cat{B}}}
& \ul{\cat{Map}}_{\scat{Mod}_{\cat{A}^e}}(\cat{A},\cat{B})[n-1] \ar[r]^-{\sim}
& (\cat{A}^\vee[n-1])\overset{\mathbb{L}}{\underset{\cat{A}^e}{\otimes}}\cat{B}^{\op} \ar[r]
& \mathcal G_n(f)\overset{\mathbb{L}}{\underset{\mathcal G_n(f)^e}{\otimes}}\mathcal G_n(f)^{\op}\\
& \ul{\cat{Map}}_{\scat{Mod}_{\cat{A}^e}}(\cat{A},\cat{A})[n-1] \ar[r]^-{\sim} \ar[u]^{f\circ}
& (\cat{A}^\vee[n-1])\overset{\mathbb{L}}{\underset{\cat{A}^e}{\otimes}}\cat{A}^{\op} \ar[r] \ar[u]_{\mathrm{id}\otimes f}
& \mathcal G_n(\cat{A})\overset{\mathbb{L}}{\underset{\mathcal G_n(\cat{A})^e}{\otimes}}\mathcal G_n(\cat{A})^{\op} \ar[u]
}
\]
This exhibits a $n$-pre-Calabi-Yau structure on our cospan. 
Let us show the non-degeneracy property. We use small resolutions again. 
First
\begin{align*}
\cG_n(f)&\simeq\mathrm{coeq}\left(\xymatrix{(\cA^\vee[n-1] \overset{\mathbb{L}}{\underset{\cat{A}^e}{\otimes}} \cB^e) \overset{\mathbb{L}}{\underset{\cat{B}^e}{\otimes}} \cG_n(f)^e \ar@<0.5ex>[r]\ar@<-0.5ex>[r]
& \cB \overset{\mathbb{L}}{\underset{\cat{B}^e}{\otimes}} \cG_n(f)^e}\right)\\
&=\mathrm{coeq}\left(\xymatrix{\cA^\vee[n-1] \overset{\mathbb{L}}{\underset{\cat{A}^e}{\otimes}}\cG_n(f)^e \ar@<0.5ex>[r]\ar@<-0.5ex>[r]
& \cB \overset{\mathbb{L}}{\underset{\cat{B}^e}{\otimes}} \cG_n(f)^e}\right)
\end{align*}
  implies \begin{align*}
\cG_n(f)^\vee[n]&\simeq\mathrm{eq}\left(\xymatrix{\cB^\vee[n] \overset{\mathbb{L}}{\underset{\cat{B}^e}{\otimes}}  \cG_n(f)^e \ar@<0.5ex>[r]\ar@<-0.5ex>[r]
& \cA[1] \overset{\mathbb{L}}{\underset{\cat{A}^e}{\otimes}} \cG_n(f)^e}\right)\\
&=\mathrm{coeq}\left(\xymatrix{\cB^\vee[n-1] \overset{\mathbb{L}}{\underset{\cat{B}^e}{\otimes}}\cG_n(f)^e \ar@<0.5ex>[r]\ar@<-0.5ex>[r]
& \cA \overset{\mathbb{L}}{\underset{\cat{A}^e}{\otimes}} \cG_n(f)^e}\right).
\end{align*}
Similarly
\begin{align*}
\cG_n(\cA) \overset{\mathbb{L}}{\underset{\cG_n(\cA)^e}{\otimes}} \cG_n(f)^e&\simeq\mathrm{coeq}\left(\xymatrix{\cA^\vee[n-1] \overset{\mathbb{L}}{\underset{\cat{A}^e}{\otimes}}\cG_n(f)^e \ar@<0.5ex>[r]\ar@<-0.5ex>[r]
& \cA \overset{\mathbb{L}}{\underset{\cat{A}^e}{\otimes}} \cG_n(f)^e}\right)
\end{align*}
and
\begin{align*}
\cG_n(\cB)  \overset{\mathbb{L}}{\underset{\cG_n(\cB)^e}{\otimes}} \cG_n(f)^e&\simeq\mathrm{coeq}\left(\xymatrix{\cB^\vee[n-1] \overset{\mathbb{L}}{\underset{\cat{B}^e}{\otimes}}\cG_n(f)^e \ar@<0.5ex>[r]\ar@<-0.5ex>[r]
& \cB \overset{\mathbb{L}}{\underset{\cat{B}^e}{\otimes}} \cG_n(f)^e}\right).
\end{align*}
Note that the resolutions of $\cG_n(\cB) \otimes^\bbL_{\cG_n(\cB)^e} \cG_n(f)^e$ and $\cG_n(f)$ are obtained from those of  $\cG_n(f)^\vee[n]$ and $\cG_n(\cA) \otimes^\bbL_{\cG_n(\cA)^e} \cG_n(f)^e$ respectively by post-composing by $f\otimes id$, which proves that
 the homotopy square in Definition \ref{definition: CYspan} is  (co)cartesian. 
\end{proof}
\begin{corollary}[Relative Calabi--Yau completions]
Given a morphism $f:\cat{A}\to\cat{B}$ of smooth dg-categories, the exact $n$-Calabi-Yau structure on $\mathcal G_n(\cat{A})$ 
lifts to an exact $n$-Calabi--Yau structure on the morphism 
\[
\mathcal G_n(\cat{A})\longrightarrow \mathcal G_{n+1}(\cat{B}|\cat{A}):=T_{\cat{B}}\Big(\mathrm{cofib}\big(
\cat{B}^\vee\overset{f^\vee}{\rightarrow}\cat{A}^\vee\overset{\mathbb{L}}{\underset{\cat{A}^e}{\otimes}}\cat{B}^e\big)[n-1]\Big)
\]
\end{corollary}
\begin{proof}
The proof is just based on the fact that exact $n$-Calabi--Yau stuctures do compose \textit{via} push-out. We have: 
\begin{equation}\label{diagGn}
\xymatrix{
&& \mathcal G_{n+1}(\cat{B}|\cat{A})  && \\
& \mathcal{G}_n(f) \ar[ru] & &  \cat{B}\ar[lu] \\
\mathcal{G}_n(\cat{A})\ar[ru] & & \mathcal{G}_n(\cat{B})\ar[lu]\ar[ru] & & ~~~~\emptyset~~~~ \ar[lu]
}
\end{equation}
The exact $n$-Calabi--Yau structure on the two cospans comes from Theorem \ref{thm:Gn-functoriality} (for the second one, 
we apply it to the initial functor $\emptyset\to\cat{B}$, and then pass to the opposite cospan). We thus simply have to prove 
that 
\[
\mathcal G_{n+1}(\cat{B}|\cat{A})\simeq 
T_{\cat{B}}(\cat{A}^\vee[n-1]\overset{\mathbb{L}}{\underset{\cat{A}^e}{\otimes}}\cat{B}^e)
\underset{T_{\cat{B}}(\cat{B}^\vee[n-1])}{\coprod}\cat{B}\,
\]
which simply follows from that $T_{\cat{B}}$ commutes with push-outs (it is a left adjoint) and $T_{\cat{B}}(0)=\cat{B}$. 
\end{proof}

\begin{remark}
The above corollary extends \cite[Theorem 3.13]{Ye0} from finite cellular to smooth dg-categories. 
\end{remark}

\begin{lemma} \label{lemma: G(f)}
Let $f:\cat{A}\to\cat{B}$ be dg-functor. Then 
\[
\mathcal{G}_n(f)\simeq \mathcal G_n(\cat{A})\underset{\cat{A}}{\coprod}\cat{B}\,.
\]
\end{lemma}
\begin{proof}[Proof of the Lemma]
Observe that we have a commuting square of right adjoints 
\[
\xymatrix{
(\scat{Cat}_k)_{\cat{B}/} \ar[r]\ar[d]	& \scat {Mod}_{\cat{B}^e}\ar[d] \\
(\scat{Cat}_k)_{\cat{A}/} \ar[r]		& \scat {Mod}_{\cat{A}^e}
}
\]
Thus, passing to left adjoints, we have, for every $\cat{A}^e$-module $\cat{M}$, 
\[
T_{\cat{B}}\left(\cat{M}\overset{L}{\underset{\cat{A}^e}{\otimes}}\cat{B}^e\right)
\simeq T_{\cat{A}}(\cat{M})\underset{\cat{A}}{\coprod}\cat{B}
\]
With $\cat{M}=\cat{A}^\vee[n-1]$, we get the desired result. 
\end{proof}

\begin{theorem}\label{theorem: compositionCY}
The construction $\mathcal{G}_n$ defines a functor 
\[
\Ho(\scat{Cat}_k^{sm})\longrightarrow \Ho(\scat{ExCY}_{[n]})\,, 
\]
where $\scat{Cat}_k^{sm}$ denotes the full subcategory of $\scat{Cat}_k$ spanned by smoothdg-categories. 
\end{theorem}
\begin{proof}
The proof is similar to the one of Proposition \ref{prop:cotangentfunctor}. 
We have already defined $\mathcal G_{n}$ on dg-categories and on functors. One sees that 
\[
\mathcal G_n\left(\cat{A}\overset{\mathrm{id}}{\longrightarrow}\cat{A}\right)
=\mathcal G_n(\cat{A})\overset{\mathrm{id}}{\longrightarrow}\mathcal G_n(\cat{A})
\overset{\mathrm{id}}{\longleftarrow}\mathcal G_n(\cat{A})\,,
\]
is an exact $n$-Calabi--Yau cospans. 
It thus remains to prove that $\mathcal G_n$ sends composition of morphisms to composition of exact $n$-Calabi--Yau cospans. 
Let $\cat{A}\overset{f}{\longrightarrow}\cat{B}\overset{g}{\longrightarrow}\cat{C}$. 
Thanks to the above Lemma, 
\begin{eqnarray*}
\mathcal G_n(f)\underset{\mathcal G_n(\cat{B})}{\coprod}\mathcal G_n(g) & \simeq & 
\left(\mathcal G_n(\cat{A})\underset{\cat{A}}{\coprod}\cat{B}\right)
\underset{\mathcal G_n(\cat{B})}{\coprod}
\left(\mathcal G_n(\cat{B})\underset{\cat{B}}{\coprod}\cat{C}\right) \\
& \simeq & \mathcal G_n(\cat{A})\underset{\cat{A}}{\coprod}\cat{C} \simeq \mathcal G_n(g\circ f)\,.
\end{eqnarray*}
So at the level of dg-categories, $\mathcal G_n$ indeed sends the composition of morphism to the composition of push-outs. 

Let us now turn to the proof that it is compatible with the composition of exact $n$-pre-Calabi--Yau structures, in the 
following sense: from Theorem \ref{thm:Gn-functoriality} we have exact $n$-pre-Calabi--Yau  structures on $\mathcal G_n(f)$, 
$\mathcal G_n(g)$, and $\mathcal G_n(g\circ f)$, and we want the push-out of the first two to be homotopic to the one on 
the third. Recall from the proof of Theorem \ref{thm:Gn-functoriality} that, with a morphism $f:\cat{A}\to\cat{B}$, one 
can associate a diagram $\mathcal{L}_n(f)$  
\[
\mathcal{L}_n(f):=\left(k[n-1] \overset{f}{\rightarrow}\ul{\cat{Map}}_{\scat{Mod}_{\cat{A}^e}}(\cat{A},\cat{B})[n-1] \simeq
(\cat{A}^\vee[n-1])\overset{\mathbb{L}}{\underset{\cat{A}^e}{\otimes}}\cat{B}^{\op} \rightarrow
\mathcal G_n(f)\overset{\mathbb{L}}{\underset{\mathcal G_n(f)^e}{\otimes}}\mathcal G_n(f)^{\op}\right)\,,
\]
in $\scat{Mod}_k$, and that we have a cospan of such diagrams 
\[
\mathcal{L}_n(\cat{A})\overset{f\circ}{\longrightarrow}\mathcal{L}_n(f)\overset{\circ f}{\longleftarrow}\mathcal{L}_n(\cat{B})\,,
\]
where $\mathcal{L}_n(\cat{A}):=\mathcal{L}_n(\mathrm{id}_{\cat{A}})$, by convention. 
This cospan of diagrams precisely determines the exact $n$-pre-Calabi--Yau structure on
\[
\mathcal{G}_n(\cA)\rightarrow\mathcal{G}_n(f)\leftarrow\mathcal{G}_n(\cat{B})\,.
\]
Given a sequence of morphisms $\cat{A}\overset{f}{\rightarrow}\cat{B}\overset{g}{\rightarrow}\cat{C}$, 
we then have a commuting diagram (of diagrams) 
\[
\xymatrix{
&& \mathcal{L}_n(g\circ f)  && \\
& \mathcal{L}_n(f) \ar[ru]^-{g\circ} & & \mathcal{L}_n(g)\ar[lu]_-{\circ f} \\
\mathcal{L}_n(\cat{A})\ar[ru]^-{f\circ} & & \mathcal{L}_n(\cat{B})\ar[lu]_-{\circ f}\ar[ru]^-{g\circ} 
& & \mathcal{L}_n(\cat{C})\ar[lu]_-{\circ g}
}
\] 
which proves the result. 
\end{proof}
\begin{remark}
The functor $\mathcal G_n$ can be lifted to the $\infty$-categorical level, from $\scat{Cat}_k^{sm}$ to $\scat{ExCY}_{[n]}$, 
but we leave this to a subsequent work. 
\end{remark}

\subsubsection{Deformed Calabi--Yau completions}\label{ss: deformed}

Let $f:\cat{A}\to\cat{B}$ be a morphism between smooth dg-categories. 
As in~\S\ref{critrel}, we want to modify~\eqref{diagGn}, using any $T_{\cat B}(c):\mathcal G_n(\cat B)\to\cat B$, 
$c\in|\cat{HH}(\cat B)[1-n]|$, instead of $T_{\cat B}(0)$. We thus define $\mathcal G_{n+1}(\cat B|\cat A,c)$ as 
the following composition of push-outs:
\begin{equation}\label{compdef}
\xymatrix{
&& \mathcal G_{n+1}(\cat{B}|\cat{A},c)  && \\
& \mathcal{G}_n(f) \ar[ru] & &  \cat{B}\ar[lu] \\
\mathcal{G}_n(\cat{A})\ar[ru] & & \mathcal{G}_n(\cat{B})\ar[lu]\ar[ru]^-{T_{\cat B}(c)} & & ~~~~\emptyset~~~~ \ar[lu]
}
\end{equation}
In a similar way as in Definition~\ref{def: deformed}, the push-out $\mathcal G_{n+1}(\cat{B}|\cat{A}, c)$ can be 
identified with the deformation of $\mathcal G_{n+1}(\cat{B}|\cat{A})$ by means of (the extension by the Leibniz rule of) 
the composed map 
\[
\mathrm{cofib}(f^\vee)[n-1]\longrightarrow \cat{B}^\vee[n]\overset{c}{\longrightarrow}\cat{B}[1]\,.
\]

Now let us assume that there exists a Calabi--Yau structure on $T_\cB(c)$ given by a negative cyclic lift of $c$, according 
to Proposition~\ref{proposition: leftCY}. Thanks to Theorem~\ref{thm:Gn-functoriality}, we get an alternate proof 
of~\cite[Theorem 3.23]{Ye0} when $\mathcal G_n(\cat A)$ is not deformed, \textit{i.e.}\ when considering the trivial Hochschild 
class on $\cat A$.

\begin{proposition} \label{prop: relCY}
Every negative cyclic lift of $c$ induces a $n$-Calabi--Yau structure on $\cG_n(\cA)\to\cG_{n+1}(\cB|\cA,c)$.
\end{proposition}

Now assume that we are given a (possibly non-trivial) class $c:k[n-1]\to\cat{HH}(\mathcal A)$, and denote $c'$ its image 
in $\cat{HH}(\cB)$. For every lift of $c$ in $\cat{HC}^-(\cat{A})$, we get a corredponding lift of $c'$ in $\cat{HC}^-(\cat{B})$, 
and thus, according to Proposition~\ref{proposition: leftCY} and Theorem~\ref{thm:Gn-functoriality}, we have diagram of 
Calabi--Yau cospans 
\begin{equation}\label{diagCYs}
\xymatrix{
&\cA\ar[rr]&&\cB\pullbackcorner&\\
\emptyset \ar[ur]\ar[dr]&&\cG_n(\cA)\ar[lu]_-{T_\cA(c)}\ar[dl]^-{T_\cA(0)}\ar[r]&\cG_n(f)\ar[u]\ar[d]&\cG_n(\cB)\ar@{-->}[ul]_-{T_\cB(c')}\ar@{-->}[dl]^-{T_\cB(0)}\ar[l]\\
&\cA\ar[rr]&& \cB\pullbackcorner[ul]&
}
\end{equation}

We postpone to a future work the technical proof of the following Lemma (that we won't use after the end of \S\ref{ss: deformed}), 
related to the $(\infty,2)$-categorical nature of iterated Calabi--Yau cospans. 
\begin{lemma}\label{lemma-conj}
Assume that we are given a $n$-Calabi--Yau cospan $\mathcal X\rightarrow\mathcal Z\leftarrow\mathcal Y$ along with two 
$n$-Calabi--Yau morphisms $\mathcal X\rightarrow \mathcal U$ and $\mathcal X\rightarrow\mathcal V$. Then the cospan
\[
\mathcal U\underset{\mathcal X}{\coprod} \mathcal V\rightarrow \mathcal U\underset{\mathcal X}{\coprod} \mathcal Z \underset{\mathcal X}{\coprod} \mathcal V\leftarrow \left( \mathcal U\underset{\mathcal X}{\coprod} \mathcal Z\right)\underset{\mathcal Y}{\coprod}\left( \mathcal Z\underset{\mathcal X}{\coprod} \mathcal V\right)
\]
is $(n+1)$-Calabi--Yau. 
\end{lemma}

Using the previous diagram~\eqref{diagCYs}, we apply the Lemma with $\mathcal U=\mathcal V=\cat{A}$ and 
$\mathcal X\rightarrow\mathcal Z\leftarrow\mathcal Y=\cG_n(\cA)\rightarrow \cG_n(f)\leftarrow \cG_n(\cB)$. 
In this case we have 
\[
\mathcal{G}_{n+1}(f,c_f):= \mathcal U\underset{\mathcal X}{\coprod} \mathcal Z \underset{\mathcal X}{\coprod} \mathcal V
\simeq \cat{B}\underset{\mathcal G_n(f)}{\coprod}\cat{B}\,,
\]
which can be identified with the deformation of $\mathcal{G}_{n+1}(f)$ by means of the (Leibniz extension of the) composed map 
\[
c_f[1]:\cA^\vee[n]\otimes^\bbL_{\cA^e}\cB^e\overset{c[1]}{\longrightarrow}\cA\otimes^\bbL_{\cA^e}\cB^e[1]
\overset{f[1]}{\longrightarrow}\cB[1]\,.
\]
\begin{corollary}
Under the assumption that Lemma \ref{lemma-conj} holds, the cospan
\[
\mathcal G_{n+1}(\cA,c)\rightarrow\cG_{n+1}(f, c_f)\leftarrow\cG_{n+1}(\cB,c')
\]
is $(n+1)$-Calabi--Yau. 
\end{corollary}
\begin{remark}
We will prove in a subsequent work the above construction defines a functor 
\[
\Ho\left((\scat{Cat}_k)^{sm}_{k[n-2]/\cat{HC}^-}\right)\longrightarrow \Ho(\scat{CY}_{[n]})\,.
\]
\end{remark}
Now assume that the negative cyclic lift of $c'$ is (homotopic to) $0$, \textit{i.e.}\ that 
$\cG_{n+1}(\cB,c')\simeq \cG_{n+1}(\cB)$ as a Calabi--Yau category. 

We thus get a Calabi--Yau push-out
\[
\xymatrix{
&& \mathcal G_{n+2}(\cat{B}|\cat{A},\tilde c)  && \\
& \mathcal{G}_{n+1}(f,c_f) \ar[ru] & &  \cat{B}\ar[lu] \\
\mathcal{G}_{n+1}(\cat{A},c)\ar[ru] & & \mathcal{G}_{n+1}(\cat{B})\ar[lu]\ar[ru]^-{T_{\cat B}(0)} & & ~~~~\emptyset~~~~ \ar[lu]
}
\]
which allows us to retrieve~\cite[Theorem 3.23]{Ye0} in its full generality. Let us be more specific. 

The fact that $c'$ is homotopic to $0$ tells us that we have a factorization of $c_f$ as follows: 
\[
\xymatrix{
\cat A^\vee[n-1]\overset{\mathbb{L}}{\underset{\cat{A}^e}{\otimes}}\cat{B}^e\ar[d]_-c\ar[r]&\mathrm{cofib}(f^\vee)[n-1]\ar[d]^{\tilde c}\\
\cat A\overset{\mathbb{L}}{\underset{\cat{A}^e}{\otimes}}\cat{B}^e\ar[r]^f&\cat B}
\]
One can then identify the push-out $\mathcal G_{n+2}(\cat{B}|\cat{A},\tilde c)$ with the deformation of 
$\mathcal G_{n+2}(\cat{B}|\cat{A})$ by means of the morphism $\tilde{c}[1]$, as we have already done on several occasions; 
whence the notation. 

\begin{theorem}
Under the assumption that Lemma \ref{lemma-conj} holds,  to any negative cyclic lift of the pair $(c,\tilde{c})$, one can associate a $(n+1)$-Calabi--Yau structure on 
$\cG_{n+1}(\cA,c)\to\cG_{n+2}(\cB|\cA,\tilde c)$.
\end{theorem}
\begin{remark}
\begin{enumerate}
\item The data we use is the same as Yeung's, who equivalently requires a lift of $c$ in relative negative cyclic homology. 
\item Note that using $T_\cB(\beta)$ for some class $\beta$ would just amount to picking a different homotopy from $c'$ 
to $0$, through the self-homotopy $\beta$ of $0$.
\end{enumerate}
\end{remark}

\subsection{Free examples}

In this subsection, we make some constructions more explicit on free dg-categories. 
Our setup is the following: $I$ is a finite set, $\mathcal M$ is a perfect cochain complex carrying a $I\times I$-grading, 
turning it into a perfect $k^I$-bimodule, and $\cat{A}:=T_{k^I}(\mathcal M)$. 
The typical situation we have in mind is the one quivers, which corresponds to the case when $\mathcal M$ is concentrated in degree $0$ 
and comes with a basis. 

\medskip

Recall that, for a free category $\cat{A}$, the dualizing bimodule can be computed easily. As $\cat{A}^e$-modules, we have 
\[
\cat{A}\simeq\mathrm{coeq}\left(\xymatrix{\cat{A}\underset{k^I}{\otimes}\mathcal M\underset{k^I}{\otimes}\cat{A}
\ar@<0.5ex>[r]^{~~~m\otimes\mathrm{id}}\ar@<-0.5ex>[r]_{~~~\mathrm{id}\otimes m} 
& \cat{A}\underset{k^I}{\otimes}\cat{A}} \right)
\]
where $m$ denotes the multiplication in $\cA$.
Hence (see for instance~\cite[Lemma 3.2]{Keller}), 
\[
\cat{A}^\vee\simeq \mathrm{eq}\left(\xymatrix{\cat{A}\underset{k^I}{\otimes}\cat{A}\ar@<0.5ex>[r]^{\!\!\!\!\!\!\!c^\op\otimes\mathrm{id}~~}\ar@<-0.5ex>[r]_{\!\!\!\!\!\!\!\mathrm{id}\otimes c~~} 
& \cat{A}\underset{k^I}{\otimes}\mathcal M^\vee\underset{k^I}{\otimes}\cat{A}}\right)
\]
where $c\in \mathcal M^\vee\otimes_{k^I} \mathcal M$ denotes the Casimir element corresponding to the identity through $\ul{\cat{Hom}}_{k^I}(\mathcal M,\mathcal M)=\mathcal M^\vee\otimes_{k^I}\mathcal M$. Note that $(\mathcal M^\vee)_{i,i'}=(\mathcal M_{i',i})^\vee$. 

\subsubsection{The relative Calabi--Yau completion as a double}

Recall that $\mathcal G_n(k^I)=\underset{i\in I}{\amalg}k[x_{i}]$, with exact $n$-Calabi--Yau structure given by 
$x:=\sum_{i\in I}x_{i}$, and where $x_i$'s are degree $1-n$ variables. 

\begin{proposition}\label{relCYdouble}
The relative Calabi--Yau completion $\mathcal G_n(k^I)\to \mathcal G_{n+1}(\cat{A}|k^I)$ is equivalent to 
the morphism 
\[
 \phi:\underset{i\in I}{\amalg}k[x_{i}] \to T_{k^I}(\mathcal M\oplus \mathcal M^\vee[n-1])
\]
sending $x$ to $(c^{\op}-c)[n-1]$ in $T_{k^I}(\mathcal M\oplus \mathcal M^\vee[n-1])$. 
\end{proposition}
\begin{proof}
By definition, $\mathcal G_{n+1}(\cat{A}|k^I)=T_{\cat{A}}\Big(\mathrm{cofib}\big(
\cat{A}^\vee{\rightarrow}(k^I)^\vee{\underset{(k^I)^e}{\otimes}}\cat{A}^e\big)[n-1]\Big)$. 
The map 
\[
\cat{A}^\vee\simeq \mathrm{eq}\left(\xymatrix{\cat{A}\underset{k^I}{\otimes}\cat{A}\ar@<0.5ex>[r]^{\!\!\!\!\!\!\!c^\op\otimes\mathrm{id}~~}\ar@<-0.5ex>[r]_{\!\!\!\!\!\!\!\mathrm{id}\otimes c~~} 
& \cat{A}\underset{k^I}{\otimes}\mathcal M^\vee\underset{k^I}{\otimes}\cat{A}}\right)
\rightarrow 
(k^I)^\vee{\underset{{k^I}^e}{\otimes}}\cat{A}^e\simeq\cat{A}\underset{k^I}{\otimes}\cat{A} 
\]
is given by $1\otimes 1\mapsto  x[1-n]\otimes1$ on the leftmost term of the equalizer (and $0$ on the rightmost one), and thus its (homotopy) cofiber
is 
\[
\cat{A}\underset{k^I}{\otimes}\mathcal M^\vee\underset{k^I}{\otimes}\cat{A}\simeq \mathcal M^\vee\underset{(k^I)^e}{\otimes}\cat{A}^e\,.
\]
Hence we have
\[
\mathcal G_{n+1}(\cat{A}|k^I)	=		T_{\cat{A}}\left(\mathcal M^\vee[n-1]\underset{(k^I)^e}{\otimes}\cat{A}^e\right)
								\simeq	T_{k^I}(\mathcal M^\vee[n-1])\underset{k^I}{{\coprod}}\mathcal A
								=		T_{k^I}(\mathcal M\oplus \mathcal M^\vee[n-1]).
\]
Since $\mathcal G_n(k^I)\to\mathcal G_n(k^I\to\cA)=T_\cA((k^I)^\vee[n-1]\otimes_{(k^I)^e}\cat{A}^e)$ maps $x$ to $x\otimes 1$, we see from above that $\mathcal G_n(k^I)\to\mathcal G_{n+1}(k^I|\cA)$
maps $x$ to the image of $(c^\op\otimes 1-1\otimes c)[n-1]$ in $T_{k^I}(\mathcal M\oplus \mathcal M^\vee[n-1])$.
\end{proof}

\begin{remark}\label{homocasi}
Note that the homotopy from $\phi(B(x))=\phi(1\otimes x)$ to zero is given by $B(c^\op[n-1])$, \textit{i.e.} the shifted 
image of the opposite of the Casimir element by Connes' boundary operator.
\end{remark}

\subsubsection{Ginzburg dg-algebras}\label{subsubginz}

Consider the same setting as in Remark~\ref{setquivfree}. Let us perform the following composition of exact 
Calabi--Yau structures, analogous to~\eqref{grodiag}:
\[
\xymatrix{
&&&\mathcal C'&&&\\
&&\mathcal C\ar[ru]&&\cG_2(kQ|k^I)\ar[lu]&&\\
&k^I\ar[ru]&&\cG_1(k^I\to kQ)\ar[ru]\ar[lu]&&kQ\ar[lu]&\\
\emptyset\ar[ru]&&\cG_1(k^I)\ar[ru]\ar[lu]&&\cG_1(kQ)\ar[ru]\ar[lu]&&\emptyset.\ar[lu]
}\]
We have
\[
\mathcal C=k^I\underset{\cG_1(k^I)}{\coprod}\cG_1(k^I\to kQ)\underset{{\scriptsize (\textrm{lemma}~\ref{lemma: G(f)})}}{\simeq}
k^I\underset{\cG_1(k^I)}{\coprod}\cG_1(k^I)\underset{k^I}{\coprod}kQ\simeq kQ
\]
hence
\[
\mathcal C'=kQ\underset{\cG_1(kQ)}{\coprod}kQ\simeq \cG_2(kQ)\,.
\]
Since $T_{k^I}(\mathcal M\oplus\mathcal M^\vee)\simeq k\overline Q$ and $\cG_1(k^I)=k[x]^I$ with $x$ in degree $0$, we get 
by Proposition~\ref{relCYdouble} that
\[
\cG_2(kQ)\simeq k^I\underset{k[x]^I}{\coprod}k\overline Q\simeq k\underset{k[x]}{\coprod}k\overline Q\,,
\]
with $k[x]\to k\overline Q$ defined by $x\mapsto \sum_{e\in E(Q)}(ee^*-e^*e)$, using the obvious basis $E(Q)$ of $\mathcal M=kE(Q)$, 
as announced in Remark~\ref{remG2}.

Let us compute the $2$-Calabi--Yau structure on the homotopy push-out $\mathcal G_2(kQ)	=  k{\coprod}_{k[x]} k\overline Q$. 
It may be done using a $k[x]$-cofibrant replacement of $k$. We consider $k[x,x^*]$ with $x^*$ in degree $-1$ and 
$\mathfrak d(x^*)=x$. The map $k[x]\to k[x,x^*]$ is the obvious inclusion, and the homotopy from the image of $1\otimes x$ 
to zero is given by $1\otimes x^*$. From this and Remark~\ref{homocasi}, we see that in the Hochschild complex of the push-out
\[
k\overline Q\underset{k[x]}{\coprod}k[x,x^*]
\]
we get a self-homotopy of zero, that is a degree $-1$-cocycle 
given as the difference 
\[
\sum_{e\in E(Q)}e\otimes e^*-x^*\,.
\]
This naturally lives in the standard Hochschild complex of $\cG_2(kQ)$ and again, applying Connes' boundary operator we get a $2$-Calabi--Yau structure on $\cG_2(kQ)$. 

\begin{remark}
A {potential} of the quiver $Q$, as defined in~\S\ref{pot3CY}, is an element $W\in kQ/[kQ,kQ]$ and may be seen as 
a class $k\to\cat{HH}(kQ)$. Hence $\delta W:k\to\cat{HC}^-(kQ)[-1]$, and the deformed $3$-Calabi--Yau completion 
$\cG_3(kQ,\delta W)$ is the same as in Definition~\ref{defW3CY}, as they are both defined as the push-out of 
\[
\xymatrix{
\cG_{2}(kQ) \ar[r]^{~~\delta W} 	\ar[d]	& kQ\\
kQ												&
					}
\]
Similarly the dg-category $\cG_3(kQ|kD,\delta W)$ associated with an inclusion $D\subseteq Q$ defined by~\eqref{compdef} 
is the same as the one appearing in Proposition~\ref{defG3QDW}.
\end{remark}

%% file: CritQuiv-Section6.tex
\label{section6.1}

In this section, we show that the cotangent stack to the moduli of objects~\cite{ToVa} of a finite type dg-category is 
equivalent, as a symplectic stack, to the moduli of object of its Calabi--Yau completion. 
Moreover, this is compatible with the functoriality of our constructions (Calabi--Yau completion on the one hand, and 
taking cotangent stacks on the other hand), in the following sense: 

\begin{theorem}\label{thm:mainmain}
There is a natural isomorphism 
\[
\xymatrix{
\Ho(\scat{Cat}_k^{f.t.})^{\op} \ar[r]^-{\sigma\circ \cG_n} \ar[d]_-{\bPerf} 
& \Ho(\scat{ExCY}_{[n]}^{f.t.}) \ar[d]^-{\bPerf} \ar@{=>}[ld] \\
\Ho(\scat{dSt}_k^{Art}) \ar[r]_-{ \bT^*[2-n]} & \Ho(\scat{ExLag}_{[2-n]})
}
\]
where $\scat{Cat}_k^{f.t.}$ is the full subcategory of $\scat{Cat}_{k}$ spanned by finite type dg-categories, and 
$\scat{ExCY}_{[n]}^{f.t.}$ is the subcategory of $\scat{ExCY}_{[n]}$ whose objects and morphisms only involve finite 
type dg-categories. 
\end{theorem}
The rest of this section is devoted to the proof of this result: 
\begin{itemize}
\item We first recall the works of To\"en--Vaqui\'e~\cite{ToVa} and Brav--Dyckerhoff~\cite{BD2} (see also To\"en's~\cite{ToCY}), 
that make sense of the vertical functors;
\item We then prove a version of the theorem that does not take into account symplectic and lagrangian structures, 
i.e.~functors take their values in $\Ho\big(\scat{Span}_1(\scat{dSt}_k)\big)$ instead of $\scat{ExLag}_{[2-n]}$; 
\item We finally check the compatibility with exact symplectic and lagrangian structures. 
\end{itemize}

\medskip

Along with a side result relating graphs of closed forms with Calabi--Yau morphisms associated with negative cyclic classes, 
this allows us to prove that the moduli of objects of $\mathcal{G}_3(kQ|kD,\delta W)$ can be indeed identified, as a lagrangian, 
with a relative (derived) critical locus. 

\subsection{From Calabi--Yau structures to shifted symplectic structures}

\subsubsection{Moduli of objects}

Let $\cat{A}$ be a dg-category, that we assume to be of finite type. We denote its derived moduli stack of objects 
(from \cite{ToVa}) by $\bPerf_{\cA}$: for a cdga $B$
\[
\scat{Perf}_{\cat{A}}(B):=\cat{Map}_{\scat{Cat}_k}(\cat{A},\cat{Mod}_B^\mathrm{perf})\,.
\]
In other words, a $B$-point $x$ is then given by a $\cat{A}-B$-bimodule $M_x$ that is pseudo-perfect (i.e.~it is perfect as 
a $B$-module). 
\begin{example}
For a finite set $I$, $\scat{Perf}_{k^I}\simeq\scat{Perf}_{kI}\simeq\scat{Perf}_k^I$: a $B$-point $x$ of $\scat{Perf}_{k^I}$ 
is an $I$-family $M_x$ of perfect $B$-modules. 
Therefore, if $\cat{A}=T_{k^I}(N)$ is free, then $\scat{Perf}_{\cat{A}}$ is a stack over $\scat{Perf}_{k^I}$, and
for every $x:\scat{Spec}(B)\to \scat{Perf}_{k^I}$, 
\[
\scat{Perf}_{\cat{A}}(x)
=\cat{Map}_{\scat{Cat}_k}\big(T_{k^I}(N),\cat{Mod}_B^\mathrm{perf}\big)\underset{(\cat{Mod}_B^\mathrm{perf})^I}{\times}\{M_x\}
=\scat{Map}_{\scat{Mod}_{(k^I)^e}}\big(N,\underline{\mathrm{End}}_B(M_x)\big)\,.
\]
Here the $k^I$-bimodule $\underline{\mathrm{End}}_B(M_x)$ can be described in the following equivalent ways: 
\begin{itemize}
\item Abstractly, it is just the $k^I$-bimodule structure on $\cat{Mod}_B^\mathrm{perf}$ that we get from the dg-functor 
$M_x:k^I\to \cat{Mod}_B^\mathrm{perf}$;
\item In concrete terms, it is the $I\times I$-graded complex given by 
\[
\mathrm{End}_B(M_x)_{i,j}=\underline{\cat{Hom}}_{\cat{Mod}_B}\big(M_x(j),M_x(i)\big)\,.
\]
\end{itemize}
\end{example}
If we borrow the notation from Example~\ref{exQuiv2} and write $E$ for the tautological sheaf on 
$\scat{Perf}_k$, then, for a $B$-point $x$ of $\scat{Perf}_{k^I}$, 
\[
M_x\simeq x^*\big(\underbrace{\bigoplus_{i\in I}p_i^*E}_{=:E^I}\big)\,,
\]
where $p_i:\scat{Perf}_k^I\to \scat{Perf}_k$ is the $i$-th projection. 
Given a $k^I$-bimodule $N$, we get a quasi-coherent sheaf $\underline{\cat{End}}(E^I)^N\in \scat{QCoh}(\scat{Perf}_k^I)$ 
defined by 
\[
x^*(\underline{\cat{End}}(E^I)^N):=\underline{\scat{Map}}_{\scat{BiMod}_{(k^I)^e}}\big(N,\underline{\mathrm{End}}_B(M_x)\big)\,.
\]
Hence, by definition, for $\cat{A}=T_{k^I}(N)$, 
\[
\scat{Perf}_{\cat{A}}(x)\simeq |x^*(\underline{\cat{End}}(E^I)^N)|\,. 
\]
In other words: 
\begin{lemma}\label{lemma: linear}
The moduli of objects $\bPerf_{\cat{A}}$ of a free dg-category $\cat{A}=T_{k^I}(N)$ is equivalent to the linear stack 
$\mathbb A_{\cat{End}(E^I)^N}$.
\end{lemma}
\begin{example}
In the quiver case (i.e.~$I=V(Q)$ and $N=kE(Q)$), one observes that 
\[
\cat{End}(E^I)^N\simeq E^Q\,,
\]
again following the notation of Example~\ref{exQuiv2}. Hence we get an equivalence 
$\scat{Perf}_{kQ}\simeq \mathbb{A}_{E^Q}$, as witnessed in Remark~\ref{labelresBG}. 
\end{example}

The following is now a direct consequence of Lemma~\ref{lemma: linear} and Remark~\ref{remlinquiv}.

\begin{proposition}\label{proposition: cotangentkQ}
The following diagram is a homotopy pullback square in the category of derived stacks
\[
\xymatrix{
\bT^* \bPerf_{k Q}  \ar[r]	\ar[d] 	& \bPerf_{ k^{ V(Q)}} \ar[d] \\
\bPerf_{k \overline{Q}}													\ar[r]		&  \bPerf_{k[x]^{V(Q)}} }
\]
The map $\bPerf_{k \overline{Q}} \to \bPerf_{k[x]^{V(Q)}} $ is induced by the $k$-linear functor 
$k[x]^{V(Q)} \to k \overline{Q}$ that maps the $v$-th generator $x$ to $\sum_{e \in E(Q)} e_v(ee^*-e^*e)e_v$ for all $v \in V(Q)$, if $e_v$ stands for the length $0$ idempotent path at $v$. 
\end{proposition}
\begin{remark}
The Proposition will also be a consequence of Corollary~\ref{theorem: compareGinzburg} together with~\S\ref{subsubginz}. 
\end{remark}

\subsubsection{Forms on the moduli of objects}\label{sss: CYLag}

Let $\cat{C}$ be a dg-category. We have a sequence of morphism $\mathrm{ch}_{\cat{C}}$ defined as the composition, 
\begin{eqnarray*}
\scat{HH}(\cat{C})	& \rightarrow	& \underset{\cat{C}\to\cat{Mod}_B^\mathrm{perf}}{\mathrm{holim}}\scat{HH}(\cat{Mod}_B^\mathrm{perf}) \\
					& \simeq		& \underset{\cat{C}\to\cat{Mod}_B^\mathrm{perf}}{\mathrm{holim}}\scat{HH}(B) \\
					& \tilde{\to}	& \underset{\cat{C}\to\cat{Mod}_B^\mathrm{perf}}{\mathrm{holim}}|\scat{DR}(B)|^\ell \\
					& \rightarrow	& \underset{\cat{C}\to\cat{Mod}_B^\mathrm{perf}}{\mathrm{holim}}|\scat{DR}(B)|^r \\
					& \simeq		& \Big|\underset{\cat{C}\to\cat{Mod}_B^\mathrm{perf}}{\mathrm{holim}}\scat{DR}(B)\Big|^r
										=:|\scat{DR}(\scat{Perf}_{\cat{C}})|^r
\end{eqnarray*}
which is natural in $\cat{C}$ and induces a commuting diagram 
\[
\xymatrix{
|\cat{HH}(\cat{C})[n+1]| \ar[r]^-{\mathrm{ch}^{(p-1)}_{\cat{C}}}\ar[d]_{\delta} & \mathcal A^{1}(\scat{Perf}_{\cat{C}},n+p) \ar[d]^{d_{dR}} \\
|\cat{HC}^-(\cat{C})[n]|\ar[r]^-{\tilde{\mathrm{ch}}^{(p)}_{\cat{C}}}\ar[d] & \mathcal A^{2,\cl}(\scat{Perf}_{\cat{C}},n+p) \ar[d] \\
|\cat{HH}(\cat{C})[n]|\ar[r]^-{\mathrm{ch}^{(p)}_{\cat{C}}} & \mathcal A^{2}(\scat{Perf}_{\cat{C}},n+p)
}
\]
that is also natural in $\cat{C}$. When $p=2$, the top part of the above diagram induces a commuting diagram 
of $\infty$-categories 
\[
\xymatrix{
\scat{ExprCY}_{[n]} \ar[r]^{\scat{Perf}}\ar[d]_{\mathbf{\Delta}} & \scat{ExIso}_{[2-n]} \ar[d]^{\mathbf{D_{dR}}} \\
\scat{prCY}_{[n]} \ar[r]^{\scat{Perf}} & \scat{Iso}_{[2-n]}
}
\]
and the bottom part ensures that $\tilde{\mathrm{ch}}^{(2)}_{\cat{C}}(-)^\sharp\simeq \mathrm{ch}^{(2)}_{\cat{C}}((-)^\natural)$. 
\begin{remark}
In the above diagram, all functors commute with the anti-involution $\sigma$ that reverses the direction of (co)spans. 
\end{remark}
\begin{example}
In weight $0$, the map 
\[
\mathrm{ch}^{(0)}:|\cat{HH}(\cat{C})[n]|\to \mathcal{A}^0(\scat{Perf}_{\cat{C}},n)=|\Gamma(\mathcal O_{\scat{Perf}_{\cat{C}}})[n]|
\]
is simply obtained from 
\[
\cat{HH}(\cat{C})	\to\underset{\cat{C}\to\cat{Mod}_B^\mathrm{perf}}{\mathrm{holim}}\cat{HH}(B)
					\to\underset{\cat{C}\to\cat{Mod}_B^\mathrm{perf}}{\mathrm{holim}}B
					=\Gamma(\mathcal O_{\scat{Perf}_{\cat{C}}})\,.
\]
\end{example}
The following example will be useful for the comparison results in the next Section.
\begin{example}[\cite{BD2}, Proposition 5.3]\label{ex:BDcompute}
Let us write $M_x$ for the pseudo-perfect $\cat{C}-B$-bimodule corresponding to a point $x:\scat{Spec}(B)\to\scat{Perf}_{\cat{C}}$, 
and recall that there are natural equivalences of $B$-modules
\[
x^*\mathbb{T}_{\scat{Perf}_{\cat{C}}}\simeq \underline{\cat{Map}}_{\scat{Mod}_{\cat{C}^{\op}}}(M_x,M_x)[1]\quad\textrm{and}\quad
x^*\mathbb{L}_{\scat{Perf}_{\cat{C}}}\simeq \underline{\cat{Map}}_{\scat{Mod}_{\cat{C}^{\op}}}
(\cat{C}^\vee\underset{\cat{C}}{\overset{\mathbb{L}}{\otimes}}M_x,M_x)[-1]\,.
\]
Therefore, applying \cite[Proposition 5.3]{BD2} to the weight one case, we get that $\mathrm{ch}^{(1)}$ is given by 
\begin{eqnarray*}
|\cat{HH}(\cat{C})[n]|\simeq \cat{Map}_{\scat{Mod}_{\cat{C}^e}}(\cat{C}^\vee,\cat{C})[n]
& \overset{-\otimes\mathrm{id}}{\longrightarrow} & 
\underset{x:\cat{C}\to\cat{Mod}_B^\mathrm{perf}}{\mathrm{holim}}\cat{Map}_{\scat{BiMod}_{\cat{C}-B}}
(\cat{C}^\vee\underset{\cat{C}}{\overset{\mathbb{L}}{\otimes}}M_x,M_x[n]) \\
& \simeq & |\Gamma(\mathbb{L}_{\scat{Perf}_{\cat{C}}})[1+n]| \\
& \rightarrow & \mathcal{A}^1(\scat{Perf}_{\cat{C}},1+n)\,.
\end{eqnarray*}
Recall that, if $\cat{C}$ is of finite type, then $\scat{Perf}_{\cat{C}}$ is Artin and thus the last morphism is 
actually an equivalence. 
\end{example}
We now recall the  
\begin{theorem}[\cite{BD2}, Theorem 5.5]
\begin{enumerate}
\item If a finite type dg-category $\cat{C}$ is equipped with a $n$-Calabi--Yau structure $c$, then 
$\tilde{\mathrm{ch}}^{(2)}_{\cat{C}}(c)$ is a $(n-2)$-shifted symplectic structure on $\scat{Perf}_{\cat{C}}$. 
\item The image by $\tilde{\mathrm{ch}}^{(2)}$ of a $n$-Calabi--Yau structure on a cospan of finite type dg-categories is a $(n-2)$-shifted 
lagrangian structure on the associated correspondence of derived moduli of objects. 
\end{enumerate}
\end{theorem}
As a consequence, the previous commuting diagram of $\infty$-categories restricts to 
\[
\xymatrix{
\scat{ExCY}^{f.t.}_{[n]} \ar[r]\ar[d]_{\mathbf{\Delta}} & \scat{ExLag}_{[2-n]} \ar[d]^{\mathbf{D_{dR}}} \\
\scat{CY}^{f.t.}_{[n]} \ar[r] & \scat{Lag}_{[2-n]}
}
\]
where the superscript $f.t.$ means that we restrict ourselves to the subcategory of cospans where all involved dg-categories are 
of finite type (instead of only being smooth).


\subsection{Calabi--Yau completions as noncommutative cotangents: underlying stacks}

In this subsection, we prove that $\bPerf_{\cG_n(\cA)} \simeq \bT^*[2-n]\bPerf_{\cA}$ for a finite type dg-category $\cA$, as 
well as relative version of this. 

We start with a generalization of Lemma~\ref{lemma: linear}. 
\begin{proposition}\label{prop:pasdenom}
Let $\cat{A}$ be a dg-category of finite type and $M$ a perfect $\cA$-bimodule. Then the derived moduli stack 
$\bPerf_{T_{\cA}(M)}$ is equivalent to the linear stack $\mathbb{A}_{ \mathcal{M}}$, where 
$\mathcal{M}\in \cat{Mod}_\cA^\mathrm{perf}$ is defined by 
\[
x^*\mathcal{M}:=\underline{\cat{Map}}_{\scat{Mod}_{\cat{A}^{\op}}}(M\overset{\mathbb{L}}{\underset{\cat{A}}{\otimes}} M_x, M_x)\,,
\]
on $B$-points $x:\scat{Spec}(B)\to\scat{Perf}_{\cat{A}}$. 
\end{proposition}
\begin{proof}
Both $\bPerf_{T_\cA(M)}$ and $\mathbb{A}_{ \mathcal{M}}$ are stacks over $\scat{Perf}_{\cat{A}}$. 
By the universal property of $T_{\cat{A}}(M)$, we have natural equivalences (in $x$)
\begin{eqnarray*}
\bPerf_{T_{\cA}(M)}(x)	& \simeq & \cat{Map}_{\scat{Mod}_{\cat{A}^e}}\big(M,\cat{Mod}_B^\mathrm{perf}\big)
\underset{\cat{Map}_{\scat{Cat}_k}(\cat{A},\scat{Mod}_B^\mathrm{perf})}{\times}\{x\} \\
						& \simeq & \cat{Map}_{\scat{Mod}_{\cat{A}^e}}\big(M,\underline{\cat{End}}_B(M_x)\big)\,,
\end{eqnarray*}
where $\underline{\cat{End}}_B(M_x)$ is the $\cat{A}$-bimodule structure on $\cat{Mod}_B^\mathrm{perf}$ given by 
$x:\cat{A}\to \cat{Mod}_B^\mathrm{perf}$; in other words, this is the dg-functor $\cat{A}^{\op}\otimes\cat{A}\to \cat{Mod}_k$ 
given by 
\[
(a,a')\longmapsto \underline{\cat{Map}}_{\cat{Mod}_B}\big(x(a),x(a')\big)\,.
\]
Finally, there is a natural equivalence 
\[
\cat{Map}_{\scat{Mod}_{\cat{A}^e}}\big(M,\cat{End}_B(M_x)\big)
\simeq \cat{Map}_{\scat{BiMod}_{\cat{A}-B}}(M\overset{\mathbb{L}}{\underset{\cat{A}}{\otimes}}M_x,M_x)=:|x^*\mathcal{M}|\,.
\]
This proves the result.
\end{proof}

\begin{corollary}\label{theorem: compareGinzburg}
Let $\cat{A}$ be a dg-category of finite type. Then the derived stacks $\bPerf_{\cG_n(\cA)}$ and 
$\bT^*[2-n] \bPerf_\cA$ are equivalent. 
\end{corollary}
\begin{proof}
Both $\bPerf_{\mathcal G_n(\cat{A})}$ and $\mathbf{T}^*[2-n]\scat{Perf}_{\cat{A}}$ are linear stacks over $\scat{Perf}_{\cat{A}}$: 
\begin{itemize}
\item On the one hand, $\bPerf_{\mathcal G_n(\cat{A})}\simeq\mathbb{A}_{\mathcal{M}}$, thanks to Proposition~\ref{prop:pasdenom}, 
where 
\[
x^*\mathcal M:=\underline{\cat{Map}}_{\scat{Mod}_{\cat{A}^{\op}}}
(\cat{A}^\vee[n-1]\overset{\mathbb{L}}{\underset{\cat{A}}{\otimes}} M_x, M_x)
\]
on a $B$-point $x$ of $\scat{Perf}_{\cat{A}}$. 
\item On the other hand, $\mathbf{T}^*[2-n]\scat{Perf}_{\cat{A}}=\mathbb{A}_{\mathbb{L}_{\bPerf_{\cA}}[2-n]}$, by definition. 
\end{itemize}
It thus remains to show that $\mathbb{L}_{\bPerf_{\cA}}[2-n]\simeq\mathcal{M}$. 
The tangent complex to $\scat{Perf}_{\cat{A}}$ at $x$ is the $B$-module 
\[
x^*\mathbb{T}_{\scat{Perf}_{\cat{A}}}\simeq\underline{\cat{Map}}_{\scat{Mod}_{\cat{A}^{\op}}}(M_x,M_x)[1]\,.
\]
Hence, the $(2-n)$-shifted cotangent complex at $x$ is the $B$-module 
\begin{eqnarray*}
x^*\mathbb{L}_{\scat{Perf}_{\cat{A}}}[2-n] & \simeq & \underline{\cat{Map}}_{\scat{Mod}_{\cat{A}^{\op}}}(M_x,M_x)[1]^\vee[2-n] \\
& \simeq & \underline{\cat{Map}}_{\scat{Mod}_{\cat{A}^{\op}}}(\cat{A}^\vee\overset{\mathbb{L}}{\underset{\cat{A}}{\otimes}}M_x,M_x)[1-n]
=:x^*\mathcal M\,.
\end{eqnarray*}
This proves the claim, as the above equivalences are natural in $x$. 
\end{proof}

\begin{corollary}\label{theorem: compareGinzburgrel}
Let $f:\cat{A} \to \cat{B}$ a functor between dg-category of finite type. Then $\bPerf$ applied to the 
cospan $\cG_n(\cA) \rightarrow \cG_n(f) \leftarrow \cG_n(\cB)$ is equivalent to 
\[
\bT^*[2-n] \bPerf_\cA \leftarrow \varphi^* \bT^*[2-n] \bPerf_\cA \rightarrow \bT^* [2-n] \bPerf_\cB\,,
\]
where $\varphi=f^*:\bPerf_\cB\to \bPerf_\cA$. 
\end{corollary}
\begin{proof}
We have equivalences
\[
		\varphi^* \bT^*[2-n]  \bPerf_\cA 
:=		\bT^*[2-n]\bPerf_\cA \underset{\bPerf_\cA}{\times} \bPerf_{\cB}
\simeq	\bPerf_{\cG_n(\cA) \underset{\cA}{\amalg} \cB}
\simeq	\bPerf_{\cG_n(f)}\,,
\]
where the first one follows from Corollary~\ref{theorem: compareGinzburg} (and the fact that $\scat{Perf}$ sends push-outs 
to pull-backs), and the second one follows from Lemma~\ref{lemma: G(f)}. 
\end{proof}

Going back to the general situation of Proposition~\ref{prop:pasdenom}, let us observe that every bimodule morphism
$c:M\to\cat{A}$ defines a global section $s_c$ of $\mathcal{M}$ whose graph 
\[
\bPerf_{\cA}\rightarrow \mathbb{A}_{\mathcal{M}}
\]
is identified with the image of $T_{\cA}(c)$ by $\scat{Perf}$. Indeed, we have a morphism 
\begin{equation}\label{eq:plusdidee}
\cat{Map}_{\scat{Mod}_{\cA^e}}(M,\cA) \overset{-\otimes\mathrm{id}}{\longrightarrow} 
\underset{x:\cat{A}\to\cat{Mod}_B^\mathrm{perf}}{\mathrm{holim}}\cat{Map}_{\scat{BiMod}_{\cat{A}-B}}
(\mathcal{M}\underset{\cat{A}}{\overset{\mathbb{L}}{\otimes}}M_x,M_x)
\simeq \Gamma(\mathcal{M})\,.
\end{equation}
As a consequence, $\scat{Perf}$ sends the deformation of $T_{\cA}(M[1])$ by $c$, which is obtained as a push-out along 
\[
\xymatrix{
T_{\cA}(M) \ar[r]^-{T_{\cA}(c)} \ar[d]_-{T_{\cA}(0)} & \cA \\
\cA & 
}
\]
to the derived zero locus of $s_c$, which is defined as the derived intersection of $s_c$ with the zero section. 
\begin{example}[Critical loci]\label{ex:criticalexample}
Observe that in the case when $M=\cA^\vee[n-1]$, \eqref{eq:plusdidee} retrieves the map $\mathrm{ch}^{(1)}$ that was made explicit 
in Example~\ref{ex:BDcompute}. 
In the case when the Hochschild class $c=(\delta W)^\natural$ is exact, then the section we get is given by $(d_{dR}\phi)^\sharp$, 
with $\phi:=\mathrm{ch}^{(0)}(W)$. We therefore obtain an equivalence between $\scat{Perf}_{\mathcal{G}_{n+1}(\cA,\delta W)}$ and 
the derived critical locus of $\phi$. 
Similarly, given a functor $f:\cA\to\cB$ and an exact class $c=(\delta W)^\natural$, with $W:k[n-2]\to \cat{HH}(\cat{B})$, 
there is an equivalence between $\scat{Perf}_{\mathcal{G}_{n+1}(\cB|\cA,\delta W)}$ and the relative derived critical locus 
$\mathbf{crit}_\pi(\phi)$, where $\pi=f^*$ and $\phi=\mathrm{ch}^{(0)}(W)$. 
\end{example}
\begin{example}[Quivers]\label{ex:criticalquiverexample}
We now specialize the previous example to the case of an inclusion of quivers $D\subseteq Q$, and $n=2$. 
In this situation $W\in kQ/[kQ,kQ]$ and the map $\mathrm{ch}^{(0)}:kQ/[kQ,kQ]\to |\Gamma(\mathcal O_{\scat{Perf}_{kQ}})|$ has a 
fairly explicit description: for every $B$-point $x$, 
\[
kQ/[kQ,kQ]\to 
\underline{\cat{End}}_{\cat{Mod}_B}(M_x)/[\underline{\cat{End}}_{\cat{Mod}_B}(M_x),\underline{\cat{End}}_{\cat{Mod}_B}(M_x)]
\overset{\mathrm{tr}}{\simeq}B
\]
In particular, we recover the identification of~\S\ref{pot3CY} \textit{via} the open embedding of derived stacks 
$\scat{DRep}\big[\mathcal{G}_3(kQ|kD,\delta W),\vec{n})/{\GL}_{\vec{n}}(k)\big] \hookrightarrow \bPerf_{\mathcal{G}_3(kQ|kD,\delta W)}$. 
\end{example}

\subsection{Calabi--Yau completions as noncommutative cotangents: symplectic aspects}

On the one hand we have shown that $\cG_n(\cA)$ has a canonical exact $n$-Calabi-Yau structure, hence according 
to~\S\ref{sss: CYLag} there is an exact $(2-n)$-shifted symplectic structure on $\bPerf_{\cG_n(\cA)}$. 
On the other hand, after the identification $\bT^*[2-n] \bPerf_{\cA} \simeq \bPerf_{\cG_n(\cA)}$ from 
Corollary~\ref{theorem: compareGinzburg} there is another exact $(2-n)$-shifted symplectic structure on $\bPerf_{\cG_n(\cA)}$. 
We now prove that they are the same: 
\begin{theorem}\label{theo: compare}
The exact $(2-n)$-shifted symplectic structure on $\bPerf_{\cG_n(\cA)}$ induced by the canonical exact $n$-Calabi-Yau structure on 
$\cG_n(\cA)$ agrees with the canonical exact $(2-n)$-shifted symplectic structure on $\bT^*[2-n] \bPerf_\cA$.
\end{theorem}
\begin{proof}
Recall that both $\bPerf_{\cG_n(\cA)}$ and $\bT^*[2-n] \bPerf_\cA$ are canonically identified with the linear stack 
$\mathbb{A}_{\mathcal M_\cA}$, where $\mathcal{M}_\cA\simeq\mathbb{L}_{\bPerf_\cA}[2-n]$ is defined by 
\[
x^*\mathcal{M}_\cA:=\underline{\cat{Map}}_{\scat{Mod}_{\cat{A}^{\op}}}
(\cA^\vee[n-1]\overset{\mathbb{L}}{\underset{\cat{A}}{\otimes}} M_x, M_x)\,
\]
for every $B$-point $x:\cat{A}\to\cat{Mod}_B^\mathrm{perf}$. The space of lifts $u:\scat{Spec}(B)\to\mathbb{A}_{{\mathcal M}_\cA}$ 
of such an $x$ is thus $|x^*\mathcal{M}_\cA|$. Abusing the notation, we allow ourselves to write $B$-points of 
$\mathbb{A}_{{\mathcal M}_\cA}$ as pairs $(x,s)$, with $x:\cat{A}\to\cat{Mod}_B^\mathrm{perf}$ and $s\in|x^*\mathcal{M}_\cA|$. 
Finally, also recall that such a pair $u=(x,s)$ defines a $\mathcal{G}_n(\cA)-B$-bimodule $M_u$ whose restriction to $\cA$ is $M_x$. 

Now observe that we have a commuting square 
\[
\xymatrix{
\ul{\cat{Map}}_{\scat{Mod}_{\cat{A}^e}}(\cA^\vee[n-1],\cA^\vee[n-1]) 
\ar[r]^-{s\circ(-\otimes\mathrm{id})}\ar[d]_-{\eqref{exact nCY}} & 
\underset{(x,s)}{\mathrm{holim}}\,
\ul{\cat{Map}}_{\scat{Mod}_{\cat{A}^{\op}}}(\cA^\vee[n-1]\overset{\mathbb{L}}{\underset{\cat{A}}{\otimes}} M_x, M_x) \ar[d] \\
\ul{\cat{Map}}_{\scat{Mod}_{\mathcal{G}_n(\cat{A})^e}}\big(\mathcal G_n(\cA)^\vee[n-1],\mathcal{G}_n(\cA)\big) 
\ar[r]^-{-\otimes\mathrm{id}} &
\underset{u}{\mathrm{holim}}\,
\ul{\cat{Map}}_{\scat{Mod}_{\mathcal{G}_n(\cat{A})^{\op}}}\big(\mathcal G_n(\cA)^\vee[n-1]\overset{\mathbb{L}}{\underset{\cG_n(\cat{A})}{\otimes}} M_u, M_u\big)
}
\]
in which the upper right corner is $\Gamma(\pi^*\mathcal{M}_\cA)\simeq\Gamma(\pi^*\mathbb{L}_{\scat{Perf}_\cA}[2-n])$, and the 
lower right corner is equivalent to $\Gamma(\mathbb{L}_{\scat{Perf}_{\mathcal{G}_n(\cA)}}[2-n])$. 

On the one hand, the tautological $1$-form of degree $2-n$ on the $(2-n)$-shifted cotangent stack is the image of 
$\mathrm{id}_{\cA^\vee[n-1]}$ through the upper part of the square. Indeed, the right vertical arrow is the pullback morphism 
$\Gamma(\pi^*\mathbb{L}_{\scat{Perf}_\cA}[2-n])\to \Gamma(\mathbb{L}_{\scat{Perf}_{\mathcal{G}_n(\cA)}}[2-n])$, and the image of 
$\mathrm{id}_{\cA^\vee[n-1]}$ through the top horizontal arrow is the tautological section (that is $s$ above a $B$-point $(x,s)$). 
On the other hand, if $c$ denotes the exact Calabi--Yau structure from Theorem~\ref{thm:CYcom}, that is the image of $\mathrm{id}_{\cA^\vee[n-1]}$ through the 
left vertical arrow, then $\mathrm{ch}^{(1)}(c)$ is the image of $\mathrm{id}_{\cA^\vee[n-1]}$ through the lower part of the square 
(thanks to \cite[Proposition 5.3]{BD2}, as recalled in Example \ref{ex:BDcompute} above). 
Hence the tautological $1$-form and $\mathrm{ch}^{(1)}(c)$ coincide. 
\end{proof}

The next result shows that the equivalence of shifted symplectic stack from Corollary~\ref{theorem: compareGinzburg} 
and Theorem \ref{theo: compare} is natural, and thus it completes the proof of Theorem~\ref{thm:mainmain}. 
\begin{theorem}\label{theo: comparerel}
Let $f: \cA \to \cB$ be a functor between finite dg categories. 
The moduli of objects functor sends the canonical exact $n$-Calabi-Yau cospan 
$\cG_n(\cA) \rightarrow \cG_n(f) \leftarrow \cG_n(\cB)$ to the canonical exact $(2-n)$-lagrangian correspondance 
\[
\bT^*[2-n]\bPerf_{\cA}  \leftarrow \varphi^*\bT^*[2-n]\bPerf_{\cA} \rightarrow  \bT^*[2-n] \bPerf_{\cB}\,,
\]
where $\varphi=f^*:\scat{Perf}_{\cB}\to\scat{Perf}_{\cA}$. 
\end{theorem}
\begin{proof}
Recall (see Theorem~\ref{thm:Gn-functoriality}) that we have a Hochschild class $c_f$ defined as 
\[
k[n-1]\overset
{f^\vee[n-1]}{\longrightarrow}
\ul{\cat{Map}}_{\cB^e}\left(\cB^\vee,\cA^\vee[n-1]\overset{\mathbb{L}}{\underset{\cA^e}{\otimes}}\cB^e\right)
\longrightarrow 
\ul{\cat{Map}}_{\mathcal{G}_n(f)^e}\big(\mathcal{G}_n(f)^\vee,\mathcal{G}_n(f)\big)
\simeq
\cat{HH}\big(\mathcal G_n(f)\big)\,,
\]
which coincides with with both pushouts of $c_{\cA}:=c_{\mathrm{id}_{\cA}}$ and $c_\cB:=c_{\mathrm{id}_{\cB}}$, thanks to the 
commuting diagram 
\[
\xymatrix{
&
\ul{\cat{Map}}_{\cA^e}(\cA^\vee,\cA^\vee[n-1]) \ar[r] \ar[d]^-{(-\otimes\mathrm{id})\circ f^\vee} & 
\ul{\cat{Map}}_{\mathcal{G}_n(\cA)^e}\big(\mathcal{G}_n(\cA)^\vee,\mathcal{G}_n(\cA)\big) \ar[d] \\
k[n-1] \ar@/^/[ru]^-{\mathrm{id}_{\cA^\vee}[n-1]}\ar[r]^-{f^\vee[n-1]}\ar@/_/[rd]_-{\mathrm{id}_{\cB^\vee}[n-1]} & 
\ul{\cat{Map}}_{\cB^e}\left(\cB^\vee,\cA^\vee[n-1]\overset{\mathbb{L}}{\underset{\cA^e}{\otimes}}\cB^e\right) \ar[r] &
\ul{\cat{Map}}_{\mathcal{G}_n(f)^e}\big(\mathcal{G}_n(f)^\vee,\mathcal{G}_n(f)\big) \\
&
\ul{\cat{Map}}_{\cB^e}(\cB^\vee,\cB^\vee[n-1]) \ar[r] \ar[u]_-{f^\vee\circ-} &
\ul{\cat{Map}}_{\mathcal{G}_n(\cB)^e}\big(\mathcal{G}_n(\cB)^\vee,\mathcal{G}_n(\cB)\big) \ar[u] 
}
\]
Recall also (see Corollary~\ref{theorem: compareGinzburgrel}) that we have an equivalence 
\[
\scat{Perf}_{\mathcal{G}_n(f)}\simeq \varphi^*\bT^*[2-n]\scat{Perf}_{\cA}
\]
of linear stacks over $\scat{Perf}_{\cB}$. Indeed, both $\scat{Perf}_{\mathcal{G}_n(f)}$ and $\varphi^*\bT^*[2-n]\scat{Perf}_{\cA}$ 
are canonically identified with $\mathbb{A}_{\mathcal{M}_f}$, 
where $\mathcal{M}_f\simeq\varphi^*\mathbb{L}_{\scat{Perf}_{\cA}}[2-n]$ is defined by 
\[
x^*\mathcal M_f:=\ul{\cat{Mod}}_{\scat{Mod}_{\cat{B}^{\op}}}
\big(\cA^\vee[n-1]\overset{\mathbb{L}}{\underset{\cA^e}{\otimes}}\cB^e\overset{\mathbb{L}}{\underset{\cB}{\otimes}}M_x,M_x\big)
\simeq\ul{\cat{Mod}}_{\scat{Mod}_{\cat{A}^{\op}}}(\cA^\vee[n-1]\overset{\mathbb{L}}{\underset{\cA}{\otimes}}f^*M_x,f^*M_x)
\]
for every $C$-point $x:\cB\to \cat{Mod}_C^\mathrm{perf}$. 

Now, as in the proof of Theorem~\ref{theo: compare}, we abuse the notation and denote a $C$-point 
$u:\mathcal G_n(f)\to\cat{Mod}_C^\mathrm{perf}$ of $\scat{Perf}_{\mathcal G_n(f)}$ as a pair $(x,s)$, where $x:\cB\to\cat{Mod}_C^\mathrm{perf}$ 
and $s\in|x^*{\mathcal M}_f|$. 
We then repeat several times an argument similar to the one appearing in the proof of Theorem~\ref{theo: compare}.

\medskip

(i) We first consider the commuting square, denoted $\mathcal Q_n(\cA)$, 
\[
\xymatrix{
\ul{\cat{Map}}_{\scat{Mod}_{\cat{A}^e}}(\cA^\vee[n-1],\cA^\vee[n-1]) 
\ar[r]^-{s\circ(-\otimes\mathrm{id})}\ar[d] & 
\underset{(x,s)}{\mathrm{holim}}\,
\ul{\cat{Map}}_{\scat{Mod}_{\cat{A}^{\op}}}(\cA^\vee[n-1]\overset{\mathbb{L}}{\underset{\cat{A}}{\otimes}}f^*M_x,f^*M_x) \ar[d] \\
\ul{\cat{Map}}_{\scat{Mod}_{\mathcal{G}_n(\cat{A})^e}}\big(\mathcal G_n(\cA)^\vee[n-1],\mathcal{G}_n(\cA)\big) 
\ar[r]^-{-\otimes\mathrm{id}} &
\underset{u}{\mathrm{holim}}\,
\ul{\cat{Map}}_{\scat{Mod}_{\mathcal{G}_n(\cat{A})^{\op}}}\big(\mathcal G_n(\cA)^\vee[n-1]\overset{\mathbb{L}}{\underset{\cG_n(\cat{A})}{\otimes}}a^*M_u,a^*M_u\big)
}
\]
in which the upper right corner is 
$\Gamma(\pi^*\varphi^*\mathcal{M}_{\cA})\simeq\Gamma(\pi^*\varphi^*\mathbb{L}_{\scat{Perf}_{\cA}}[2-n])$, 
the lower right corner is equivalent to $\Gamma(\alpha^*\mathbb{L}_{\scat{Perf}_{\mathcal G_n(\cA)}}[2-n])$, and where 
$a:\mathcal G_n(\cA)\to\mathcal G_n(f)$ and $\alpha=a^*$. 
The image of $\mathrm{id}_{\cA}$ through the upper part of the square is $\alpha^*\lambda_{\scat{Perf}_\cA}$, its image 
through the lower part of the square is $\alpha^*\mathrm{ch}^{(1)}(c_\cA)$, and the homotopy we obtain between these is the pull-back 
\textit{via} $\alpha$ of the homotopy from Theorem~\ref{theo: compare}. 

\medskip

(ii) We then consider the commuting square, denoted $\mathcal Q_n(\cB)$, 
\[
\xymatrix{
\ul{\cat{Map}}_{\scat{Mod}_{\cat{B}^e}}(\cB^\vee[n-1],\cB^\vee[n-1]) 
\ar[r]^-{s\circ(-\otimes\mathrm{id})}\ar[d] & 
\underset{(x,s)}{\mathrm{holim}}\,
\ul{\cat{Map}}_{\scat{Mod}_{\cat{B}^{\op}}}(\cB^\vee[n-1]\overset{\mathbb{L}}{\underset{\cat{B}}{\otimes}} M_x, M_x) \ar[d] \\
\ul{\cat{Map}}_{\scat{Mod}_{\mathcal{G}_n(\cat{B})^e}}\big(\mathcal G_n(\cB)^\vee[n-1],\mathcal{G}_n(\cB)\big) 
\ar[r]^-{-\otimes\mathrm{id}} &
\underset{u}{\mathrm{holim}}\,
\ul{\cat{Map}}_{\scat{Mod}_{\mathcal{G}_n(\cat{B})^{\op}}}\big(\mathcal G_n(\cB)^\vee[n-1]\overset{\mathbb{L}}{\underset{\cG_n(\cat{A})}{\otimes}}b^*M_u,b^*M_u\big)
}
\]
in which the upper right corner is $\Gamma(\pi^*\mathcal{M}_{\cB})\simeq\Gamma(\pi^*\mathbb{L}_{\scat{Perf}_{\cB}}[2-n])$, the 
lower right corner is equivalent to $\Gamma(\beta^*\mathbb{L}_{\scat{Perf}_{\mathcal G_n(B)}}[2-n])$, and where 
$b:\mathcal G_n(\cA)\to\mathcal G_n(f)$ and $\beta=b^*$. 
As before, starting with $\mathrm{id}_{\cB}$ in the upper left corner of the square leads to the pull-back by $\beta$ 
of the homotopy between $\lambda_{\scat{Perf}_\cB}$ and $\mathrm{ch}^{(1)}(c_\cB)$ from Theorem~\ref{theo: compare}. 

\medskip

(iii) We finally consider the commuting square, denoted $\mathcal Q_n(f)$, 
\[
\xymatrix{
\ul{\cat{Map}}_{\scat{Mod}_{\cat{B}^e}}(\cB^\vee[n-1],\cA^\vee[n-1]\overset{\mathbb{L}}{\underset{\cat{A^e}}{\otimes}}\cB^e) 
\ar[r]^-{s\circ(-\otimes\mathrm{id})}\ar[d] & 
\underset{(x,s)}{\mathrm{holim}}\,
\ul{\cat{Map}}_{\scat{Mod}_{\cat{B}^{\op}}}(\cB^\vee[n-1]\overset{\mathbb{L}}{\underset{\cat{B}}{\otimes}}M_x,M_x) \ar[d] \\
\ul{\cat{Map}}_{\scat{Mod}_{\mathcal{G}_n(f)^e}}\big(\mathcal G_n(f)^\vee[n-1],\mathcal{G}_n(f)\big) 
\ar[r]^-{-\otimes\mathrm{id}} &
\underset{u}{\mathrm{holim}}\,
\ul{\cat{Map}}_{\scat{Mod}_{\mathcal{G}_n(f)^{\op}}}\big(\mathcal G_n(f)^\vee[n-1]\overset{\mathbb{L}}{\underset{\cG_n(f)}{\otimes}} M_u, M_u\big)
}
\]
in which the upper right corner is $\Gamma(\pi^*\mathcal{M}_{\cB})\simeq\Gamma(\pi^*\mathbb{L}_{\scat{Perf}_{\cB}}[2-n])$, and 
the lower right corner is equivalent to $\Gamma(\mathbb{L}_{\scat{Perf}_{\mathcal G_n(f)}}[2-n])$. 
Starting with $f^\vee$ in the upper left corner gives us $\mathrm{ch}^{(1)}(c_f)$. 

\medskip

Finally observe that there is a commuting diagram 
\[
\xymatrix{
& k \ar[ld]_-{\mathrm{id}_{\cA^\vee}}\ar[d]_{f^\vee}\ar[rd]^-{\mathrm{id}_{\cB^\vee}} & \\
\mathcal Q_n(\cA)\ar[r] & \mathcal{Q}_n(f) & \mathcal Q_n(\cB)\ar[l]
}
\]
which indeed exhibits an identification between the homotopies $\alpha^*\big(\mathrm{ch}^{(1)}(c_{\cA})\sim\lambda_{\scat{Perf}_\cA}\big)$ and 
$\beta^*\big(\mathrm{ch}^{(1)}(c_{\cB})\sim\lambda_{\scat{Perf}_\cB}\big)$. 
\end{proof} 
Let $\cB$ be a finite type dg-category and $g: k[n-1] \to \cat{HH}(\cB)$. 
We have seen in the previous subsection that $\scat{Perf}$ sends the morphism $T_{\cB}(g):\mathcal{G}_n(\cB)\to\cB$ 
to the graph of $\mathrm{ch}^{(1)}(g)$. 
Recall from proposition~\ref{proposition: leftCY} that every negative cyclic lift $\tilde{g}:k[n-1] \to \cat{HC}^-(\cB)$ of $g$ 
induces a Calabi--Yau structure on $T_{\cB}(g)$. Indeed, the image of the canonical Hochschild class $c_\cB$ through 
$T_{\cB}(g)$ is the Hochschild class $g$, so that the image of $\delta c$ is $\delta g$, which is homotopic to $0$ for 
every choice of negative cyclic lift $\tilde{g}$. 

\medskip

It then follows from Example~\ref{ex:BDcompute} together with Theorem~\ref{theo: compare} that the pull-back of the $2$-form 
$\omega_{\scat{Perf}_{\cB}}=d_{dR}\lambda_{\scat{Perf}_{\cB}}$ along the graph of $\mathrm{ch}^{(1)}(g)$ is 
$d_{dR} \mathrm{ch}^{(1)}(g)=\tilde{\mathrm{ch}}^{(2)}(\delta g)$. Therefore the lagrangian structure induced by the Calabi--Yau structure 
that is determined by the negative cyclic lift $\tilde{g}$ is the corresponding de Rham closed lift $\tilde{\mathrm{ch}}^{(1)}(\tilde{g})$. 

\medskip

Hence, if $g=\delta W$, then $\scat{Perf}$ gives back the lagrangian structure on the graph of $d_{dR}\mathrm{ch}^{(0)}(W)$. 
As a consequence, using that $\scat{Perf}$ defines a functor $\scat{CY}_{[n]}^{f.t.}\to \scat{Lag}_{[2-n]}$, we 
get the following lagrangian extension of Example~\ref{ex:criticalexample}: 
\begin{corollary}
Let $f:\cA\to\cB$ a morphism between finite type dg-categories, $W:k[n-2]\to \cat{HH}(\cB)$, $\phi:=\mathrm{ch}^{(0)}(W)$, and 
$\pi:=f^*:\scat{Perf}_{\cB}\to\scat{Perf}_{\cA}$. Then there is an equivalence of lagrangian morphisms between 
\[
\scat{Perf}_{\mathcal{G}_{n+1}(\cB|\cA,\delta W)}\to \scat{Perf}_{\mathcal{G}_n(\cA)}\quad\textrm{and}\quad
\mathbf{crit}_\pi(\phi)\to \bT^*[2-n]\scat{Perf}_{\cA}\,.
\]
\end{corollary}
\begin{example}
If we specialize the above corollary to the quiver case as in Example~\ref{ex:criticalquiverexample} with $n=2$, then we get an 
equivalence of lagrangian morphisms between 
\[
\scat{Perf}_{\mathcal{G}_{3}(kQ|kD,\delta W)}\to \scat{Perf}_{\mathcal{G}_2(kD)}\quad\textrm{and}\quad
\mathbf{crit}_\pi(\phi)\to \bT^*\scat{Perf}_{kD}\,.
\]
\end{example}

\begin{example}
Let $Q$ be a quiver, and let $Q^+$ be its framing:  
\[
V(Q^+)=V(Q)\cup\{v^+|v\in V(Q)\}\quad\textrm{and}\quad E(Q^+)=E(Q)\cup\{(v,v^+)|v\in V(Q)\}\,.
\]
Consider the sub-quiver $D\subset Q^+$ defined by $V(D)=\{v^+|v\in V(Q)\}$ and $E(D)=\emptyset$. 
There is a Calabi--Yau morphism $\mathcal{G}_1(kD)\to \mathcal G_2(kQ^+|kD)$. According to the above corollary, 
there is an equivalence of lagrangian morphisms between 
\[
\scat{Perf}_{\mathcal{G}_{2}(kQ^+|kD)}\to \scat{Perf}_{\mathcal{G}_1(kD)}\quad\textrm{and}\quad
\bT^*_{\scat{Perf}_{kQ^+}}[1]\scat{Perf}_{kD}\to \bT^*[1]\scat{Perf}_{kD}\,.
\]
When restricted to open substacks of perfect modules of amplitude $0$ with fixed dimension vectors, we obtain 
lagrangian morphisms 
\[
\big[\mathbf{DRep}\big(\mathcal G_2(kQ^+|kD);\vec{n},\vec{m}\big)/{\GL}_{\vec{n},\vec{m}}\big]
\longrightarrow \big[\mathbf{DRep}\big(\mathcal{G}_2(kD);\vec{m}\big)/{\GL}_{\vec{m}}\big]
\simeq \big[\mathfrak{gl}_{\vec{m}}/{\GL}_{\vec{m}}\big]\,.
\]
Finally, taking the fiber product with the lagrangian morphism 
\[
\mathfrak{gl}_{\vec{m}}\longrightarrow \big[\mathfrak{gl}_{\vec{m}}/{\GL}_{\vec{m}}\big]\,,
\]
one gets a $0$-shifted symplectic structure on 
\[
\big[\mathbf{DRep}\big(\mathcal G_2(kQ^+|kD);\vec{n},\vec{m}\big)/{\GL}_{\vec{n}}\big]\,.
\]
These are derived stacky enhancements of Nakajima quiver varieties \cite{Nak96}. 
\end{example}

%% file: CritQuiv-Appendix.tex
The statement about the dimension hereunder is mentioned in~\cite[Remark 1]{PreTop} for $\mathfrak{sl}_n$ 
(hence the $-1$ difference). However we could not find any proof of the $\mathfrak{gl}_n$ case in the literature. 
So we set $\mathfrak g=\mathfrak{gl}_n(\mathbb C)$ and $\mathfrak g_x=\{y\in\mathfrak g\mid[y,x]=0\}$ for any $x\in\mathfrak g$.

\begin{lemma}\label{vsdim}
Consider $x$ nilpotent of type $\lambda$, a partition of $n$. Then $\{[y,z]\mid(y,z)\in\mathfrak g_x^2 \}$ is a vector space of codimension 
$\lambda_1$ in $\mathfrak g_x$. 
\end{lemma}

\begin{proof}
In an appropriate basis we have 
\[
x= 
  \left(\!\!\!\!\!\!
     \raisebox{0.5\depth}{
       \xymatrixcolsep{1ex}
       \xymatrixrowsep{1ex}
       \xymatrix{
         J_{\lambda_1} \ar @{.}[dddrrr] & 
         0 \ar @{.}[rr]\ar @{.}[ddrr]&
         &  0 \ar@{.}[dd]\\
         0 \ar@{.}[dd] \ar@{.}[ddrr]&\\
         &&& 0\\
         0 \ar@{.}[rr]  & & 0 & J_{\lambda_r}
       }
     }
   \right)  \]
where\[
J_p=\left(\!\!\!\!\!\!
     \raisebox{0.5\depth}{
       \xymatrixcolsep{1ex}
       \xymatrixrowsep{1ex}
       \xymatrix{
         0 \ar@{.}[ddd]\ar @{.}[dddrrr] & 
         1 \ar @{.}[ddrr]
         & 0 \ar @{.}[r] \ar@{.}[dr]&  0 \ar@{.}[d]\\
           &&&0\\
         &&& 1 \\
         0 \ar@{.}[rrr] &  &  & 0
       }
     }
   \right)\in \Mat_p.\]
   Elements $C\in\mathfrak g_x$ are given by blocks $(C_{i,j})_{1\le i,j\le r}$ of the following form\[
\Mat_{\lambda_i\times\lambda_j}\ni   C_{i,j}=\left\{\begin{aligned}
& \left(\begin{array}{cc}0_{\lambda_i\times\lambda_j-\lambda_i}&P_{i,j}(J_{\lambda_i})\end{array}\right)&\text{ if }i\ge j\\
 &  \left(\begin{array}{c}P_{i,j}(J_{\lambda_j}) \\0_{\lambda_i-\lambda_j\times\lambda_j}\end{array}\right)&\text{ if }i<j
 \end{aligned}\right.
   \]
   where $P_{i,j}\in\mathbb C[X]$. Consider another element $C'\in\mathfrak g_x$ defined by polynomials $P'_{i,j}$. 
   Fix $i\neq j$ and assume $P_{s,t}=\delta_{(s,t),(i,i)}$, $P'_{s,t}=0$ if $(s,t)\neq(i,j)$. Then \[
   [C,C']_{s,t}=\delta_{(s,t),(i,j)}\left\{\begin{aligned}
& \left(\begin{array}{cc}0_{\lambda_i\times\lambda_j-\lambda_i}&P'_{i,j}(J_{\lambda_i})\end{array}\right)&\text{ if }i\ge j\\
 &  \left(\begin{array}{c}P'_{i,j}(J_{\lambda_j}) \\0_{\lambda_i-\lambda_j\times\lambda_j}\end{array}\right)&\text{ if }i<j
 \end{aligned}\right.
   \]
   which clearly describes a vector space, the same as $\{C_{s,t}\mid C\in\mathfrak g_x\}$.
   Then we need to understand the diagonal blocks that can be reached by commutators $[C,C']$.
   We have \begin{align*}
   [C,C']_{i,i}&=\sum_{j<i} \left(\begin{array}{cc}0_{\lambda_i\times\lambda_j-\lambda_i}&P_{i,j}(J_{\lambda_i})\end{array}\right)\left(\begin{array}{c}P'_{j,i}(J_{\lambda_i}) \\0_{\lambda_j-\lambda_i\times\lambda_i}\end{array}\right)\\
   &\qquad\qquad-\left(\begin{array}{cc}0_{\lambda_i\times\lambda_j-\lambda_i}&P'_{i,j}(J_{\lambda_i})\end{array}\right)\left(\begin{array}{c}P_{j,i}(J_{\lambda_i}) \\0_{\lambda_j-\lambda_i\times\lambda_i}\end{array}\right)\\
   &+\sum_{j>i} \left(\begin{array}{c}P_{i,j}(J_{\lambda_j}) \\0_{\lambda_i-\lambda_j\times\lambda_j}\end{array}\right)\left(\begin{array}{cc}0_{\lambda_j\times\lambda_i-\lambda_j}&P'_{j,i}(J_{\lambda_j})\end{array}\right)\\
   &\qquad\qquad-\left(\begin{array}{c}P'_{i,j}(J_{\lambda_j}) \\0_{\lambda_i-\lambda_j\times\lambda_j}\end{array}\right)\left(\begin{array}{cc}0_{\lambda_j\times\lambda_i-\lambda_j}&P_{j,i}(J_{\lambda_j})\end{array}\right)\\
   &=\sum_{j<i}X^{\lambda_j-\lambda_i}(P_{i,j}P'_{j,i}-P'_{i,j}P_{j,i})(J_{\lambda_i})\\
   &+\sum_{j>i}X^{\lambda_i-\lambda_j}(P_{i,j}P'_{j,i}-P'_{i,j}P_{j,i})(J_{\lambda_i})
   \end{align*}
Taking specific values ($C,C'$ with only one nonzero block) we see that 
\[
\{[C,C']_{i,i}\mid C,C'\in\mathfrak g_x, 1\le i\le r\}
\]
is the vector space $Im(f)$ where $f$ is the linear map
\[
\oplus_{i<j}\mathbb C_{\lambda_j-1}[X]\rightarrow\prod_i \Mat_{\lambda_i}(\mathbb C)
\]
that sends $(F_{i,j})_{i<j}$ to 
\[
\left(  \sum_{i<j}X^{\lambda_i-\lambda_j}F_{i,j}-\sum_{j<i}X^{\lambda_j-\lambda_i}F_{j,i}\right)(J_{\lambda_i})\,.
\]
Consider $(F_{i,j})\in\ker (f)$. There exists a family $(F_i)$ of polynomials such that for every $i$ we have
\begin{align*}
  &\sum_{i<j}X^{\lambda_i-\lambda_j}F_{i,j}-\sum_{j<i}X^{\lambda_j-\lambda_i}F_{j,i}=X^{\lambda_i}F_i\\
  &\Leftrightarrow \sum_{i<j}X^{\lambda_i-\lambda_j}F_{i,j}=\sum_{j<i}X^{\lambda_j-\lambda_i}F_{j,i}+X^{\lambda_i}F_i\\
  &\Leftrightarrow \sum_{i<j}X^{\lambda_i-\lambda_j}F_{i,j}=\sum_{j<i}X^{\lambda_j-\lambda_i}F_{j,i}
\end{align*}
since the left-hand side has degree $<\lambda_i$. The system of equations we get is equivalent to
\[
X^{\lambda_k-\lambda_{k+1}}F_{k,k+1}=-\sum_{h<k<l}X^{\lambda_h-\lambda_l}F_{h,l}-\sum_{k+1<l}X^{\lambda_k-\lambda_l}F_{k,l}
\]
for every $k$.
Since on the right-hand side the power of $X$ is always $\ge\lambda_k-\lambda_{k+1}$, these equations determine $F_{k,k+1}$ and 
\[
\dim \ker(f)=\sum_{i+1<j}\lambda_j\Rightarrow \dim \mathrm{Im}(f)=\sum_{j\ge2}\lambda_j
\]
as expected.
\end{proof}